\documentclass{amsart}
\usepackage{amssymb}
\usepackage{hyperref}
\usepackage{graphicx} 
\usepackage{float}
\usepackage[matrix, arrow,all,cmtip,color]{xy}
\usepackage{comment}
\usepackage{tikz}

\newtheorem{theorem}{Theorem}[section]
\newtheorem*{theorem*}{Theorem}

\newtheorem{fact}[theorem]{Fact}
\newtheorem{lemma}[theorem]{Lemma}
\newtheorem{corollary}[theorem]{Corollary}
\newtheorem{claim}[theorem]{Claim}
\newtheorem{proposition}[theorem]{Proposition}

\newtheorem{alphthm}{Theorem}

\theoremstyle{definition}
\newtheorem{definition}[theorem]{Definition}
\newtheorem{example}[theorem]{Example}
\newtheorem{non-example}[theorem]{Non-Example}
\newtheorem{convention}[theorem]{Convention}
\newtheorem{punchline}[theorem]{Punchline}
\newtheorem{notation}[theorem]{Notation}

\newtheorem{assumption}[theorem]{Assumption}
\newtheorem{remark}[theorem]{Remark}

\newcommand{\acl}{\operatorname{acl}}

\newcommand{\tp}{\operatorname{tp}}

\makeatletter
\let\@wraptoccontribs\wraptoccontribs
\makeatother

\title[Topological Reconstruction Theorems]{Topological Reconstruction Theorems over Uncountable Algebraically Closed Fields}

\author{Benjamin Castle}
\address{Department of Mathematics, University of Illinois Urbana-Champaign}
\email{btcastl2@illinois.edu}

\author{Ronan O'Gorman}
\address{Department of Mathematics, University of California Berkeley}
\email{ronan\_ogorman@berkeley.edu}

\thanks{Castle was supported by NSF grant DMS-2452735.}

\begin{document}

\maketitle

\begin{abstract}
    Working over uncountable algebraically closed fields, we extend the theorems of Koll\'ar-Lieblich-Olsson-Sawin on reconstructing varieties from their Zariski topological spaces. In particular, we adapt their results to arbitrary quasi-projective varieties in arbitrary characteristic, and thus we give positive answers to each of the relevant `speculations' made by the original authors in our setting. Our proofs use techniques from model theory: in particular, we employ a general model-theoretic setting for algebro-geometric reconstruction problems, known as the `Zilber trichotomy for ACF-relics'. 
\end{abstract}

\tableofcontents

\section{Introduction}

\subsection{Motivation} In \cite{KLOS} (which we often refer to as `KLOS'), the authors prove the following remarkable result.\footnote{The authors also prove slightly weaker statements over countable fields of characteristic zero; we state the uncountable version here because it is clean and closest to the present paper.}

\begin{notation}
    Below, if $X$ is a scheme, then we denote its underlying topological space by $|X|$; and if $f:X\rightarrow Y$ is a morphism of schemes, then we similarly denote its underlying topological map by $|f|:|X|\rightarrow|Y|$. (We will later identify $X$ with $|X|$ when clear from context; see Convention \ref{Con: set of points}).
\end{notation}

\begin{fact}[\cite{KLOS}, uncountable case]\label{F: KLOS}
Let $K_1$ and $K_2$ be uncountable fields of characteristic $0$, and let $X_1$ and $X_2$ be geometrically integral, normal, projective varieties of dimension at least 2 over $K_1$ and $K_2$, respectively. Then any homeomorphism
\[
\varphi : |X_1| \to |X_2|
\]
    factors as a field isomorphism followed by an isomorphism of $K_2$-varieties.\footnote{The statement in \cite[Main Theorem 1.2.1]{KLOS} is a bit more precise, because it tracks the embedding into projective space across the isomorphism, but the paraphrasing above is the main point.} That is, there are a field isomorphism $\sigma:K_1\rightarrow K_2$, and an isomorphism $f:\sigma(X_1)\rightarrow X_2$ of $K_2$-varieties, such that on $|X_1|$ we have $\varphi=|f|\circ\sigma$.
\end{fact}

This says that in many cases, the full algebraic structure of a variety is (quite surprisingly) encoded entirely in its underlying Zariski topological space. Of course, one is naturally led to wonder if the same holds e.g. for \textit{affine} varieties, and in all characteristics. Indeed, the authors of \cite{KLOS} state several `speculations' about potential adaptations of Fact \ref{F: KLOS} to varieties which are (i) quasi-projective\footnote{We should note that over uncountable algebraically closed fields of characteristic zero, the authors \textit{did} prove the analog of Fact \ref{F: KLOS} for \textit{proper} irreducible normal varieties (in place of projective ones) in \cite{OlderKLOS}; but as far as we know, the general quasi-projective case (and in particular the \textit{affine} case) has remained open.}, (ii) potentially non-normal, (iii) potentially reducible, and (iv) defined over positive characteristic fields. In particular, they point out that their argument fails in various ways in each of these more general settings. \textbf{In the current paper, we use model theory to prove each of the given speculations in the special case that the $K_i$ are algebraically closed.} 

\subsection{Statement of the main theorem in the irreducible case}

Before stating the main result, we note (as do the authors of \cite{KLOS}) that the formulation in Fact \ref{F: KLOS} \textit{cannot} work in general, for two main reasons:

\begin{enumerate}
    \item There are many varieties with bijective normalization maps, and these give homeomorphisms which are not isomorphisms (and are not even reducible to isomorphisms modulo field isomorphisms). In fact, the \textit{inverse} of a bijective normalization gives a homeomorphism which can't even be reduced to a \textit{morphism} (let alone an isomorphism).
    \item In positive characteristic, the Frobenius map introduces many new homeomorphisms in the form of \textit{purely inseparable maps}. For example, one can check that the map $\varphi:\mathbb A^2\rightarrow\mathbb A^2$, $(x,y)\mapsto(x,y^p)$, is a homeomorphism over any perfect field of characteristic $p$ which does not admit a factorization of the form in Fact \ref{F: KLOS}. \footnote{As noted in \cite{KLOS}, a similar projective example can be arranged using the product of two elliptic curves. Interestingly, the $p$th root map $(x,y)\mapsto(\operatorname{
    Fr}^{-1}(x),\operatorname{Fr}^{-1}(y))$ on $\mathbb A^2$ \textit{does} factor in the form of Fact \ref{F: KLOS}, using the inverse Frobenius as the field isomorphism $\sigma$; the point of the $(x,y)\mapsto(x,y^p)$ example is to have two \textit{different} Frobenius powers in the same map, so that no field isomorphism can handle both simultaneously.}
\end{enumerate}

In light of (1) and (2) above, when considering arbitrary varieties, one must at minimum accept two changes to the statement of Fact \ref{F: KLOS}:

\begin{enumerate}
    \item One should only hope to see an algebraic map of $K_2$-varieties after passing to the \textit{normalizations} of the $X_i$;
    \item and one should only hope for this algebraic map to be a \textit{\textit{purely inseparable morphism}}, not an isomorphism.
\end{enumerate}

We have now motivated our main theorem in the irreducible case (Theorem \ref{T: main} below -- the reducible case will be discussed in the next subsection). The theorem says that for irreducible varieties over uncountable algebraically closed fields, Fact \ref{F: KLOS} holds in full generality up to normalizations and purely inseparable morphisms.

    Recall that a \textit{universal homeomorphism} of schemes is a morphism of schemes which induces a homeomorphism on the underlying topological spaces, even after arbitrary base change. A universal homeomorphism of $K$-varieties is a morphism of $K$-varieties which is a universal homeomorphism of schemes. Note that if $K$ is algebraically closed, this is equivalent to the simpler assertion that $f$ is a morphism over $K$ and $|f|$ is a homeomorphism (as the statement about base change follows). 
    

\begin{convention}
    Throughout the paper, we denote the normalization of a scheme $X$ by $\widetilde{X}$.
\end{convention}

\begin{theorem}[Main theorem, irreducible case]\label{T: main}
Let $K_1$ and $K_2$ be uncountable algebraically closed fields, and for $i=1,2$ let $X_i$ be an irreducible quasi-projective variety of dimension at least 2 over $K_i$. Let
\[
\varphi : |X_1| \to |X_2|
\]
be a homeomorphism. Then there are a field isomorphism $\sigma : K_1 \to K_2$, and a universal homeomorphism $f:\widetilde{\sigma(X_1)}\rightarrow\widetilde{X_2}$ of $K_2$-varieties, such that we have a commuting square
\[
\xymatrix{
\widetilde{|\sigma(X_1)|} \ar[r]^{|f|} \ar[d] & |\widetilde{X_2}| \ar[d] \\
|\sigma(X_1)|\ar[r]^{\psi} & |X_2|
}
\]
where:
\begin{enumerate}
\item $\varphi=\psi\circ\sigma$ on $|X_1|$.
\item The vertical maps are the (underlying topological maps of the) normalization maps.
\end{enumerate}
\end{theorem}

\begin{remark} In fact, the same statement holds for the larger class of \textit{strongly connected varieties} -- see Definition \ref{D: str con} and Theorem \ref{T: main reducible}.
\end{remark}

As in Fact \ref{F: KLOS}, Theorem \ref{T: main} provides a factorization of the original homeomorphism $\varphi$ into two different homeomorphisms, where the first is induced by a field isomorphism. The difference is that now, rather than being an isomorphism of the varieties themselves, the second map \textit{lifts to a universal homeomorphism of the normalizations}. Note that by Zariski's main theorem (\cite[\href{https://stacks.math.columbia.edu/tag/05K0}{Tag 05K0}]{stacks-project}), \cite[\href{https://stacks.math.columbia.edu/tag/01S2}{Tag 01S2}]{stacks-project} and \cite[\href{https://stacks.math.columbia.edu/tag/04DF}{Tag 04DF}]{stacks-project}, a morphism of varieties which induces a homeomorphism of underlying topological spaces is a universal homeomorphism if and only if it is purely inseparable (i.e. radicial). Thus, we are really saying that (i) our original homeomorphism $\varphi$ lifts to a homeomorphism of normalizations $\tilde\varphi:|\widetilde{X_1}|\rightarrow|\widetilde{X_2}|$, and (ii) after factoring out an appropriate field isomorphism, the lifted map of normalizations arises from a purely inseparable morphism of varieties. 

 In light of the discussion preceding Theorem \ref{T: main}, this theorem statement seems as strong as one could expect for arbitrary (irreducible) varieties. In particular, using Theorem \ref{T: main}, it is straightforward to deduce each of the relevant speculations in \cite{KLOS} (over uncountable algebraically closed fields); see the last section of this paper for more details. For example, if the $K_i$ have characteristic zero and the $X_i$ are normal, then $f$ is just an isomorphism $\sigma(X_1)\rightarrow X_2$, and we get exactly the original conclusion of Fact \ref{F: KLOS} (i.e. over algebraically closed fields, this shows that Fact \ref{F: KLOS} holds verbatim without assuming projectivity, which is precisely \cite[Speculation 2.2.14]{KLOS}).

 \subsection{Extensions to reducible varieties}

 Let us now describe how Theorem \ref{T: main} generalizes to reducible varieties. In \cite[Theorem 2.3.3]{KLOS}, the authors also give a purely scheme-theoretic formalization of their main theorem. Namely, let $X_1,X_2$ be integral, normal, projective varieties of dimension at least 2 over uncountable fields $K_i$ of characteristic zero. Then \textit{every homeomorphism $|X_1|\rightarrow|X_2|$ underlies a scheme isomorphism} (the idea being that both factors $|f|$ and $\sigma$ of the composition $\varphi=|f|\circ\sigma$ from Fact \ref{F: KLOS} underlie scheme isomorphisms).

 At this level of abstract schemes, we get a clean generalization to the reducible case, which in particular recovers \cite[Speculation 2.2.10]{KLOS} in our setting (again, see the last section of this paper for a full proof):

 \begin{theorem}[Main theorem, reducible case]\label{T: reducible schemes}
     Let $K_1,K_2$ be uncountable algebraically closed fields, and let $X_i$ be quasi-projective varieties over $K_i$, all of whose irreducible components have dimension at least 2. Then any homeomorphism $|X_1|\rightarrow |X_2|$ lifts to a universal homeomorphism of normalized schemes $\widetilde{X_1}\rightarrow\widetilde{X_2}$.
 \end{theorem}

Note that Theorem \ref{T: reducible schemes} does \textit{not} say anything about factorizations in the form of Theorem \ref{T: main}, and (as pointed out in \cite{KLOS}) that issue is a bit more subtle; see Example \ref{E: planes meeting at a point}.

Theorem \ref{T: reducible schemes} is a straightforward corollary of Theorem \ref{T: main}. We note that a homeomorphism of varieties induces homeomorphisms on irreducible components, then apply Theorem \ref{T: main} to get a universal homeomorphism of normalizations on each component. But normalizations separate components, so our component-wise universal homeomorphisms can then be glued into a global universal homeomorphism. Note that this last step crucially needs the scheme-theoretic formulation: while we do \textit{separately} get factorizations $\varphi_j=\psi_j\circ\sigma_j$ on each normalized component of $X_1$, there is no reason that the field isomorphisms $\sigma_j$ should be the same (thus we cannot automatically glue componentwise into a morphism of \textit{varieties}). Indeed, essentially for this reason, the naive translation of Theorem \ref{T: main} to the reducible case \textit{cannot} work:

\begin{example}\label{E: planes meeting at a point}
   Let $X$ be a quasi-projective variety over the uncountable algebraically closed field $K$, and assume $X$ is the union of two proper closed sets $X=C\cup D$ with finite intersection (e.g. consider two planes meeting in a point, see \cite[Remark 2.2.11]{KLOS}). Let $F\leq K$ be a finitely generated field capable of defining $X$, $C$, $D$, and all points of $C\cap D$. Now define a homeomorphism $\varphi:X\rightarrow X$ by acting as the identity on $C$ and acting as some $\tau\in\operatorname{Gal}(K/F)$ on $D$. Then one can show that for suitably general $\tau$, the resulting homeomorphism $\varphi$ cannot be factored into a \textit{single} field automorphism composed with a morphism of varieties (see Lemma \ref{L: counterexample reducible}).
\end{example}

In \cite[Remark 2.2.11]{KLOS}, the authors mention that 0-dimensional intersections (as in Example \ref{E: planes meeting at a point}) might be the only obstruction to obtaining a factorization as in Fact \ref{F: KLOS}. This is the content of our final main theorem:

\begin{definition}\label{D: str con}
    Let $X$ be a variety over an algebraically closed field $K$. Say that $X$ is \textit{strongly connected} if one cannot write $|X|$ as the union of two proper closed subsets with finite intersection.
\end{definition}

\begin{theorem}[Main theorem, strongly connected version]\label{T: main reducible}
    Let $K_1,K_2$ be uncountable algebraically closed fields, and for $i = 1,2$ let $X_i$ be a strongly connected quasi-projective variety over $K_i$, all of whose irreducible components have dimension at least 2. Then any homeomorphism $\varphi:|X_1|\rightarrow |X_2|$ satisfies the conclusion of Theorem \ref{T: main}. That is, there are a field isomorphism $\sigma : K_1 \to K_2$ and a commuting square
\[
\xymatrix{
\widetilde{|\sigma(X_1)|} \ar[r]^{|f|} \ar[d] & |\widetilde{X_2}| \ar[d] \\
|\sigma(X_1)| \ar[r]^{\psi} & |X_2|
}
\]
where $\varphi=\psi\circ\sigma$ on $|X_1|$ and $f:\widetilde{\sigma(X_1)}\rightarrow\widetilde{X_2}$ is a universal homeomorphism of $K_2$-varieties.
\end{theorem}

As before, in light of Example \ref{E: planes meeting at a point}, the statement of Theorem \ref{T: main reducible} seems as strong as one could hope for.

\subsection{Linear equivalence and the KLOS strategy}
We have now finished describing our main theorems. We now discuss the basic proof strategy in the irreducible case (which contains most of the work). To begin, let us review the general idea of the proof in \cite{KLOS}, with the goal of motivating the appearance of model theory. The `slogan' is that the proof has two main parts:
\begin{itemize}
    \item \textbf{Part 1:} Recover linear equivalence.
    \item \textbf{Part 2:} Apply a variant of the fundamental theorem of projective geometry.
\end{itemize}

Below we elaborate on these two parts and how we will adapt them. First, let us state the following once and for all:

\begin{convention}\label{Con: set of points}
    \textbf{Throughout the rest of this paper, we adopt the following conventions: 
    \begin{enumerate}
        \item We identify a variety $X$ over an algebraically closed field $K$ with the underlying Zariski topological space on the set $X(K)$ of closed points. In particular, the notation `$x\in X$' should be read as `$x$ is a closed point of $X$'. If $Y$ is another variety over $K$, we identify a morphism of varieties $X\rightarrow Y$ with the underlying map on closed points. 
        \item If $Y$ is another variety, a \textit{homeomorphism} $X \to Y$ will be any homeomorphism of (the closed points of) the underlying topological spaces.
        \item A \textit{universal homeomorphism} will continue to refer to a morphism of varieties which induces a homeomorphism on the underlying topological spaces after arbitrary base change.
    \end{enumerate} 
    }
    
\end{convention}

\begin{remark}
    Convention \ref{Con: set of points} is harmless here, because a homeomorphism of (underlying spaces of) schemes in particular induces a homeomorphism of closed points, and over algebraically closed fields the converse holds also. Working only with closed points will allow us to view varieties as model-theoretic structures in a natural way.
\end{remark}

Now let us assume we have a homeomorphism $\varphi:X_1\rightarrow X_2$ as above, where the $X_i$ are irreducible quasi-projective varieties of dimension at least 2 over uncountable algebraically closed fields $K_i$. We discuss the basic proof strategy of \cite{KLOS} and its limitations.

To start, suppose $K_1=K_2=\mathbb C$ and $X_1=X_2=\mathbb P^2$. So $\varphi$ is a self homeomorphism of the complex projective plane. In this case, a reasonable idea is to show (by tracking the number of intersection points of curves) that $\varphi$ sends lines to lines, and then apply the \textit{classical} fundamental theorem of projective geometry (which gives exactly the conclusion of Fact \ref{F: KLOS} for permutations of the projective plane sending lines to lines). Indeed, it is possible to make this sketch precise, and one obtains a fairly quick proof of Fact \ref{F: KLOS} in this case.

Now assume only that $K_1,K_2,X_1,X_2,\varphi$ satisfy the hypotheses of Fact \ref{F: KLOS}. A naive attempt at adapting the above strategy would be to show that $\varphi$ sends \textit{hyperplane sections} to \textit{hyperplane sections}, and then apply an appropriate adaptation of the fundamental theorem of projective geometry for hyperplane sections. Unfortunately this cannot work, because the family of hyperplane sections is not an invariant of a variety (it depends on the embedding into projective space). Instead, the correct move (and that taken in \cite{KLOS}) is to consider arbitrary \textit{linear systems of divisors} (of which the hyperplane sections form a canonical example). Indeed, the main intermediate step in the proof of Fact \ref{F: KLOS} shows that $\varphi$ preserves the linear equivalence relation on divisors (\cite[Theorem 7.6.1]{KLOS}); the authors then develop a `generic' (in their words, `definable') version of the fundamental theorem of projective geometry (\cite[Theorem 3.1.5]{KLOS}), aimed at reconstructing a variety from its linear systems.

Now if we generalize to arbitrary varieties in arbitrary characteristic, this strategy begins to fail; the slogan here is that linear equivalence is `too precise' and `too rigid', and only behaves as expected under strong assumptions. For instance, linear equivalence is often trivial in affine varieties (e.g. any two divisors in $\mathbb A^n$ are linearly equivalent), and therefore recovering linear equivalence cannot possibly be enough to recover all algebraic structure; moreover, in positive characteristic, one can build purely inseparable maps that do not even respect linear equivalence in the first place; and finally, even assuming characteristic zero and projectivity, the recovery of linear equivalence in \cite{KLOS} uses very precise geometric information that does not seem applicable to non-normal varieties.

\subsection{The role of model theory}

So what should one do? The first main challenge in adapting Fact \ref{F: KLOS} is to identify an appropriate coarse analog of linear equivalence (i.e. to serve as an analog of Part 1 above). Precisely, let $\mathcal H$ be the collection of irreducible components of hyperplane sections in $X_1$, and let $\varphi(\mathcal H)$ be the family of images in $X_2$. Then one must identify the `right' property of $\varphi(\mathcal H)$ which is (1) weak enough to be true and provable in full generality, and (2) strong enough to imply a version of the fundamental theorem of projective geometry.

\begin{punchline} \textbf{In this paper, we argue that the `right' notion to study is model-theoretic \textit{definability} in the language of fields.}
\end{punchline}

Indeed, our main intermediate step (Theorem \ref{T: definability of hyperplane images}), in analogy with that of \cite{KLOS}, is that $\varphi(\mathcal H)$ is a \textit{definable} family of varieties in the sense of the field structure on $K_2$. Importantly, we \textit{do not} claim $\varphi(\mathcal H)$ is a linear system, or even that it forms a locally closed family; the proof \textit{only} gives abstract definability. On the other hand, deep results in model theory imply that definability is sufficient: indeed, we then proceed to apply a very coarse and flexible model-theoretic analog of the fundamental theorem of projective geometry -- namely the \textit{Zilber trichotomy} (\cite{CasHasYe}, discussed more below) -- which runs purely on abstract definability data. The upshot is that, using model theory, we can adapt the argument from \cite{KLOS} to run only on very coarse geometric information on $X_1$ and $X_2$ -- and this turns out to be well suited for handling homeomorphisms between less rigid classes of varieties.

To summarize, we have the following diagram comparing the proof strategies of \cite{KLOS} and the current paper:

\begin{table}[h]
\centering
\renewcommand{\arraystretch}{1.3}
\begin{tabular}{|p{0.45\textwidth}|p{0.45\textwidth}|}
\hline
\textbf{KLOS} & \textbf{This Paper} \\
\hline
\textbf{Part 1:} Recover linear equivalence of divisors
&
\textbf{Part 1:} Show that $\varphi(\mathcal H)$ is definable
\\
\hline
\textbf{Part 2:} Apply a variant fundamental theorem of projective geometry
&
\textbf{Part 2:} Apply Zilber's trichotomy
\\
\hline
\end{tabular}
\end{table}

\subsection{Full relics and isomorphisms}

Let us now give more background on the model-theoretic content of the paper, expanding on Parts 1 and 2 above. Our proof of Theorem \ref{T: main} is expresed most naturally in the language of \textit{relics} of algebraically closed fields (or \textit{ACF-relics}):

\begin{definition}[See \cite{CasHasVA}, Section 4]\label{D: relic}
    Let $K$ be an algebraically closed field.
    \begin{enumerate}
        \item A \textit{$K$-relic} is a model-theoretic structure $\mathcal M=(M,...)$ where
        \begin{enumerate}
            \item The universe $M$ of $\mathcal M$ is a constructible\footnote{Recall that a \textit{constructible} set is a finite union of locally closed sets, or equivalently a finite Boolean combination of closed sets. By \textit{quantifier elimination} in algebraically closed fields, the constructible subsets of a given projective space are exactly those subsets which are definable in the language of fields.} subset of some $\mathbb P^n(K)$.
            \item Every basic relation of $\mathcal M$ is a constructible subset of some $M^m$.
        \end{enumerate}
        \item A $K$-relic $\mathcal M=(M,...)$ is \textit{full} if \textit{every} constructible subset of every power of $M$ is definable in $\mathcal M$.
    \end{enumerate}
\end{definition}

One should think of a \textit{relic} as a variety equipped with partial data, and a \textit{full relic} as a relic admitting a \textit{reconstruction theorem} (i.e. where the given partial data determines all algebraic structure on the variety). This is expressed concretely in the following \textit{isomorphism theorem for full relics} (\cite[Lemma 2.6]{CaHaDC}) (in particular, note the resemblance of \ref{F: iso thm} to both Fact \ref{F: KLOS} and Theorem \ref{T: main}):\footnote{In full generality, this statement first appeared in \cite{CaHaDC}. However, the argument goes back at least to \cite{PoizatBT} and has been used in other papers since, e.g. \cite[Lemma 2.3]{CasHasAV}.}

\begin{fact}[Isomorphism theorem for full relics]\label{F: iso thm} Let $K_1,K_2$ be algebraically closed fields. Let $\mathcal M_1$ be a full $K_1$-relic, and let $\mathcal M_2$ be a full $K_2$-relic\footnote{In fact, in Fact \ref{F: iso thm} we do not need to assume $\mathcal M_2$ is full, because by \cite[Remark 4.13]{CasHasVA}, any ACF-relic isomorphism to a full ACF-relic is itself full.} in the same language as $\mathcal M_1$. Then any isomorphism of structures $\varphi:\mathcal M_1\rightarrow\mathcal M_2$ decomposes into a field isomorphism composed with a definable isomorphism. That is, there are a field isomorphism $\sigma:K_1\rightarrow K_2$, and a $K_2$-definable isomorphism of structures $f:\sigma(\mathcal M_1)\rightarrow\mathcal M_2$, so that on $\mathcal M_1$ we have $\varphi=f\circ\sigma$.
\end{fact}

Now suppose $\varphi:X_1\rightarrow X_2$ is our homeomorphism (so $X_i$ is an irreducible, quasi-projective variety of dimension at least 2 over the uncountable algebraically closed field $K_i$). We want to prove Theorem \ref{T: main} by applying Fact \ref{F: iso thm} to $\varphi$. However, this does \textit{not} immediately work. The problem is that the Zariski topology on a variety $X$ is \textit{not} a definable object, so does \textit{not} obviously determine a $K$-relic structure on $X$ (in model-theoretic language, the collection of Zariski open sets is not uniformly definable). Indeed, arguably the main challenge of the whole paper is to identify a map of $K_i$-relics to which Fact \ref{F: iso thm} can be applied. 

Our solution to the above problem is precisely where the `two-part' strategy outlined above comes into play. Indeed, let us assume $X_1$ has already been embedded as a locally closed subvariety of some $\mathbb P^n$, which is moreover non-degenerate (i.e. not contained in any hyperplane). Then we will prove Theorem \ref{T: main} by applying Fact \ref{F: iso thm} to the following structures:

\begin{definition}
    Let $\mathcal X_1^{hyp}$ be the structure with universe $X_1$ and with the following basic relations:
    \begin{itemize}
        \item The $(n+1)$-ary relation $CH_1\subset X_1^{n+1}$, where $(a_1,...,a_{n+1})\in CH_1$ if and only if $a_1,...,a_{n+1}$ lie on a common irreducible component of some hyperplane section in $X_1$.
        \item A unary predicate for each $K_0$-definable subset of $X_1$, where $K_0\leq K$ is a fixed countable algebraically closed field over which $X_1$ is defined.
    \end{itemize}
    Then we let $\mathcal X_2^{hyp}=(X_2,...)$ be the structure on $X_2$ obtained by applying $\varphi$ to all basic relations of $\mathcal X_1^{hyp}$.
\end{definition}

Then $\mathcal X_1^{hyp}$ is a $K_1$-relic, which encodes hyperplane incidence (via $CH_1$), in addition to some topological data (via the unary predicates).

Trivially, $\varphi$ defines an isomorphism $\mathcal X_1^{hyp}\rightarrow\mathcal X_2^{hyp}$; so we hope to apply Fact \ref{F: iso thm}. There are two obstacles to making this work, and they correspond precisely to the two-part slogan described above:

\begin{enumerate}
    \item Fact \ref{F: iso thm} requires that $\mathcal X_2^{hyp}$ is also a $K_2$-relic. This precisely amounts to the statement that the family of hyperplane images $\varphi(\mathcal H)$ considered above is definable in $K_2$, and thus is solved by Part 1 above.
    \item Fact \ref{F: iso thm} also requires that the $\mathcal X_i^{hyp}$ are \textit{full} $K_i$-relics. Fortunately, there are powerful model-theoretic tools for verifying fullness of ACF-relics -- the most important being \textit{Zilber's trichotomy} mentioned above -- and indeed, fullness follows here from Zilber's trichotomy combined with a short model-theoretic analysis of $\mathcal X_1^{hyp}$. So this is solved by Part 2 above.
\end{enumerate}

To summarize, we can now rewrite our `proof strategy' table in the language of relics. This will serve as an organizational guideline for the rest of the paper.

\begin{table}[h]
\centering
\renewcommand{\arraystretch}{1.3}
\begin{tabular}{|p{0.45\textwidth}|p{0.45\textwidth}|}
\hline
\textbf{KLOS} & \textbf{This Paper} \\
\hline
\textbf{Part 1:} Recover linear equivalence of divisors
&
\textbf{Part 1:} Show that $\mathcal X_2^{hyp}$ is a $K_2$-relic.
\\
\hline
\textbf{Part 2:} Apply a variant fundamental theorem of projective geometry
&
\textbf{Part 2:} Show that the $\mathcal X_i^{hyp}$ are full and apply the isomorphism theorem for full relics.
\\
\hline
\end{tabular}
\end{table}

After Parts 1 and 2 are finished, we have achieved a massive reduction in the problem at hand: namely, in proving Theorem \ref{T: main}, we may assume $\varphi$ is a \textit{definable homeomorphism} between varieties over the \textit{same field}. In particular, the graph of $\varphi$ is now a constructible subset of $X_1\times X_2$, and it follows easily that $\varphi$ is given generically by a Frobenius power composed with a rational map. Theorem \ref{T: main} now follows after a (fairly short) analysis of generically rational homeomorphisms (see Theorem \ref{T: normalization square}); the idea is to show that if $\varphi$ agrees \textit{generically} with a rational map $r$, it must agree \textit{everywhere} with $r$; thus $\varphi=r$ is just a bijective morphism. This is conceptually similar to \cite[Section 5.1]{KLOS}. Note that if we were in the setting of \cite{KLOS} (normal varieties in characteristic zero), then our bijective morphism would be an isomorphism, and we would be done. In our more general case, a similar strategy works only for the composite map $\widetilde{X_1}\rightarrow X_1\rightarrow X_2$. Ultimately, we get the analogous conclusion that $\widetilde{X_1}\rightarrow X_2$ is a morphism, which must factor through $\widetilde{X_2}$ by the universal property of normalizations. This results in the commuting square from Theorem \ref{T: main}. Then the fact that $\widetilde{X_1}\rightarrow\widetilde{X_2}$ is bijective (i.e. a universal homeomorphism) follows from applying the same reasoning to $\varphi^{-1}$.

\subsection{Summary of Part 1}

Since Part 1 is the longest and most technical part of the paper, we give a brief motivational guide to the argument. Recall that we want to show $\varphi(\mathcal H)$ is a $K_2$-definable family, where $\varphi:X_1\rightarrow X_2$ is our homeomorphism and $\mathcal H$ is the collection of irreducible components of hyperplane sections in $X_1$. This will be obtained as a corollary to four separate theorems, stated below (in a simplified form) as Theorems \ref{T: intro link existence}, \ref{T: intro link preservation}, \ref{T: intro hyperplane sweeping}, and \ref{T: intro definably shaped}. We hope the reader will find the current outline acceptable motivation for why definability of $\varphi(\mathcal H)$ is the right invariant to study.

Our approach is inspired by the material in \cite[Section 7.3]{KLOS} on \textit{topological pencils}. These are essentially combinatorial analogs of one-dimensional linear systems, formed by partitioning a variety into codimension 1 subvarieties which are potentially allowed to overlap along a fixed base locus (a motivating example is the collection of hyperplanes containing a given codimension 2 linear space). Topological pencils play a crucial role in \cite{KLOS} -- indeed, a key step in the proof is to show that topological pencils are `algebraic' \cite[Corollary 7.3.5]{KLOS}. The authors then proceed to reconstruct various other notions by describing them using pencils. Unfortunately, this process involves very precise geometry, and (just as with linear systems) the arguments seem too rigid to work in our more general setting. In full generality, topological pencils really only tell us something about one-dimensional families, which are not enough to see the family $\mathcal H$ of hyperplanes under consideration.

Instead, we will take `topological pencils are algebraic' as the \textit{base case for an induction}, whose higher-dimensional output is captured by Theorems \ref{T: intro link existence} and \ref{T: intro link preservation} below. In very coarse language, we prove something like `topological $k$-pencils are definable for all $k$'. To be precise, this description is intended only as a guiding heuristic: we do not actually define topological $k$-pencils, and Theorems \ref{T: intro link existence} and \ref{T: intro link preservation} are really more intricate statements in a different language. However, the idea of `extending topological pencils inductively' was and still is the main motivation for everything that follows. In particular, the resulting statements are more complicated than their analogs in \cite{KLOS}, but they are also more powerful. The idea is that the family $\varphi(\mathcal H)$ under consideration is essentially a topological $n$-pencil -- so if our induction is set up correctly, it essentially finishes Part 1. 

To make the above sketch precise, we introduce the language of \textit{sweeping orbits} (more precisely, we will recognize $\varphi(\mathcal H)$ definably using so-called \textit{sweeping data} of a certain orbit). This language is inspired by the model-theoretic notion of sweeping developed in \cite[Section 5]{CasHasVA}. Everything is explicitly defined in Section 2, but we give a summary here. Suppose $K$ is any uncountable algebraically closed field, and $X$ is an irreducible quasi-projective $K$-variety of dimension at least 2. Then:

\begin{enumerate}
    \item A \textit{$k$-sweeping orbit in $X$} is a $\operatorname{Gal}(K/F)$-orbit $O$ of closed irreducible codimension 1 subvarieties of $X$, for some countable subfield $F\leq K$, with the property that any $k$ independent generic points in $X$ lie on a member of $O$.
    \item If the member of $O$ containing a generic $k$-tuple is unique, then $O$ is \textit{uniquely $k$-sweeping}.
    \item Given a uniquely $k$-sweeping orbit $O$, we say that $O$ $k$-\textit{sweeps} a closed irreducible subvariety $Y\subset X$ if any $k$ independent generic points of $Y$ lie on a generic member of $O$ (where `generic' is over the field of definition of $Y$).
\end{enumerate}

\begin{example}\label{E: intro ex sweeping} The canonical example of a uniquely $k$-sweeping orbit is the orbit of generic hyperplane sections in $X$ (which is uniquely $k$-sweeping as long as $X$ is non-degenerately embedded into $\mathbb P^k$). More generally, a $k$-dimensional linear system of divisors on $X$ determines a uniquely $k$-sweeping orbit (but not the other way around). So uniquely $k$-sweeping orbits are intended as a combinatorial generalization of linear systems.
\end{example}

\begin{example}\label{E: intro ex rel sweeping} The `relativization' of sweeping to a subvariety $Y\subset X$ (i.e. (3) above) is intended to capture whether a $k$-dimensional linear system remains $k$-dimensional inside $Y$. For example, if $X=\mathbb P^n$, then the generic hyperplanes form an $n$-dimensional linear system on $X$, and therefore a uniquely $n$-sweeping orbit $O^{hyp}$ on $X$. But if we relativize everything to a hyperplane $Y$ (replacing each $H$ with $H\cap Y$), then the dimension of the system collapses to $n-1$. This says that $O^{hyp}$ does not $n$-sweep $Y$.
\end{example}

Now return to our homeomorphism $\varphi:X_1\rightarrow X_2$, and let $O_1$ be the uniquely $n$-sweeping orbit of generic hyperplane sections in $X_1$. We now state Theorems \ref{T: intro link existence}, \ref{T: intro link preservation}, and \ref{T: intro hyperplane sweeping} in this language: 

 \begin{alphthm}[= Theorem \ref{T: existence of linked orbits}]\label{T: intro link existence} There is a uniquely $n$-sweeping orbit $O_2$ in $X_2$ so that $\varphi(O_1)$ has `non-negligible intersection' with $O_2$ (the precise meaning of non-negligible is given in Definition \ref{D: linked}). 
 \end{alphthm}

\begin{alphthm}[= Theorem \ref{T: preservation}]\label{T: intro link preservation} Fix $O_2$ as above. Then for all but finitely many closed irreducible codimension 1 subvarieties $Y\subset X_1$, we have that $O_1$ $n$-sweeps $Y_1$ if and only if $O_2$ $n$-sweeps $\varphi(Y_1)$.
\end{alphthm}

\begin{alphthm}[= Theorem \ref{L: hyperplane sweeping}]\label{T: intro hyperplane sweeping} $\mathcal H$ is precisely the collection of closed irreducible codimension 1 subvarieties of $X_1$ which are not $n$-swept by $O_1$.
\end{alphthm}

Theorem \ref{T: intro link existence} says that we can recover a non-negligible portion of hyperplane sections; Theorem \ref{T: intro link preservation} says that this non-negligible portion is enough recover the `sweeping data' of the generic hyperplane sections; and Theorem \ref{T: intro hyperplane sweeping} says that this sweeping data is enough to recover \textit{all} hyperplane sections. In particular, we immediately obtain:

\begin{corollary}
    With at most finitely many exceptions, $\varphi(\mathcal H)$ is precisely the collection of closed irreducible codimension 1 subvarieties of $X_2$ which are not $n$-swept by $O_2$. 
\end{corollary}

For example, it already follows that $\varphi(\mathcal H)$ is $\operatorname{Gal}(K_2/F)$-invariant for some countable $F\leq K_2$. With a bit more work (see Section 7), we even get:

\begin{corollary}[= Proposition \ref{P: ctbly many families}]\label{C: intro ind definable}
    $\varphi(\mathcal H)$ is the union of countably many definable families in $K_2$.
\end{corollary}

To conclude, we need a way to reduce a countable union to a finite union. This motivates the fourth main theorem of Part 1:

\begin{alphthm}[= Corollary \ref{C: ind-definable implies definable}]\label{T: intro definably shaped} Suppose $Y\subset X_1^k$ and $\varphi(Y)\subset X_2^k$ are both countable unions of definable sets. Then $Y$ is definable if and only if $\varphi(Y)$ is.
\end{alphthm}

Theorem \ref{T: intro definably shaped} says that out of all countable unions of definable sets (called \textit{ind-definable sets}), we can topologically single out the definable ones. One should view this as a sort of first-order compactness theorem pushed through the map $\varphi$. Indeed, the idea is that an ind-definable set is definable if and only if it satisfies every conceivable finiteness condition. We make this precise in Section 8 via the notion of a \textit{definably shaped} set. Definably shaped sets are roughly those which appear `finitary' in a large number of specified ways. Then we show that (1) homeomorphisms preserve the class of definably shaped sets, and (2) an ind-definable set is definable if and only if it is definably shaped.

Now given Theorem \ref{T: intro definably shaped}, we finally conclude that $\varphi(\mathcal H)$ is definable. Indeed, let $CH_1$ be the set of $(n+1)$-tuples in $X_1$ belonging to a common member of $\mathcal H$, and let $CH_2=\varphi(CH_1)$. Then $\mathcal H$ is precisely the collection of closed irreducible codimension 1 subvarieties of $X_1$ occuring as fibers of $CH_1$ over $X_1^n$, and $\varphi(\mathcal H)$ is the analogous collection of fibers of $CH_2$ over $X_2^n$. It is easy to see that $CH_1$ is $K_1$-definable, and Corollary \ref{C: intro ind definable} gives that $CH_2$ is a countable union of definable sets. So by Theorem \ref{T: intro definably shaped}, $CH_2$ is $K_2$-definable, and thus so is $\varphi(\mathcal H)$. 

\subsection{Organization of the Paper}

We conclude the introduction by describing the various sections of the paper:\\

\textbf{Part 1:}

\begin{itemize}
    \item Section 2 gives preliminaries on some model-theoretic notions (such as dimension and canonical bases) that we use throughout.
    \item Section 3 introduces the language of sweeping orbits, with the main result being Theorem \ref{T: intro hyperplane sweeping} (= Theorem \ref{L: hyperplane sweeping}). Another key notion we develop here is that of an $O$-\textit{good} subvariety of $X$, where $O$ is a given orbit inside $X$. These are the `all but finitely many' subvarieties for which Theorem \ref{T: intro link preservation} works. In codimension 1, the `bad' subvarieties all come from the irreducible components of the singular loci of $X_1$ and $X_2$ -- which is why there are only finitely many.
    \item Section 4 then introduces the basic setup and terminology for our homeomorphism $\varphi:X_1\rightarrow X_2$, focusing on model-theoretic notions of genericity and how they transfer across a homeomorphism.
    \item Section 5 contains the proof of Theorem \ref{T: intro link preservation}. In particular, in this section we clarify the meaning of `non-negligible' in the statement of Theorem \ref{T: intro link existence}. 
    \item Section 6 contains the proof of Theorem \ref{T: intro link existence}.
    \item Section 7 extracts Corollary \ref{C: intro ind definable} and some related statements from Theorems \ref{T: intro link existence}, \ref{T: intro link preservation}, and \ref{T: intro hyperplane sweeping}.
    \item Finally, Section 8 introduces `definably shaped sets' and proves Theorem \ref{T: intro definably shaped}, thereby ending Part 1 with the proof that $\varphi(\mathcal H)$ is definable.
\end{itemize}

\textbf{Part 2:}
\begin{itemize}
    \item Section 9 is the model-theoretic heart of the paper. Here we define the relics $\mathcal X_1^{hyp}$ and $\mathcal X_2^{hyp}$ descirbed above and show that they are full. This involves an application of the Zilber trichotomy, together with an analysis of the notion of \textit{internality} from model theory.
    \item Section 10 then applies the isomorphism theorem for full relics, thereby reducing Theorem \ref{T: main} to the case of definable homeomorphisms.
\end{itemize}

\textbf{The Rest}

\begin{itemize}
    \item Section 11 begins the study of definable homeomorphisms. Any such homeomorphism is a Frobenius power composed wth a generically rational map. This section considers only the generically rational case, and shows that such a map lifts to a morphism of normalizations.
    \item Section 12 then shows how to apply Section 10 simultaneously to both $\varphi$ and $\varphi^{-1}$ (this is a bit subtle in positive characteristic, because we cannot assume both $\varphi$ and $\varphi^{-1}$ are generically rational). The result is that the \textit{morphism} above is upgraded to a \textit{universal homeomorphism}. Thus, we deduce Theorem \ref{T: main}.
    \item Section 13 considers reducible varieties, and proves Theorems \ref{T: reducible schemes} and \ref{T: main reducible}. The jump from Theorem \ref{T: main} to Theorem \ref{T: main reducible} is again expressed using model-theoretic \textit{internality}. The idea is that under the hypotheses of Theorem \ref{T: main reducible}, we can simply define a full relic structure separately on each component of $X_1$, and by the `strongly connected' hypothesis and an internality argument, these automatically patch together into a full relic structure on all of $X_1$ (so that the rest of the proof of Theorem \ref{T: main} can be copied verbatim). 
    \item Finally, Section 14 ends the paper by showing how to explicitly derive all of the relevant `speculations' from \cite{KLOS} using our main theorems.
\end{itemize}

\section{Preliminaries}

The paper frequently uses model-theoretic language. Most of the time the structure in question is an algebraically closed field. In this case, the model-theoretic language is fairly mild, and fairly similar to the language of classical varieties over such fields. We will make some conventions explicit below. If additional confusion arises, we direct the reader to any standard text in model theory (e.g. \cite{MaBook}).

In a couple of specific sections (e.g. Sections 9, 10, and 13), we need to consider other structures (called `relics'), and for these we delve deeper into the language of \textit{stability theory}. For this, we direct the algebro-geometric reader to the appendix of \cite{CasHasAV}, which develops a dictionary between various stability-theoretic notions and their counterparts in algebraic geometry (in particular, every structure we consider in this paper satisfies the assumptions of that appendix). For completeness, we also describe some of the basic notions here. This will allow us to fix some useful conventions. Readers would likely benefit from skipping this section and referring back to it when needed.

\subsection{Imaginaries, Saturation, Dimension, and Parameter Sets} Throughout, we work with saturated, uncountable, finite Morley rank structures in countable languages. Fix such a structure $\mathcal M$. Recall that $\mathcal M^{eq}$ denotes the expansion of $\mathcal M$ by naming new `sorts' for all quotients of $\emptyset$-definable sets\footnote{Recall that a \textit{$\emptyset$-definable set} is a subset of $M^n$ singled out by a formula without parameters.} by $\emptyset$-definable equivalence relations. One often calls the elements of $\mathcal M^{eq}$ \textit{imaginaries}, and the definable sets of $\mathcal M^{eq}$ \textit{interpretable sets}.

\begin{convention}
    \textbf{Throughout this paper, in the presence of a structure $\mathcal M$ as above, the term \textit{parameter set} will refer to a \textit{countable} subset of $\mathcal M^{eq}$}. 
\end{convention}

If $A$ is a parameter set, we may say that a set $X$ is $A$-definable (this means it can be defined using only parameters from $A$). We stress that parameter sets are \textit{always countable} (this will be rather important). So for example, unless said otherwise the phrase `$X$ is $A$-definable' assumes that $A$ is a countable set.

Given any definable (or even interpretable) set $X$, one can associate an integer valued `dimension' $\dim(X)$ (for the model theorists, we mean Morley rank), and this assignment has various nice properties.\footnote{Actually Morley rank can behave a bit oddly in general, but in algebraically closed fields it is quite tame.} In an algebraically closed field, $\dim(X)$ is just the algebraic dimension of the Zariski closure $\overline X$. In particular, if $X$ is a quasi-projective variety, then $\dim(X)$ is just the dimension of $X$ as a variety, and thus we use the notation $\dim$ unambiguously.

Now given a tuple $a\in\mathcal M^{eq}$ and a parameter set $B$, one sets $\dim(a/B)$ to be the smallest value of $\dim(X)$ where $X$ ranges over all $B$-interpretable sets containing $a$. So for any $B$-interpretable set $X$ realizing this dimension, we have $\dim(a/B)=\dim(X)$. In this case, we call $a$ a \textit{$B$-generic point of $X$}. In algebraically closed fields, if $V$ is a variety defined over $A$, then an $A$-generic point of $V$ is just a point not belonging to any proper closed subvariety of $V$ defined over $A$.\footnote{This is literally different than scheme-theoretic generic points, but the two notions are morally the same. Indeed, the assertion that $a\in V$ is $A$-generic (in our language) just says that $a$ lies over the (scheme-theoretic) generic point of $V$ over the field generated by $A$.} Note that generic points always exist (that is, every $A$-definable set $X$ has an $A$-generic point); this follows by a compactness argument, using that $A$ is countable.

In algebraically closed fields, the function $\dim$ is the same as the transcendence degree map. Thus in this special case, we have the \textit{additivity formula} (which we use throughout): $$\dim(ab/C)=\dim(a/C)+\dim(b/Ca).\footnote{Here and throughout, we adopt the model-theoretic convention that in dimension-like expressions as above, the notation $AB$ denotes $A\cup B$ and the notation $ab$ denotes the pair $(a,b)$}$$

Using additivity, one can define the notion `$a$ and $b$ are independent over $C$' in a natural way, and check its basic properties. We do not dwell on this.

\subsection{Codes and Canonical Bases}\label{S: codes}

Here we outline the quite important notion of canonical bases, and (in the process) the related notion of codes. We use these notions (especially canonical bases) throughout the paper. \textbf{Geometers should view each of these as a model-theoretic version of Hilbert schemes}: in each case, the idea is to find a canonical parametrizing space for a model-theoretic object (codes parametrize definable sets, and canonical bases parametrize them up to almost equality). For example, in the appropriate context we will use notation such as $c=\operatorname{Cb}(X)$ (made precise below); this is analogous to `$c$ is the point corresponding to the variety $X$ on the relevant Hilbert scheme'.

Now let us be more precise. The construction (of both codes and canonical bases) will rely on the expansion to $\mathcal M^{eq}$, so we work with interpretable sets. First, if $X$ is interpretable in $\mathcal M$, then a \textit{code} of $X$ is a tuple $c$ so that (1) $X$ is definable by some formula $\varphi(x,c)$ and (2) for all $c'\neq c$, the formula $\varphi(x,c')$ does not define $X$. The idea is that we can now view the whole set $X$ as a single tuple in $\mathcal M^{eq}$ by identifying it with $c$. Codes can be constructed relatively easily using $\mathcal M^{eq}$, by starting with any formula $\varphi(x,b)$ defining $X$, and then replacing $b$ with its equivalence class under an obvious equivalence relation. If $c$ is a code of $X$, one often writes $c=\lceil X\rceil$ (this is well-defined in a sense because any two codes of $X$ are interdefinable).

Canonical bases are a bit more involved, and make non-trivial use of $\omega$-stability (a more general version works assuming only \textit{stability}, but this will not matter to us). First, say that an interpretable set $X$ is \textit{stationary} if $X$ is not the disjoint union of two subsets of the same dimension as $X$. Then say that two stationary interpretable sets $X,Y$ are \textit{almost equal} if the symmetric difference $X\Delta Y$ has smaller dimension than $X$ and $Y$. Then (by a non-trivial argument) one can construct `codes up to almost equality' of stationary sets, and these are called canonical bases. Precisely, for a stationary $X$, a \textit{canonical base} of $X$ is a tuple $c$ so that (1) some formula $\varphi(x,c)$ defines a set almost equal to $X$, and (2) for any $c'\neq c$, $\varphi(x,c)$ does not define such a set. As with codes, any two canonical bases of $X$ are interdefinable, so in practice one identifies them and just writes $c=\operatorname{Cb}(X)$.

One might benefit from the following intuition: over an algebraically closed field, canonical bases correspond to irreducible projective varieties, and codes correspond to arbitrary constructible sets. For example, suppose $X\subset\mathbb P^n$ is an irreducible quasi-projective variety, and $\overline X$ is its closure. Then both $\overline X$ and $X$ are stationary interpretable sets. In this case, $X$ and $\overline X$ have the \textit{same} canonical bases, but will generally have \textit{different} codes. In other words, $\operatorname{Cb}(X)$ only remembers the projectivized version $\overline X$, while $\lfloor X\rfloor$ remembers the exact set of points of $X$.

If desired, one can also make the construction of canonical bases more explicit:  
\begin{example}
    In an algebraically closed field, every irreducible quasi-projective variety $X\subset\mathbb P^n$ gives a stationary interpretable set; and in this case, $\operatorname{Cb}(X)$ is some (or any) generator of the minimal field of definition of the closure $\overline X$.
\end{example}

In a general $\omega$-stable theory, a definable stationary set might not be definable over its canonical base (it is just definable up to a set of smaller dimension). But in algebraically closed fields, things are a bit better, and this will be important:

\begin{fact}\label{F: def over cb}
    Suppose $\mathcal M$ is an algebraically closed field, and $X$ is a quasi-projective variety defined over $A$. Then for any (relatively) closed, irreducible subvariety $Y\subset X$, $Y$ is definable over $A$ together with $\operatorname{Cb}(Y)$.
\end{fact}

Fact \ref{F: def over cb} just says that one can recover $Y$ precisely from any definable set almost equal to it (since for any $Z$ almost equal to $Y$, $Y$ is the unique top-dimensional component of the relative closure of $Z$ in $X$). We use this fact without mention from now on.

\subsection{Algebraic Closure}

If $a$ is a tuple and $B$ is a parameter set, we say that $a$ is \textit{algebraic over $B$}, denoted $a\in\acl(B)$, if $a$ belongs to some finite $B$-definable set. This notion agrees with usual algebraicity over algebraically closed fields. For example, we frequently use that an $A$-definable set $X$ always splits (up to almost equality) into a union of stationary $\acl(A)$-definable sets; and over an algebraically closed field, the irreducible components of an $A$-definable variety are themselves $\acl(A)$-definable.

For example, the following simple application will be used many times throughout the paper, and serves as a good illustration of the various tools we have discussed in this section.

\begin{lemma}\label{L: codim 1 definable over acl}
    Let $W$ be an $A$-definable irreducible quasi-projective variety over an uncountable algebraically closed field, and let $V\subset W$ be a closed irreducible codimension 1 subvariety defined over $B\supset A$. Suppose that some $B$-generic element of $V$ is not $A$-generic in $W$. Then $V$ is defined over $\acl(A)$.
\end{lemma}
\begin{proof}
    Let $v\in V$ be $B$-generic but not $A$-generic in $W$. Then $V$ belongs to some proper $A$-definable subvariety $Z\subset W$. By dimension considerations (i.e. since $V$ has codimension 1), we get $$\dim(v/A)=\dim(v/B)=\dim(W)-1=\dim(V)-\dim(Z).$$ So $v$ is $B$-generic in $V$ and belongs to $Z$, which (by closedness and irreducibility) gives $V\subset Z$. Then since $\dim(V)=\dim(Z)$, $V$ is an irreducible component of $Z$. So since $Z$ is defined over $A$, $V$ is defined over $\acl(A)$.
\end{proof}

\part{Definability of Hyperplane Images}

\section{Orbits and Sweeping}

In this section we introduce the language of sweeping orbits and study the basic properties of such orbits. The main goal is to prove Theorem \ref{T: intro hyperplane sweeping}.

\textbf{Throughout this section, fix $K$, an uncountable algebraically closed field, and $X$, an irreducible quasiprojective variety over $K$ of dimension at least 2. We also fix a non-degenerate projective embedding, thereby assuming $X\subset\mathbb P^n(K)$ for some $n$, and that $X$ is not contained in any hyperplane. Absorbing parameters, we assume throughout that $X$ is definable over $\emptyset$.}

\subsection{Orbits}

We start with the definitions:

\begin{definition}
    Let $A$ be a parameter set.
    \begin{enumerate}
        \item An $A$-\textit{orbit} (or \textit{orbit over $A$}) is an orbit of closed irreducible codimension 1 subvarieties of $X$ under the action of $\operatorname{Gal}(K/A)$ (the automorphisms of $K$ fixing $A$ pointwise).
        \item If $Y$ is a closed irreducible codimension 1 subvariety of $X$, the unique $A$-orbit containing $Y$ is denoted $[Y]_A$.
        \item The set of all $A$-orbits is denoted $\operatorname{Orb}(A)$.
        \item If $O$ is an $A$-orbit, then a \textit{sub-orbit} of $A$ is an orbit of the form $P=[Y]_B$ where $B\supset A$ and $Y\in O$. We then say that $P$ is a \textit{sub-orbit of $O$ over $B$}, and we write $P\leq O$.
    \end{enumerate}
\end{definition}

Algebro-geometrically, one should think of an $A$-orbit as a family (over $A$) of subvarieties of $X$, where two such families are identified if they have the same generic members (and the members of the orbit are just the generic members of any such family). Much of our work can go through with usual families of varieties in place of orbits, but this would lead to unpleasant technicalities.

Model-theoretically, an orbit is really the same as the \textit{type of a canonical base}, in the sense that one can identify $[Y]_A$ with $\tp(c/A)$ where $c:=\operatorname{Cb}(Y)$ is the canonical base of the generic type of $Y$. In particular, there is a well-defined notion of the \textit{dimension} of an orbit: one sets $\dim([Y]_A)=\dim(\operatorname{Cb}(Y)/A)$ (this is just the dimension of any (canonically parametrized) family of varieties over $A$ with $Y$ as a generic member). Of course, if $P$ is a sub-orbit of $O$, then $\dim(P)\leq\dim(O)$.

\begin{definition}
    Let $A\subset B$, $O\in\operatorname{Orb}(A)$, and $P\in\operatorname{Orb}(B)$, with $P\leq O$. Say that $P$ is \textit{generic} in $O$ (or $P$ is a \textit{generic sub-orbit of $O$}) if $\dim(P)=\dim(O)$.
\end{definition}

Generic sub-orbits correspond to \textit{base change} in algebraic geometry, and to \textit{non-forking extensions} in model theory. (The idea is that $P$ contains the $B$-generic members of the same family over $A$ that gave rise to $O$). In particular, there is always at least one generic sub-orbit.

\begin{definition}
    Let $O\in\operatorname{Orb}(A)$. Say that $O$ is \textit{reducible} if there are $B\supset A$ and two distinct generic sub-orbits of $O$ over $B$. Otherwise, say that $O$ is \textit{irreducible}.
\end{definition}

Irreducibility corresponds to \textit{geometric irreducibility} (as a family) in algebraic geometry, and to \textit{stationarity} in model theory. In particular, every orbit over an algebraically closed field is irreducible. 

If $O\in\operatorname{Orb}(A)$ is irreducible, then for each $B\supset A$ there is exactly one generic sub-orbit of $O$ over $A$. We denote this generic sub-orbit by $O_B$.

\begin{definition}
 Given irreducible orbits $O\in\operatorname{Orb}(A),P\in\operatorname{Orb}(B)$, we say that $O$ and $P$ are \textit{parallel} if $O_{A\cup B}=P_{A\cup B}$ (i.e. they have the same generic extension to $A\cup B$).
\end{definition}

Parallelism says that $O$ and $P$ are `almost equal' (i.e. they arise as generics of the same irreducible family). It gives an equivalence relation on irreducible orbits, which is the same as the model-theoretic notion of parallelism of stationary types. Note that parallelism of $O$ and $P$ is equivalent to the existence of any $C\supset A\cup B$ with $O_C=P_C$.


The main fact we need about orbits is the following:

\begin{fact}\label{F: omega stable} For each (countable) $A$, the set $\operatorname{Orb}(A)$ is countable.
\end{fact}
\begin{proof} One can prove this algebraically by identifying each orbit with an ideal in some finitely generated $K(A)$-algebra, and then applying Hilbert's Basis Theorem (this takes a little work because one has to reduce to the affine case and then realize that the $K(A)$-algebra in question also depends on the orbit).

Alternatively, the countability of $\operatorname{Orb}(A)$ is a rather immediate corollary of the $\omega$-stability of algebraically closed fields in model theory. Indeed, each $A$-orbit is equal to $[\varphi(K)]_A$ where $\varphi=\varphi(x,b)$ is a formula with parameters $b\in K^n$ for some $n$. Then the $\operatorname{Gal}(K/A)$-orbit of $\varphi(K)$ is determined by $\tp(b/A)$. In other words, an $A$-orbit is determined by (1) a formula in the language of rings, and (2) a type over $A$. Each of these choices has only countably many possibilities (for (1) this is clear and for (2) it is $\omega$-stability). 
\end{proof}

\subsection{Sweeping and Uniquely Sweeping Orbits}

We now define the crucial `sweeping' terminology:

\begin{definition}\label{D: sweeping orbits}
    Let $O$ be an irreducible $A$-orbit, and fix a positive integer $k$.
    \begin{enumerate}
        \item $O$ is \textit{$k$-sweeping} if for $A$-generic $a\in X^k$, there is $Y\in O$ with $a\in Y^k$.
        \item $O$ is \textit{uniquely $k$-sweeping} if for $A$-generic $a\in X^k$, there is exactly one $Y\in O$ with $a\in Y^k$.
        \item More generally, let $Z\subset X$ be an $A$-definable closed irreducible subvariety. Then $O$ \textit{$k$-sweeps $Z$} (resp. \textit{uniquely $k$-sweeps $Z$}) if for $A$-generic $a\in Z^k$, there is at least one (resp. exactly one) $Y\in O$ with $a\in Y^k$.
        \item More generally, let $(Z_1,...,Z_k)$ be closed irreducible $A$-definable subvarieties of $X$. Then $O$ \textit{sweeps $(Z_1,...,Z_k)$} (resp. \textit{uniquely sweeps $(Z_1,...,Z_k)$}) if for $A$-generic $a\in Z_1\times...\times Z_k$, there is at least one (resp. exactly one) $Y\in O$ wth $a\in Y^k$.
    \end{enumerate}
\end{definition}

We immediately note the following two lemmas, which we use throughout:

\begin{lemma}\label{L: dim of sweeping orbit}
    If $O\in\operatorname{Orb}(A)$ is uniquely $k$-sweeping, then $\dim(O)=k$.
\end{lemma}
\begin{proof}
    Let $a\in X^k$ be $A$-generic, and let $Y$ be the unique member of $O$ containing $a$. Let $c=\operatorname{Cb}(Y)$. We want to show that $\dim(c/A)=k$. First, since $\dim(Y)=\dim(X)-1$, we have $$\dim(a/Ac)\leq k\cdot(\dim(X)-1)=\dim(a/A)-k,$$ from which it follows that $\dim(c/A)\geq k$. The rest of the proof will show that $\dim(c/A)\leq k$.
    
    Let $b=(b_1,...,b_{k+1})\in Y^{k+1}$ be $Ac$-generic. First we claim that $b$ is not $A$-generic in $X^{k+1}$. Indeed, if $b$ is $A$-generic, then so is $b'=(b_1,...,b_k)$ (in $X^k$). So since $O$ is uniquely $k$-sweeping, $Y$ is the unique member of $O$ containing $b'$. So $Y$ is $\operatorname{Gal}(K/Ab')$-invariant, and thus $Y$ is $Ab'$-definable. But then $\dim(b_{k+1}/Ab')\leq\dim(Y)<\dim(X)$, so that $b$ is not $A$-generic after all, a contradiction.
    
    So there is some $i\in\{0,...,k\}$ so that $b_{i+1}$ is not generic in $X$ over $Ab_1...b_i$. Fix the smallest such $i$. Then a generic element of $Y$ over $Acb_1...b_i$ is not generic in $X$ over $Ab_1...b_i$; by Lemma \ref{L: codim 1 definable over acl}, this means $Y$ is definable over $\acl(Ab_1...b_i)$. In particular, $c\in\acl(Ab_1...b_i)$, and thus $$\dim(cb_1...b_i/A)=\dim(b_1...b_i/A)=i\cdot\dim(X).$$ On the other hand, counting in the opposite direction gives $$\dim(cb_1...b_i/A)=\dim(c/A)+\dim(b_1...b_i/Ac)=\dim(c/A)+i\cdot(\dim(X)-1)$$ $$=i\cdot\dim(X)+(\dim(c/A)-i).$$ So $\dim(c/A)-i=0$, and thus $\dim(c/A)=i\leq k$, which proves the lemma.
    \end{proof}

\begin{lemma}\label{L: sweeping invariant}
    Each of the notions in Definition \ref{D: sweeping orbits} is parallelism-invariant. That is, if we replace $A$ with $B\supset A$ and replace $O$ with $O_B$ in Definition \ref{D: sweeping orbits}, the resulting notions would be equivalent to the original versions.
\end{lemma}
\begin{proof}
    For example, we show that $O$ is $k$-sweeping if and only if $O_B$ is. If $O_B$ is $k$-sweeping, then any witness also shows that $O$ is $k$-sweeping. Conversely, suppose $O$ is $k$-sweeping. Let $a\in X^k$ be $A$-generic, and let $Y\in O$ with $a\in Y^k$. Let $c=\operatorname{Cb}(Y)$. Then choose $(a',c')\models\tp(a,c/A)$ which is independent from $B$ over $A$, and let $\sigma\in\operatorname{Gal}(K/A)$ with $\sigma(a,c)=(a',c')$. Then $a'\in\sigma(Y)$ is a witness showing that $O_B$ is $k$-sweeping.

    The other cases of the lemma are similar, and we leave them to the reader.
\end{proof}

By Lemma \ref{L: sweeping invariant}, we can use terminology such as `$O$ sweeps $(Z_1,...,Z_k)$' in Definition \ref{D: sweeping orbits}, even if $O\in\operatorname{Orb}(A)$ and $Z_1,...,Z_k$ are not defined over $A$. In this case, the meaning of `$O$ sweeps $(Z_1,...,Z_k)$' is simply that $O_B$ sweeps $(Z_1,...,Z_k)$ for some (equivalently any) $B\supset A$ defining $Z_1,...,Z_k$.

\begin{example}
    Suppose $X=\mathbb P^n$, and $O$ is the $\emptyset$-orbit of generic hyperplanes. Let $H$ be any hyperplane. Then $O$ \textit{does not} $n$-sweep $H$ (and this is crucial). Indeed, any $n$ independent generics of $H$ \textit{do} lie on a hyperplane -- but that hyperplane is $H$ itself, which does not belong to $O_A$ for any $A$ defining $H$. 
\end{example}

\begin{notation} Suppose $Z_1,...,Z_k$ are closed irreducible subvarieties of $X$ so that $(Z_1,...,Z_k)$ is uniquely $k$-swept by $O\in\operatorname{Orb}(A)$. Then for generic $a\in Z_1\times...\times Z_k$, we denote the unique member of $O$ containing $a$ by $O(a)$.
\end{notation}

As with Definition \ref{D: sweeping orbits}, the map $a\mapsto O(a)$ is parallelism-invariant, so long as we restrict to those $a$ which are generic over some parameter set defining both $Z$ and some orbit in the parallelism class of $O$.

\begin{example}[Hyperplane Families] The canonical way to construct uniquely $k$-sweeping orbits is with hyperplane sections. Namely, let $L\leq\mathbb P^n(K)$ be a generic linear subspace of codimension $d$, which is $A$-definable for some set $A$. Then (using Bertini's irreducibility theorem, \cite[Theorem 6.3(4)]{BertiniIrreducibility}) there is a unique irreducible $A$-orbit in $X$ consisting of the $A$-generic hyperplane sections in $X$ over $L$, i.e. intersections of the form $H\cap X$ where $H$ is $A$-generic among all hyperplanes containing $L$. We call this orbit $H^L_A(X)$.
\end{example}

Given $L$ as above, any $d-1$ generic points determine a unique hyperplane containing $L$, and this hyperplane remains $A$-generic. Translated into the language of orbits, we have the following (if the reader prefers, this also follows from Theorem \ref{L: hyperplane sweeping} by replacing $X$ with $\mathbb P^n$ and $Y$ with $X$):

\begin{fact}
    Let $L\leq\mathbb P^n(K)$ be a generic codimension $d$ linear subspace. Then $\mathcal H^L(X)$ is uniquely $k$-sweeping for $k:=d-1$. 
\end{fact}

\begin{example}[Sub-orbit through a tuple] Suppose $k\geq 2$ and $O\in\operatorname{Orb}(A)$ uniquely sweeps $(Z_1,...,Z_k)$, where each $Z_i$ is $A$-definable. Then for $A$-generic $a\in Z_1$, we can form the `sub-orbit through $a$', consisting of those members of the form $O(a,b)$ with $b$ an $Aa$-generic point of $Z_2\times...\times Z_k$. We denote this orbit by $O\restriction a$. One sees easily that $O\restriction a$ uniquely sweeps $(Z_2,...,Z_k)$. In particular, if $O$ is uniquely $k$-sweeping, then for generic $a\in X$, the restriction $O\restriction a$ is uniquely $(k-1)$-sweeping. Moreover, one sees easily that the parallelism class of $O\restriction a$ only depends on $a$ and the parallelism class of $O$. Note that if $O=H_A^L(X)$ for some $L$, then $O\restriction a$ is just $H_A^{L'}(X)$, where $L'=L\langle a\rangle$ (the linear span of $a$ together with $L$). 

In a dual way, if $O$ uniquely sweeps $(Z_1,...,Z_k)$ and $b\in Z_2\times...\times Z_k$ is $A$-generic, then we can form the orbit $O\restriction b$ consisting of $O(a,b)$ for $Ab$-generic $a\in Z_1$. Similar to above, this orbit is uniquely 1-sweeping. We will use orbits of this form in the proof of Theorem \ref{T: intro link preservation} (see Proposition \ref{P: preservation inductive step}), and we hope no confusion will arise from using the usual restriction symbol regardless of which coordinates we are restricting. 
    
\end{example}

\subsection{Good Subvarieties}

We now introduce `good' and `bad' subvarieties of $X$ with respect to a given orbit. As described in the introduction, the bad subvarieties are the source of the `all but finitely many' clause of Theorem \ref{T: intro link preservation}.

\begin{definition}
    Let $O\in\operatorname{Orb}(A)$ be irreducible, and $Z\subset X$ a closed irreducible subvariety. $Z$ is \textit{$O$-common} if $Z\subset Y$ for some (equivalently, any) $Y\in O_{\operatorname{Cb}(Z)}$.
\end{definition}

\begin{definition}\label{D: good}
    Let $O\in\operatorname{Orb}(A)$ be irreducible, and $Z\subset X$ a closed irreducible subvariety. $Z$ is $\textit{$O$-bad}$ if it is either $O$-common or contained the singular locus of $X$. Otherwise, $Z$ is \textit{$O$-good}.
\end{definition}

By the definition, goodness is parallelism invariant (i.e. if $O,P$ are parallel, then $Y\subset X$ is $O$-good if and only if it is $P$-good).

As promised, we immediately get:

\begin{lemma}\label{L: good cofinite}
    Let $O\in\operatorname{Orb}(A)$ be irreducible of positive dimension. Then only finitely many closed irreducible codimension 1 subvarieties of $X$ are $O$-bad.
\end{lemma}
\begin{proof}
      Suppose $Y\subset X$ is closed, irreducible, codimension 1, and $O$-bad. There are two ways this could happen:
      \begin{itemize}
          \item $Y$ is contained in the singular locus of $X$. Then since $Y$ has codimension 1, it must be a component of the singular locus, and there are only finitely many such components.
          \item $Y$ is $O$-common. In this case, let $c=\operatorname{Cb}(Y)$. Then for any $Z\in O_c$ we have $Y\subset Z$, and since $Y$ has codimension 1, dimension considerations give $Y=Z$. Thus $O_c$ has just a single member, contradicting that $O$ has positive dimension. So this case is impossible.
      \end{itemize}
\end{proof}

We also need some basic examples of subvarieties we know are good:

\begin{lemma}\label{L: good preservation}
    Let $O\in\operatorname{Orb}(A)$ be uniquely $k$-sweeping.
    \begin{enumerate}
        \item $X$ is $O$-good.
        \item Suppose $k\geq 2$ and $O$ uniquely sweeps $(Z_1,...,Z_k)$, where each $Z_i$ is $A$-definable. Let $Y\subset X$ be $A$-definable and $O$-good. Then for $A$-generic $a\in Z_1$, $Y$ remains $O\restriction a$-good. Moreover, for $A$-generic $b\in Z_2\times...\times Z_k$, $Y$ also remains $O\restriction b$-good.
        
    \end{enumerate}
\end{lemma}
\begin{proof}
    \begin{enumerate}
        \item Since the members of $O$ have codimension 1 in $X$, none of them can contain $X$. Moreover, since the singular locus of $X$ is a proper closed subvariety, it does not contain $X$. 
        \item The `moreover' clause about $O\restriction b$ follows after rearranging the coordinates and then inductively restricting one at a time. So we need only show the first clause, i.e. that $Y$ is $O\restriction a$-good for $A$-generic $x\in Z_1$. So fix such an $a\in Z_1$.
        
        Now since $Y\subset X$ is $O$-good, it is not contained in the singular locus of $X$, and clearly this property is preserved under passing to $O\restriction a$. So we only need to show that $Y$ is not $O\restriction a$-common. So suppose $Y$ is $O\restriction a$-common. Let $c=\operatorname{Cb}(Y)$. Since $Y$ is $A$-definable, so is $c$, so that $O_c=O$ and $(O\restriction a)_c=O\restriction a$. Then by commonness, there is some $Z\in O\restriction a$ with $Y\subset Z$. Then $Z\in O=O_c$ and $Y\subset Z$, which shows that $Y$ is $O$-common, a contradiction.  
    \end{enumerate}
\end{proof}

Sweeping and unique sweeping behave better in $O$-good subvarieties (and this is the reason Theorem \ref{T: intro link preservation} appears the way it does). We will prove two main facts to this effect, namely Lemmas \ref{L: generic sweeping witness} and \ref{L: relativized sweeping is unique} below. 

Each of these two lemmas needs a bit of algebraic geometry (the idea being that we need to isolate what goes wrong geometrically inside the singular locus of $X$). For Lemma \ref{L: generic sweeping witness}, we need the well-known `sub-addivitiy of codimesion' in smooth varieties (which model theorists may know as the `dimension theorem' in Zariski structures). We state the special case we will use:

\begin{fact}\label{F: dimension theorem}\cite[\href{https://stacks.math.columbia.edu/tag/0AZP}{Lemma 0AZP}]{stacks-project} Suppose $W_1,W_2\subset V$ are two closed irreducible subvarieties of the smooth variety $V$, and $W_2$ has codimension 1 in $V$. Then every irreducible component of $W_1\cap W_2$ has dimension at least $\dim(W_1)-1$.
\end{fact}

Now the idea of Lemma \ref{L: generic sweeping witness} is to show that when a good subvariety is swept, it is swept in a `generic' way.

\begin{definition}\label{D: generic sweeping witness}
    Suppose $O\in\operatorname{Orb}(A)$ is irreducible, $Z_1,...,Z_k$ are $A$-definable closed irreducible subvarieties, and $O$ sweeps $(Z_1,...,Z_k)$. A \textit{generic witness to $O$ sweeping $(Z_1,...,Z_k)$} consists of a member $Y\in O$, with canonical base $c$, and a tuple $a=(a_1,...,a_k)\in Z_1\times...\times Z_k$, so that:
    \begin{enumerate}
        \item $a\in Y^k$.
    \item $a$ is $A$-generic in $Z_1\times...\times Z_k$.
    \item $\dim(Y\cap Z_i)=\dim(Z_i)-1$ for all $i$.
    \item $a$ is $Ac$-generic in $(Y\cap Z_1)\times...\times(Y\cap Z_k)$.
    \end{enumerate}
\end{definition}

\begin{lemma}\label{L: generic sweeping witness}
   Let $O,A,Z_1,...,Z_k$ be as in Definition \ref{D: generic sweeping witness}. If each $Z_i$ is $O$-good, then there is a generic witness to $O$ sweeping $(Z_1,...,Z_k)$.
\end{lemma}
    \begin{proof}
    Let $b =(b_1,...,b_k)\in Z_1\times...\times Z_k$ be $A$-generic, and (by sweeping) let $Y\in O$ with $b\in Y^k$. Let $c=\operatorname{Cb}(Y)$. By goodness, each $Y\cap Z_i$ is a proper closed subvariety of $Z_i$ containing the generic point $a_i\in Z_i$ (otherwise $Z_i$ would be $O$-common). Let $C_i$ be a component of $Y\cap Z_i$ containing $a_i$. So $C_i$ is $\operatorname{acl}(Ac)$-definable. Now since $a_i\in Z$ is generic, and $Z_i$ is $O$-good, we get that $a_i$ is a smooth point of $X_i$. By applying Fact \ref{F: dimension theorem} to the smooth locus of $X_i$, we conclude that $C_i$ has codimension 1 in $Z_i$. In particular, this implies (3) from Definition \ref{D: generic sweeping witness}.
    
    Now let $a=(a_1,...,a_k)$ be $Ac$-generic in $C_1\times...\times C_k$. Then automatically we have (1) and (4) from Definition \ref{D: generic sweeping witness}, so it remains to show (2) (i.e. that $a$ is $A$-generic in $Z_1\times...\times Z_k$). But if not, there is some proper $A$-definable closed subvariety $W\subset Z_1\times...\times Z_k$ containing $a$. By automorphism invariance, every $Ac$-generic element of $C_1\times...\times C_k$ belongs to $W$, so that by irreducibility, $C_1\times...\times C_k\subset W$. But then $b\in W$, contradicting that $b$ is $A$-generic in $Z_1\times...\times Z_k$.


\end{proof}

Next, Lemma \ref{L: relativized sweeping is unique} will show that relativized sweeping in $O$-good subvarieties is \textit{always} unique, as long as $O$ is itself uniquely sweeping. After appropriate setup, this is essentially a restatement of Zariski's Main Theorem. Below we state the version most appropriate for our needs. In what follows, let us say that a morphism $f:V\rightarrow W$ of irreducible varieties is \textit{generically bijective} if there are dense open subvarieties $V'\subset V$, $W'\subset W$ so that, if $\Gamma\subset V\times W$ is the graph of $f$, then $\Gamma\cap(V'\times W')$ is the graph of a bijection $V'\rightarrow W'$.

\begin{fact}\label{F: ZMT}
    Let $f:V\rightarrow W$ be a quasi-finite morphism of irreducible varieties, and assume $W$ is normal. If $f$ is generically bijective, then $f$ is injective.
\end{fact}
\begin{proof}
    In characteristic zero, any $f$ as above is birational. Then by Zariski's main theorem (\cite[Corollaire 18.12.13]{ega4}), $f$ is an isomorphism from $V$ to an open subvariety of $W$. This gives injectivity.
    
    In characteristic $p>0$, $f$ might not be birational, but this does not matter. In this case, we can factor $f$ as $V\rightarrow V'\rightarrow W$, where $V\rightarrow V'$ is a bijective morhpism and $V'\rightarrow W$ is separable (e.g. see \cite[Section 1.2]{ezfields} and \cite[Lemma 9.26]{CasHasYe}). Now the argument in the previous paragraph shows that $V'\rightarrow W$ is injective, and since $V\rightarrow V'$ is bijective, this is enough to see that $V\rightarrow W$ is also injective.
\end{proof}

Using Fact \ref{F: ZMT}, we now show:

\begin{lemma}\label{L: relativized sweeping is unique}
    Let $O\in\operatorname{Orb}(A)$ be uniquely $k$-sweeping, and let $Z_1,...,Z_k$ be $A$-definable $O$-good subvarieties. If $O$ sweeps $(Z_1,...,Z_k)$, then $O$ uniquely sweeps $(Z_1,...,Z_k)$.
\end{lemma}

\begin{proof}
    First, we realize $O$ as the set of generic members of an algebraic family over $A$. For example, this can be done efficiently using model theory as follows. Let $Y\in O$, let $c=\operatorname{Cb}(Y)$, and let $a\in Y$ be $Ac$-generic. Then $\tp(c/A)$ is stationary (since $O$ is irreducible), so is the generic type of an irreducible variety $T$ defined over $A$; and then $\tp(ac/A)$ is also stationary, so is the generic type of an irreducible variety $\mathcal Y\subset X\times T$ defined over $A$. Then the $A$-generic fibers of $\mathcal Y$ over $T$ are just the $\operatorname{Gal}(K/A)$ conjugates of the fiber over $c$, which are equivalently the members of $O$. Note that after replacing $T$ with an $A$-definable open subset, we may assume all fibers of $\mathcal Y\rightarrow T$ are irreducible of dimension $\dim(X)-1$.

    Now let $\mathcal Y^k_T$ be the $k$-fold fiber product of $\mathcal Y$ over $T$. Then all fibers of $\mathcal Y^k_T\rightarrow T$ are irreducible of the same dimension; it follows that $\mathcal Y^k_T$ has a unique irreducible component $V$ of top dimension, which therefore contains all $A$-generic elements of $\mathcal Y_T^k$. In particular, since the fibers of $\mathcal Y_T^k\rightarrow T$ are irreducible, this means $V$ contains the \textit{entire} fiber at every $A$-generic point of $T$.
    
    Next, we use that $O$ is uniquely $k$-sweeping. In geometric terms, this says precisely that $V\rightarrow X^k$ is generically bijective. Our goal is to set up an application of Fact \ref{F: ZMT} to this map (or rather a dense open restriction of it).

    \begin{claim} Let $z\in Z_1\times...\times Z_k$ be $A$-generic, and let $t\in T$ be $A$-generic with $z\in\mathcal Y_t^k$. Then $t\in\acl(Az)$.
    \end{claim}
    \begin{proof}
        Equivalently, we show that if $Y\in O$ contains $z$, and $c=\operatorname{Cb}(Y)$, then $c\in\acl(Az)$.

        Let $m=\sum_i\dim(Z_i)=\dim(z/A)$. By goodness, each $\dim(Z_i\cap Y)<\dim(Z_i)$, so adding over all $i$ gets $$\dim(z/Ac)\leq m-k.$$ But by Lemma \ref{L: dim of sweeping orbit}, $\dim(c/A)=\dim(O)=k$, so in total it follows that $\dim(cz/A)\leq m$. Since $\dim(z/A)=m$, the claim follows.
    \end{proof}
    
    Now suppose $O$ does not uniquely sweep $(Z_1,...,Z_k)$. Then there are an $A$-generic $z\in Z_1\times...\times Z_k$ and distinct $A$-generics $s,t\in T$ with $z\in\mathcal Y_s\cap\mathcal Y_t$. So $(z,s),(z,t)\in\mathcal Y_T^k$. Since $s,t$ are $A$-generic in $T$, we get as above that their entire fibers in $\mathcal Y_T^k$ are contained in $V$. Thus $(z,s),(z,t)\in V$.
    
    By the claim, every point in $V$ lying over $z$ is contained in $\acl(Az)$. By compactness, this implies the fiber in $V$ over $z$ is finite. Thus, both $(z,s)$ and $(z,t)$ belong to the quasi-finite locus $V_0$ of $V\rightarrow X^k$. Recall that the quasi-finite locus is open in $V$ (\cite[\href{https://stacks.math.columbia.edu/tag/01TI}{Lemma 01TI}]{stacks-project}). It follows that the restriction $V_0\rightarrow X^k$ is another generically bijective morphism of varieties, which is moreover quasi-finite.
    
    We have almost set up an application of Fact \ref{F: ZMT}; the last step is ensuring a normal target, and this is where we need goodness. Namely, since $a\in Z_1\times...\times Z_k$ is generic, and each $Z_i$ is $O$-good, we get that $a$ belongs to the smooth locus $W$ of $X^k$. Now let $V_1\subset V_0$ be the preimage of $W$, so that $(z,s),(z,t)\in V_1$. Since the smooth locus is dense open in $X^k$, it follows that $V_1\rightarrow W$ is yet another generically bijective morphism; but this time, we have both quasi-finiteness (inherited from $V_0$) \textit{and} a normal target (since $W$ is smooth). So we can apply Fact \ref{F: ZMT} to conclude that $V_1\rightarrow W$ is injective. But this contradicts that $(z,s),(z,t)\in V_1$ both map to $z\in W$.
\end{proof}

\subsection{Characterization of Unique Sweeping} Later on, we will need to recognize that a $k$-sweeping orbit is uniquely $k$-sweeping using purely topological means. Similar to Lemma \ref{L: generic sweeping witness}, it will help to know that we can look `generically' for a failure of unique sweeping:

\begin{proposition}\label{P: generic witness to non sweeping}
    Let $O\in\operatorname{Orb}(A)$ be irreducible and $k$-sweeping but not uniquely $k$-sweeping. Then there are members $Y,Z\in O$ with canonical bases $c,d$ respectively, and a tuple $a\in X^k$, so that the following hold:
    \begin{enumerate}
        \item $\dim(Y\cap Z)=\dim(X)-2$ (in particular $Y\neq Z$).
        \item $a$ is $A$-generic in $X^k$, $Ac$-generic in $Y^k$, $Ad$-generic in $Z^k$, and $Acd$-generic in $(Y\cap Z)^k$.
        \item $c$ and $d$ are independent over both $A$ and $Aa$.
    \end{enumerate}
\end{proposition}
\begin{proof}

    Throughout, for ease of notation, we will abbreviate $\dim(X)$ by $\hat d$.
    
    By Lemma \ref{L: good preservation}, $X$ is $O$-good. So by Lemma \ref{L: generic sweeping witness}, let $Y,a$ be a generic witness to $O$ seeping $(X,...,X)$ ($k$ times). Let $c=\operatorname{Cb}(Y)$.
    
    \begin{claim}
        There is $Z\neq Y$ in $O$ with canonical base $d$ so that $a$ is $Ad$-generic in $Z$, $\dim(c/Aa)=\dim(d/Aa)$, and $c$ and $d$ are independent over $Aa$. 
    \end{claim}
    \begin{proof}
        If $\dim(c/Aa)>0$ this is automatic by taking $d\models\tp(c/Aa)$ independent from $c$ over $Aa$ (then $c\neq d$ because $\dim(d/Aa)=\dim(c/Aa)>0$).

        Now assume $\dim(c/Aa)=0$. In this case, let $Z$ be \textit{any} other member of $O$ with $a\in Z^k$ (which exists because $O$ is not uniquely $k$-sweeping). We claim that $Z$ and $d=\operatorname{Cb}(Z)$ satisfy the claim. It is clear that $c$ and $d$ are independent over $Aa$, since $\dim(c/Aa)=0$. So it remains to show that $a$ is $Ad$-generic in $Y^k$ and $\dim(d/Aa)=\dim(c/Aa)=0$.
        
        Now we have $\dim(ac/A)=\dim(a/A)=k\hat d$, while $\dim(a/Ac)=k\cdot(\hat d-1)$; thus $\dim(c/A)=k$, and thus $\dim(d/A)=k$ (as $c,d$ are automorphism conjugate over $A$ by definition). Since $a\in Z^k$, we also have $\dim(a/Ad)\leq k\cdot(\hat d-1)$. So: $$\dim(ad/A)=\dim(d/A)+\dim(a/Ad)\leq k+k\cdot(\hat d-1)=k\hat d=\dim(a/A)\leq\dim(ad/A).$$ So all inequalities above are equalities. In particular, we get $\dim(a/Ad)=k\cdot(\hat d-1)$, so that $a$ is $Ad$-generic in $Z^k$; and we get $\dim(ad/A)=\dim(a/A)$, so that $\dim(d/Aa)=0$. This proves the claim.
        \end{proof}
        
        Now let $Z$ and $d$ be as in the claim. We show that they satisfy the requirements of the proposition. It remains to show that $c$ and $d$ are independent over $A$, and that $\dim(a/Acd)=\hat d-2$ (which combined with $Y\neq Z$ shows that $\dim(Y\cap Z)=2$ and $a$ is $Acd$-generic in $Y\cap Z$).
        
        Now let $m=\dim(c/A)=\dim(d/A)=\dim(O)$. Then $\dim(ac/A)=m+k\cdot(\hat d-1)$, while $\dim(a/A)=k\cdot d$. Thus $\dim(c/Aa)=m-k$, and thus $$\dim(acd/A)=\dim(ac/A)+\dim(d/Aac)=\dim(ac/A)+\dim(c/Aa)$$ $$=m+k\cdot(\hat d-1)+m-k=2m+k\cdot(\hat d-2).$$ (where we use $\dim(d/Acd)=\dim(d/Aa)=\dim(c/Aa)$, which follows from the claim).
        
        On the other hand, clearly $\dim(cd/A)\leq 2m$ and (since $Y\cap Z$) $\dim(a/Acd)\leq k\cdot(\hat d-2)$. So since $\dim(acd/A)=2m+k\cdot(\hat d-2)$, both of these inequalities must be equalities. This finishes the proof of the proposition.
\end{proof}

\subsection{Characterization of hyperplane sections via sweeping}

Here we prove Theorem \ref{T: intro hyperplane sweeping}. This theorem says that the collection of \textit{all} irreducible components of hyperplane sections in $X$ is determined by the sweeping data of just the \textit{generic} hyperplanes. Recall that $n$ is the dimension of the ambient projective space containing $X$.

\begin{theorem}\label{L: hyperplane sweeping} Let $O$ be the orbit of generic hyperplane sections in $X$ (i.e. $O=H_\emptyset^{\emptyset}(X)$). Then for any closed irreducible subvariety $Z\subset X$, the following are equivalent:
\begin{enumerate}
    \item $O$ $n$-sweeps $Z$.
    \item $O$ uniquely $n$-sweeps $Z$.
    \item $Z$ is non-degenerate.
\end{enumerate}

In particular, the irreducible components of the hyperplane sections in $X$ are precisely the closed irreducible codimension 1 subvarieties of $X$ which are not $n$-swept by $O$.
\end{theorem}
\begin{proof}

    The second clause is immediate from the equivalence of (1) and (3): the only thing to note is that if $\dim(Z)=\dim(X)-1$, then saying $Z$ is degenerate (i.e. contained in a hyperplane section) is equivalent to saying that $Z$ is an irreducible component of that hyperplane section (which is clear by dimension considerations).

    The rest of the proof will show that (1) implies (3) and (3) implies (2) (since (2) implies (1) is trivial). For each of these, by Lemma \ref{L: sweeping invariant}, we may assume $Z$ is $\emptyset$-definable, since adding parameters does not change whether $O$ $n$-sweeps $Z$ (uniquely or not).
    
    Now we first show (1) implies (3). So toward a contradiction, assume $O$ $n$-sweeps $Z$ and $Z$ is contained in the hyperplane $H_0$. Adding parameters again, we may assume $H_0$ is $\emptyset$-definable. Now let $a=(a_1,...,a_n)\in Z^n$ be generic, and (by sweeping) let $Y\in O$ with $a\in Y^k$. Then $Y=H\cap H_0$ for some generic hyperplane $H$. Let $c=\operatorname{Cb}(Y)$, and let $t$ be a code for the intersection $H\cap H_0$. Viewing $H\cap H_0$ as a generic hyperplane inside the $(n-1)$-dimensional projective space $H_0$, we have $\dim(t)=n-1$, and thus (since $Y$ is determined by $t$) we get $\dim(c)\leq n-1$.
    
    Now we have $$\dim(ac)\geq\dim(a)=n\cdot\dim(Z),$$ so that $$\dim(a/c)=\dim(ac)-\dim(c)\geq n\cdot\dim(Z)-(n-1)>n\cdot(\dim(Z)-1).$$ It follows that $\dim(a_i/c)=\dim(Z)$ for some $i$. So a generic member of $Z$ belongs to $Y$, and thus (by closedness and irreducibility) $Z\subset Y$. But then by automorphism invariance (since $Z$ is $\emptyset$-definable), $Z$ belongs to \textit{every} generic hyperplane section (that is, $Z$ is $O$-common) -- and this is impossible because the intersection of $n+1$ hyperplanes in general position is empty. (Precisely, we are contradicting that $Z$ is assumed irreducible, thus non-empty).
    
    Now we show (3) implies (2). So assume $Z$ is non-degenerate. Let $a\in Z^n$ be generic, and let $H$ be a hyperplane containing $a$ (here we use that any $n$ points lie on some hyperplane). This time, let $t$ be a code for $H$ (see Section \ref{S: codes}). Since $Z$ is non-degenerate over $L$ we have $\dim(H\cap Z)<\dim(Z)$, so that $$\dim(a/t)\leq n\cdot(\dim(Z)-1)=\dim(a)-n.$$ By additivity, it must be that $\dim(t)\geq n$, i.e. that $H$ is a generic hyperplane. Clearly, then, $\dim(t)$ is exactly $n$, which implies all equalities above are equalities. Thus $\dim(a/t)=\dim(a)-n$, and so adding to $\dim(t)=n$ gives $\dim(at)=\dim(a)$, i.e. $t\in\acl(a)$. Since $t$ was arbitrary (subject to coding a hyperplane through $a$), it follows by compactness that only finitely many hyperplanes pass through $a$. But the hyperplanes through $a$ form a linear subspace of $\mathbb P^n$, and the only finite subspace is the trivial one; so in fact there is only one hyperplane through $a$, which must be our chosen hyperplane $H$ from above. So we have shown that $H\cap X$ is the unique generic member of $O$ containing $a$, which proves (2). 
\end{proof}

\section{The Setup}

We now introduce the setup and terminology we will use for the remainder of Part 1 of the paper. The main mathematical goal of this section is to show that homeomorphisms `usually' behave well with respect to model-theoretic genericity and independence. This will be crucial in the next section for transferring sweeping data through a homeomorphism.

\subsection{Notation and Conventions}

\begin{convention}\label{Convention homeomorphism} \textbf{Until stated otherwise, we fix $\varphi:X_1\rightarrow X_2$, a homeomorphism between irreducible quasiprojective varieties $X_i$ of dimension $\hat d\geq 2$ over uncountable algebraically closed fields $K_i$. We can and do assume $X_1$ is a non-degenerate locally closed subvariety of $\mathbb P^n(K_1)$ for some fixed $n$.} 
\end{convention}

\begin{convention} We moreover fix the following notational conventions until otherwise stated:
\begin{enumerate}
    \item The subscripts 1 and 2 will correspond to the fields $K_1$ and $K_2$. For example, a parameter set denoted $A_1$ should be assumed to reside in $K_1$.
    \item All model-theoretic notions are interpreted in the sense of the fields $K_1$ and $K_2$, which we may augment as we go by adding countably many constants to each. \textbf{In particular, we assume throughout that $X_i$ is $\emptyset$-definable in $K_i$ for each $i$.}
    \item We reiterate that all parameter sets are assumed countable. A \textit{pair} of parameter sets is denoted $(A_1,A_2)$ (or similar with $B$, $C$, ...) and refers to a choice of countable sets $A_i$ from each $K_i$.
    \end{enumerate}
\end{convention}

\subsection{Generic points}

We would like to continue using the model-theoretic tools associated with generic points and independence. However, this is now a bit subtle, because a priori $\varphi$ might not preserve genericity. For example, one might imagine a generic point $a_1\in X_1^k$ for some $k$ so that $a_2=\varphi(a_1)\in X_2^k$ is not generic. Fortunately, this turns out to be a rare occurrence. Our goal in this subsection is to make precise the statement that `we can always take generic points that remain generic under $\varphi$'.

First we make the relevant definitions:

\begin{definition}\label{D: generic pair}
    Let $Y_i^1,...,Y_i^k\subset X_i$ be $A_i$-definable with each $\varphi(Y_1^j)=Y_2^j$.
    \begin{enumerate}
        \item Say that $y_1\in Y_1^1\times...\times Y_1^k$ is $(A_1,A_2)$-\textit{generic} if $y_1$ is $A_1$-generic in $Y_1^1\times...\times Y_1^k$ and $\varphi(y_1)$ is $A_2$-generic in $Y_2^1\times...\times Y_2^k$. 
        \item Similarly, say that $y_2\in Y_2^1\times...\times Y_2^k$ is $(A_1,A_2)$-generic if $y_2$ is $A_2$-generic in $Y_2^1\times...\times Y_2^k$ and $\varphi^{-1}(y_2)$ is $A_1$-generic in $Y_1^1\times...\times Y_1^k$.
        \item The phrases `$(y_1,y_2)$ is $(A_1,A_2)$-generic in $(Y_1^1\times...\times Y_1^k,Y_2^1\times...\times Y_2^k)$' and `$(y_1,y_2)\in(Y_1^1\times...\times Y_1^k,Y_2^1\times...\times Y_2^k)$ is $(A_1,A_2)$-generic' mean that each $y_i$ is $A_i$-generic in $Y_i^1\times...\times Y_i^k$, and moreover $\varphi(y_1)=y_2$.
    \end{enumerate}
\end{definition}

So `$y_1$ is $(A_1,A_2)$-generic in $Y_1^1\times...\times Y_1^k$' is equivalent to `$(y_1,\varphi(y_1))$ is $(A_1,A_2)$-generic in $(Y_1^1\times...\times Y_1^k,Y_2^1\times...\times Y_2^k)$', and similarly for $y_2$ and $Y_2$. 

\begin{remark} It is tempting to define $(A_1,A_2)$-genericity for arbitrary definable subsets of powers of $X$; but that does not a priori make sense, because we do not know that $\varphi$ preserves all definable sets. The point of Definition \ref{D: generic pair} is that $\varphi$ \textit{does} preserve definable sets which are products of unary sets, so in that case the definition \textit{does} make sense.
\end{remark}

Now we show in a precise sense that $(A_1,A_2)$-generic points always exist:

\begin{lemma}\label{L: generics exist}
    Let $\varphi(Y_1^j)=Y_2^j$ for $j=1,...,k$, where each $Y_i^j\subset X_i$ is $A_i$-definable. Then there is an $(A_1,A_2)$-generic point of $(Y_1^1\times...\times Y_1^k,Y_2^1\times...\times Y_2^k)$.
\end{lemma}
\begin{proof}
    First assume $k=1$, and write $Y_i^1$ as just $Y_i$. Consider the collection $\mathcal C$ of all definable sets $Z\subset Y_1$ so that $\dim(Z)<\dim(Y_1)$ and either $Z$ is $A_1$-definable or $\varphi(Z)$ is $A_2$-definable. Since $A_1,A_2$ are countable, so is $\mathcal C$. Now if there are no $(A_1,A_2)$-generic points of $(Y_1,Y_2)$, then every element of $Y_1$ belongs to an element of $\mathcal C$. Thus $Y_1$ is contained in the union of countably many definable subsets of smaller dimension, and this contradicts that $K_1$ is uncountable (model-theoretically, we contradict that $K_1$ is $\aleph_1$-saturated). 

    Now consider $Y_i^1,...,Y_i^{k+1}$ where $k\geq 1$ and the lemma is true for $1\leq j\leq k$. By the above paragraph, there is an $(A_1,A_2)$-generic element $(a_1,a_2)\in(Y_1^1\times...\times Y_1^k,Y_2^1\times...\times Y_2^k)$. Then by the inductive hypothesis, there is an $(A_1a_1,A_2a_2)$-generic element $(b_1,b_2)\in (Y_1^{k+1},Y_2^{k+1})$. Then $(a_1b_1,a_2b_2)$ is $(A_1,A_2)$-generic in $(Y_1^1\times...\times Y_1^{k+1},Y_2^1\times...\times Y_2^{k+1})$.
\end{proof}

\subsection{Coherent Parameter Sets}

In the case of unary sets (i.e. subsets of $X$ only and not their products), Lemma \ref{L: generics exist} can be strengthened to a much cleaner statement. The idea is that we can force $\varphi$ to \textit{exactly} preserve generic points, at the cost of adding parameters to the language. This formalism will help, for example, in the base case of Theorem \ref{T: intro link preservation} (but it is not too helpful after we move to higher arities).

The following is inspired by the notion of coherence in \cite{CasACF0}: 

\begin{definition}
    Let $(A_1,A_2)$ be a pair of parameter sets. Say that $(A_1,A_2)$ is \textit{coherent} if for every definable $Y_i\subset X_i$ with $\varphi(Y_1)=Y_2$, we have have that $Y_1$ is $A_1$-definable if and only if $Y_2$ is $A_2$-definable.
\end{definition}

So if $(A_1,A_2)$ is coherent, then for all $a\in X_1$ we have $\dim(a/A_1)=\dim(\varphi(a)/A_2)$.

\begin{lemma}\label{L: existence of coherent sets}
    Let $(A_1,A_2)$ be any pair of parameter sets. Then there is a coherent pair $(B_1,B_2)$ with each $B_i\supset A_i$.
\end{lemma}
\begin{proof}
    Write $B_i$ as the union of a chain $A_i=B_i^0\subset...\subset B_i^j\subset...$, where $B_1^{j+1}$ can define the $\varphi$-preimages of all $B_2^j$-definable sets and $B_2^{j+1}$ can define the $\varphi$-images of all $B_1^j$-definable sets. 
\end{proof}


\subsection{Large and Small Sets}

We now begin to foreshadow the meaning of `non-negligible' in the statement of Theorem \ref{T: intro link existence}. The idea is that we want to isolate a class of \textit{small} subsets of $X_i^k$ for each $k$, which should meaningfully reflect the topology of $X_i$ in some way. For $k=1$, the small sets will be exactly the \textit{meager} sets in the sense of descriptive set theory (so what we are really doing is setting up some Baire category style arguments about generic and negligible behavior; this crucially uses the uncountability of the fields $K_i$, and is actually one of the key uses of uncountability in this paper).

For $k>1$, the small sets are defined inductively by analyzing fibers of projections. This makes the resulting notion rather awkward (for example the small sets are not invariant under coordinate permutations), but it will still be easy to use and sufficient for our needs.

Let us proceed with the definition:

\begin{definition}
    We inductively define the \textit{large} and \textit{small} subsets of $X_1^k$ for $k\geq 1$ as follows:
    \begin{enumerate}
        \item Suppose $S\subset X_1$. Then $S$ is \textit{small} if it is contained in the union of countably many proper closed subvarieties of $S$, and is \textit{large} otherwise.
        \item Now suppose $k\geq 1$ and we have defined the large and small subsets of $X_1^j$ for $1\leq j\leq k$. Let $S\subset X_1^{k+1}$, and let $\pi:X_1^{k+1}\rightarrow X$ be the leftmost projection. Now say that $S$ is \textit{large} if there is a large set $S'\subset X_1$ so that for each $a\in S'$, the set $S_a=\{x\in X_1^k:(a,x)\in S\}$ is large. Otherwise, say that $S$ is \textit{small}.
    \end{enumerate}
\end{definition}

Lemma \ref{L: sigma ideal} says that the small sets behave like a notion of `small':

\begin{lemma}\label{L: sigma ideal}
    For each $k\geq 1$, the small subsets of $X_1^k$ form a $\sigma$-ideal. That is:
    \begin{enumerate}
    \item $X_1^k$ is large.
    \item Any subset of a small set is small.
    \item The union of countably many small sets is small.
    \end{enumerate}   
\end{lemma}
\begin{proof}
\begin{enumerate}
    \item For $k=1$ this is because $X_1$ is not the union of countably many proper closed subvarieties (for example because $K_1$ is $\aleph_1$-saturated, being uncountable and algebraically closed). Then the general case is clear by induction.
    \item Clear.
    \item For $k=1$ this is clear. Now working by induction, let $k\geq 1$ and let $S_1,S_2,...$ be small subsets of $X_1^{k+1}$. Toward a contradiction, suppose $S=S_1\cup S_2\cup...$ is large. Let $S'\subset X_1$ be large so that $S_a$ is large for all $a\in S'$. Noting that $S_a=\bigcup_i(S_i)_a\subset X_1^k$, it follows by the inductive hypothesis that for each $a\in S'$, some $(S_i)_a$ is large. Let $T_i$ be the set of $a$ so that $(S_i)_a$ is large. Then the $T_i$ union to $S'$, and by the base case, some $T_i$ is large. Then by definition, $S_i$ is large, a contradiction.
\end{enumerate}
\end{proof}

Lemma \ref{L: sigma ideal} justifies common $\sigma$-ideal terminology when speaking of large and small sets. For example, we might say that a property holds `large often', or `for large many $a\in X_1^k$', etc.

Lemma \ref{L: generics in large set} and Corollary \ref{C: non-generics are small} refine Lemma \ref{L: generics exist} by showing that `most' points are generic on both sides of $\varphi$:

\begin{lemma}\label{L: generics in large set}
    Let $S\subset X_1^k$ be large. Then for any $(A_1,A_2)$, $S$ contains an $(A_1,A_2)$-generic element of $X_1^k$.
\end{lemma}
\begin{proof}
    Essentially by the same argument as Lemma \ref{L: generics exist}. First assume $k=1$ and consider the (countable) collection $\mathcal C$ of all proper closed subvarieties $Y\subset X_1$ so that either $Y$ is $A_1$-definable or $\varphi(Y)$ is $A_2$-definable. Since $S$ is large, it cannot be covered by $\mathcal C$ -- and any element of $S-\bigcup\mathcal C$ is $(A_1,A_2)$-generic.

    Now assume $k\geq 1$ and the lemma is true for $1\leq j\leq k$. By definition, there is a large set $S'\subset X_1$ so that for each $a\in S'$ the set $S_a$ is large. By the base case, we can choose $a_1\in S'$ which is $(A_1,A_2)$-generic in $X_1$. Let $a_2=\varphi(a_1)$. Then by the inductive hypothesis, we can choose $b_1\in S_{a_1}$ which is $(A_1a_1,A_2a_2)$-generic in $X_1^k$. Then $a_1b_1\in X_1^{k+1}$ is $(A_1,A_2)$-generic and belongs to $S$. 
\end{proof}

\begin{corollary}\label{C: non-generics are small}
    \begin{enumerate}
        \item For any $(A_1,A_2)$, the set $N_k$ of $a\in X_1^k$ which are not $(A_1,A_2)$-generic is small.
        \item If $S\subset X_1^k$ is large, then for any $(A_1,A_2)$, the set $S'$ of $(A_1,A_2)$-generics in $X_1^k$ which belong to $S$ is also large.
    \end{enumerate}
\end{corollary}
\begin{proof}
    \begin{enumerate}
        \item Otherwise the previous lemma would give an $(A_1,A_2)$-generic element of $N_k$, which is impossible.
        \item Let $N_k$ be the small set from (1). Then $S=S'\cup N_k$. If $S'$ were small, then since $N_k$ is also small by (1), we would get that $S$ is small, a contradiction.
    \end{enumerate}
\end{proof}

\section{Linked Orbits and Their Properties}

In this section, we prove Theorem \ref{T: intro link preservation}. First, we need to state it formally -- meaning we need to clarify what it means for the image of an orbit under $\varphi$ to have `non-negligible overlap' with another orbit. This is our most crucial use of large and small sets from the previous section.

\subsection{The Definition}

\begin{definition}\label{D: linked} For each $i=1,2$, let $O_i$ be a uniquely $k$-sweeping orbit in $X_i$. We say that $O_1,O_2$ are \textit{linked} if for large many $a\in X_1^k$ we have $$\varphi(O_1(a))=O_2(\varphi(a)).$$
\end{definition}

\begin{remark} Note that the equation $\varphi(O_1(a))=O_2(\varphi(a))$ above implicitly assumes that both sides are defined -- equivalently, that $a$ is $(A_1,A_2)$-generic where $O_i\in\operatorname{Orb}(A_i)$.
\end{remark}

As with many other notions we have defined, we should immediately note:

\begin{lemma}\label{L: linked invariant}
    Linkedness is parallelism invariant. That is, suppose $O_i$ are uniquely $k$-sweeping orbits over $A_i$, and $B_i\supset A_i$. Then $O_1,O_2$ are linked if and only if $(O_1)_{B_1},(O_2)_{B_2}$ are.
\end{lemma}
\begin{proof}
    Let $S_A$ be the set of $(A_1,A_2)$-generics $a\in X_1$ satisfying $\varphi(O_1(a))=O_2(\varphi(a))$, and let $S_B$ be defined analogously for $(O_1)_{B_1}$ and $(O_2)_{B_2}$. We want to show that $S_A$ is large if and only if $S_B$ is.

    By definition we have $S_B\subset S_A$ (i.e. generic points remain generic over smaller parameters), so if $S_B$ is large then so is $S_A$. Conversely, suppose $S_A$ is large. Then by Corollary \ref{C: non-generics are small}(2), $S_A$ contains a large set of $(B_1,B_2)$-generic points, and all such points belong to $S_B$. So $S_B$ is also large.
\end{proof}

In particular, we obtain the following, which we use frequently:

\begin{lemma}\label{L: weak linked}
    Let $O_1,O_2$ be linked, uniquely $k$-sweeping orbits over $A_1,A_2$ respectively. For any $B_i\supset A_i$, there are elements $Y_i\in(O_i)_{B_i}$ with $\varphi(Y_1)=Y_2$.
\end{lemma}
\begin{proof}
     By Lemma \ref{L: linked invariant}, $(O_1)_{B_1}$ and $(O_2)_{B_2}$ are linked, which is more than enough.
\end{proof}

We also note the following, which will enable inductive arguments in proofs about linked families:

\begin{lemma}\label{L: linked restriction}
    Suppose $k>1$ and $O_i\in\operatorname{Orb}(A_i)$ are uniquely $k$-sweeping and linked. Then there is a large set $S\subset X_1$ consisting of $(A_1,A_2)$-generic elements $a\in X_1$ so that $O_1\restriction a$ and $O_2\restriction\varphi(a)$ are linked. 
\end{lemma}
\begin{proof}
    Since $O_1,O_2$ are linked, there is a large set $S\subset X_1^k$ consisting of $(A_1,A_2)$-generic elements $b$ so that that $\varphi(O_1(b))=O_2(\varphi(b))$. By definition, there is a large set $S'\subset X_1$ so that for all $a\in S'$, the fiber $S_a=\{x:(a,x)\in S\}$ is large. This shows that $O_1\restriction a$ and $O_2\restriction\varphi(a)$ are linked for each $a\in S'$. Indeed, for $a\in S'$ and $x\in S_a$, we get $$\varphi((O_1\restriction a)(x))=\varphi(O_1(a,x))=O_2((\varphi(a),\varphi(x))=(O_2)_{\varphi(a)}(\varphi(x)).$$
\end{proof}

\subsection{Statement of Preservation Theorem} We now give the formal version of Theorem \ref{T: intro link preservation}. Informally, we want to say that linked orbits sweep the same subvarieties, with few exceptions. As described in the introduction, it is more convenient to prove a more complicated statement by induction.

\begin{theorem}[= Theorem \ref{T: intro link preservation}]\label{T: preservation}
    Let $O_1,O_2$ be linked, where each $O_i$ is a uniquely $k$-sweeping $A_i$-orbit in $X_i$. Suppose $\varphi(Z_1^j)=Z_2^j$ for $j=1,...,k$, where each $Z_i^j$ is $O_i$-good. Then:
    \begin{enumerate}
        \item $O_1$ sweeps $(Z_1^1,...,Z_1^k)$ if and only if $O_2$ sweeps $(Z_2^1,...,Z_2^k)$.
        \item Suppose each $O_i$ $k$-sweeps $(Z_i^1,...,Z_i^k)$, and each $Z_i^j$ is $A_i$-definable. Then for any $(A_1,A_2)$-generic $(a_1,a_2)\in(Z_1^1\times...\times Z_1^k,Z_2^1\times...\times Z_2^k)$ we have $\varphi(O_1(a_1))=O_2(a_2)$.
    \end{enumerate}
\end{theorem}

Clause (1) is the main point we need. Clause (2) says informally that linked orbits are `almost everywhere linked', and is a crucial part of both the inductive step and the proof of Theorem \ref{T: intro link existence} (in fact, this clause is why we prove Theorem \ref{T: intro link preservation} before Theorem \ref{T: intro link existence}).

The base case and inductive steps of Theorem \ref{T: preservation} are both rather involved. Because of this, we give them separate subsections below.

\subsection{The Base Case:}

\begin{proposition}\label{P: base case}
    Theorem \ref{T: preservation} holds if $k=1$.
\end{proposition}

\begin{proof}
    Assume the setup of Theorem \ref{T: preservation}, and moreover assume $k=1$. For simplicity we write $Z_i=Z_i^1$ for each $i$. We may assume each $Z_i$ is $A_i$-definable: otherwise, it is definable over a set $B_i\supset A_i$; and then we can replace $A_i$ with $B_i$ and $O_i$ with $(O_i)_{B_i}$ (as by Lemmas \ref{L: linked invariant} and \ref{L: sweeping invariant}, this does not change whether the $O_i$ are linked and whether the $Z_i$ are swept). So assume $Z_i$ is $A_i$-definable from now on.

    Since $O_1,O_2$ are linked, there is a large set $S\subset X_1$ so that for $b\in S$ we have $\varphi(O_1(b))=O_2(\varphi(b))$.
    
    First we show (1). So let us assume $O_1$ sweeps $Z_1$ and show that $O_2$ sweeps $Z_2$ (the other direction is similar).

     By Lemma \ref{L: existence of coherent sets}, let $(B_1,B_2)$ be coherent with each $B_i\supset A_i$; then by Lemma \ref{L: generics in large set}, let $(b_1,b_2)$ be $(B_1,B_2)$-generic in $(X_1,X_2)$ with $b_1\in S$. So $\varphi(O_1(b_1))=O_2(b_2)$. Since $O_1$ sweeps $Z_1$, Lemma \ref{L: generic sweeping witness} gives a $B_1$-generic $a_1\in Z_1$ with $a_1\in O_1(b_1)$. Then $a_2=\varphi(a_1)\in O_2(b_2)$, and by coherence, $a_2$ is $B_2$-generic in $Z_2$. So $a_2\in O_2(b_2)$, which belongs to $(O_2)_{B_2}$. This shows $(O_2)_{B_2}$ sweeps $Z_2$, and thus so does $O_2$.

    Now we show (2). So assume each $O_i$ sweeps $Z_i$, and let $(a_1,a_2)\in(Z_1,Z_2)$ be $(A_1,A_2)$-generic. Let $Y_1=O_1(a_1)$ and $Y_2=\varphi(Y_2)$. Again by Lemma \ref{L: existence of coherent sets}, let $(B_1,B_2)$ be coherent with $B_i\supset A_i$; we may assume that each $Y_i$ is $B_i$-definable (by adding a parameter defining $Y_i$ to $A_i$ before forming $B_i$). Then again let $(b_1,b_2)$ be $(B_1,B_2)$-generic in $(X_1,X_2)$ with $b_1\in S$, and let $W_i=O_i(b_i)$; so $\varphi(W_1)=W_2$.

    Note that $b_1\notin Y_1$ since it is $B_1$-generic; thus $Y_1\neq W_1$. In particular, since $O_1$ is uniquely 1-sweeping, no $B_1$-generic element of $Y_1$ can belong to $W_1$ (otherwise such a point would belong to two members of $O_1$). By coherence, no $B_2$-generic element of $Y_2$ belongs to $W_2$. By automorphism invariance, no $B_2$-generic element of $Y_2$ can belong to a member of $(O_2)_{B_2}$. On the other hand:

        \begin{claim} Every $B_2$-generic element of $Y_2$ is contained in a member of $O_2$.
    \end{claim}
    \begin{proof}
        If every $B_2$-generic element $y\in Y_2$ is $A_2$-generic in $X_2$, then we satisfy the claim automatically since $y\in O_2(y)$.
        
        Otherwise, assume every $B_2$-generic element of $Y_2$ is not $A_2$-generic in $X_2$. Then by Lemma \ref{L: codim 1 definable over acl}, $Y_2$ is $\acl(A_2)$-definable. But $a_2$ is $A_2$-generic in $Z_2$ and belongs to $Y_2$, so if $Y_2$ is $\operatorname{acl}(A_2)$-definable, this gives $Z_2\subset Y_2$. Thus $Z_1\subset Y_1$, and thus $Z_1$ is $O_1$-common, contradicting that $Z_1$ is $O_1$-good.
    \end{proof}

    So by the claim and the paragraph before it, every $B_2$-generic member of $Y_2$ belongs to some member of $O_2-(O_2)_{B_2}$. On the other hand:
    
    \begin{claim}
        There are only countably many members of $O_2-(O_2)_{B_2}$.
    \end{claim}
    \begin{proof} Let $U_2\in O_2-(O_2)_{B_2}$ with canonical base $e_2$. If $u_2\in U_2$ is $A_2e_2$-generic, then $u_2$ is not $B_2$-generic in $X_2$ (otherwise $U_2=O_2(u_2)\in(O_2)_{B_2}$). So by Lemma \ref{L: codim 1 definable over acl}, $U_2$ is $\operatorname{acl}(B_2)$-definable. Then the claim follows since there are only countably many $\acl(B_2)$-definable sets.
    \end{proof}

    Finally, by the two claims, compactness, and irreducibility, $Y_2$ is contained in a single member $U_2\in O_2-(O_2)_{B_2}$. By dimension considerations and irreducibility, $Y_2=U_2$, and thus $U_2=O_2(a_2)$ (since $a_2\in Y_2$). So $$\varphi(O_1(a_1))=\varphi(Y_1)=Y_2=U_2=O_2(a_2),$$ which proves (2) and thus finishes the proof of Proposition \ref{P: base case}.
    \end{proof}


\subsection{The Inductive Step}

\begin{proposition}\label{P: preservation inductive step}
    Let $k,O_i,A_i,Z_i^j$ be as in Theorem \ref{T: preservation}. Assume that $k\geq 2$ and Theorem \ref{T: preservation} holds for $k'<k$. Then Theorem \ref{T: preservation} holds for $k,O_i,A_i,Z_i^j$.
\end{proposition}
\begin{proof}
    As in the base case, we assume throughout that each $Z_i^j$ is $A_i$-definable. Throughout the proof, fix a large set $S\subset X_1$ so that for all $a\in S$ the orbits $O_1\restriction a$ and $O_2\restriction\varphi(a)$ are linked (which exists by Lemma \ref{L: linked restriction}). Throughout the proof, we will apply the inductive hypothesis to different restrictions of the $O_i$ at generic points. Each time, we are using that the $Z_i^j$ remain good after generic restrictions (i.e. Lemma \ref{L: good preservation}(2)).
    
    Assume that $O_1$ sweeps $(Z_1^1,...,Z_1^k)$. We will show that $O_2$ sweeps $(Z_2^1,...,Z_2^k)$, and that for all $(A_1,A_2)$-generic elements $(a_1,a_2)\in(Z_1^1\times...\times Z_1^k,Z_2^1\times...\times Z_2^k)$ we have $\varphi(O_1(a_1))=O_2(a_2)$. By symmetry, this is enough to prove the whole proposition.

    We proceed with a sequence of claims.

    \begin{claim}\label{Cl: preservation inductive step claim 1} For generic $a\in X_1$, $O_1\restriction a_1$ sweeps $(Z_1^2,...,Z_1^k)$.
    \end{claim}
    \begin{proof}
        Let $z_1=(z_1^1,...,z_1^k)\in Z_1^1\times...\times Z_1^k$ be $A_1$-generic, and let $Y_1=O_1(z_1)$. If some generic element $y_1\in Y_1$ is $A_1z_1^2...z_1^k$-generic in $X_1$, then $Y_1$ belongs to $O(y_1)$ and contains the $A_1z_1^2...z_1^k$-generic point $(z_1^2,...,z_1^k)\in Z_1^2\times...\times Z_1^k$, which implies the claim. Otherwise, assume every generic element of $Y_1$ is not $A_1z_1^2...z_1^k$-generic in $X_1$. Then by Lemma \ref{L: codim 1 definable over acl}, $Y_1$ is $\operatorname{acl}(A_1z_1^2...z_1^k)$-definable. But then $$\dim(z_1^1/A_1z_1^2...z_1^k)\leq\dim(Y_1\cap Z_1^1)<\dim(Z_1^1),$$ contradicting the genericity of $z$ in $Z_1^1\times...\times Z_1^k$. (Note that the inequality $\dim(Y_1\cap Z_1^1)<\dim(Z_1^1)$ is because $Z_1^1$ is $O_1$-good, thus not $O_1$-common). 
    \end{proof}

    \begin{claim}\label{Cl: claim 2 inductive step}
        For generic $a_2\in X_2$, $O_2\restriction a_2$ sweeps $(Z_2^2,...,Z_2^k)$.
    \end{claim}
    \begin{proof}
        Let $a_1\in S$ be $(A_1,A_2)$-generic. By Claim \ref{Cl: preservation inductive step claim 1}, $O_1\restriction a_1$ sweeps $(Z_1^2,...,Z_1^k)$. So by the inductive hypothesis, $O_2\restriction a_2$ sweeps $(Z_2^2,...,Z_2^k)$. By automorphism invariance, the same holds for all $A_2$-generic elements of $X_2$.
    \end{proof}

    \begin{claim}\label{Cl: claim 3 inductive step}
    Each $O_i$ sweeps $(X_i,Z_i^2,...,Z_i^k)$.
    \end{claim}
    \begin{proof}
        Clear be the previous two claims. That is, let $a\in X_i$ be $A_i$-generic. We showed above that $O_i\restriction a$ sweeps $(Z_i^2,...,Z_i^k)$, so there are $Y\in O_i\restriction a$ and an $A_ia$-generic $b\in Z_i^2\times...\times Z_i^k$ with $b\in Y^{k-1}$. The $(a,b)$ is $A_i$-generic in $X_i\times Z_i^2\times...\times Z_i^k$ and contained in $Y\in O_i$, which gives the claim.
    \end{proof}

    By Claim \ref{Cl: claim 3 inductive step}, Lemma \ref{L: good preservation}, and Lemma \ref{L: relativized sweeping is unique}, each $O_i$ uniquely sweeps $(X_i,Z_i^2,...,Z_i^k)$. Thus, if $a\in X_i$ and $b\in Z_i^2\times...\times Z_i^k$ are independent generics over $A_i$, then $O_i(a,b)\in O_i$ is well-defined. Now recall:

    \begin{notation} For $A_i$-generic $b\in Z_i^2\times...\times Z_i^k$, $O_i\restriction b$ denotes the $A_ib$-orbit consisting of all members of $O_i$ of the form $O_i(a,b)$ for $A_ib$-generic $a\in X_i$.
    \end{notation}

    Since $O_i$ uniquely sweeps $(X_i,Z_i^2,...,Z_i^k)$, it is clear that for $A_i$-generic $b\in Z_i^2\times...\times Z_i^k$ the orbit $O_i\restriction b$ is uniquely 1-sweeping.

    \begin{claim}\label{Cl: claim 4 inductive step}
        Let $(b_1,b_2)\in(Z_1^2\times...\times Z_1^k,Z_2^2\times...\times Z_2^k)$ be $(A_1,A_2)$-generic. Then $O_1\restriction b_1$ and $O_2\restriction b_2$ are linked.
    \end{claim}
    \begin{proof}
        Recall that $S\subset X_1$ is a large set so that $O_1\restriction a$ and $O_2\restriction\varphi(a)$ are linked for all $a\in S$. By Corollary \ref{C: non-generics are small}(2), there is a large set $S'\subset S$ whose elements are all $(A_1b_1,A_2b_2)$-generic.

        We show that for all $(a_1,a_2)$ with $a_1\in S'$ and $\varphi(a_1)=a_2$ we have $(O_1\restriction b_1)(a_1)=(O_2\restriction b_2)(a_2)$. Since $S'$ is large, this gives the claim. So, let $\varphi(a_1)=a_2$ with $a_1\in S'$. Note for each $i$ that $$(O_i\restriction b_i)(a_i)=O_i(a_i,b_i)=(O_i\restriction a_i)(b_i).$$ So we can equivalently show that $\varphi$ sends $(O_1\restriction a_1)(b_1)$ to $(O_2\restriction a_2)(b_2)$. But this is automatic by the inductive hypothesis of Theorem \ref{T: preservation} (applied to the orbits $O_i\restriction a_i$, which sweep $(Z_i^2,...,Z_i^k)$).
    \end{proof}

    \begin{claim}\label{Cl: claim 5 inductive step}
    Let $(b_1,b_2)\in(Z_1^2\times...\times Z_1^k,Z_2^2\times...\times Z_2^k)$ be $(A_1,A_2)$-generic. Then each $O_i\restriction b_i$ sweeps $Z_i^1$.
    \end{claim}
    \begin{proof}
        By Claim \ref{Cl: claim 4 inductive step} and the base case (applied to $O_i\restriction b_i$), it suffices to show this for $i=1$. But the case $i=1$ follows since $O_1$ sweeps $(Z_1^1,...,Z_1^k)$. Namely, let $z\in Z_1^1$ be $A_1b_1$-generic. Then by Lemma \ref{L: generic sweeping witness}, there is $Y\in O_1$ so that $(z,b_1)\in Y^k$ is a generic witness to $O_1$ sweeping $(Z_1^1,...,Z_1^k)$. Let $a\in Y$ be generic over $A_1zb_1$. If $a$ is not $A_1b_1$-generic in $X_1$, then by Lemma \ref{L: codim 1 definable over acl}, $Y$ is $\acl(A_1b_1)$-definable; but then since $z$ is $A_1b_1$-generic in $Z_1^1$, we get $Z_1^1\subset Y$, contradicting that $Z_1^1$ is not $O_1$-common. So $a$ must be $A_1b_1$-generic in $X_1$. But then $(a,b_1)\in X_1\times Z_1^2\times...\times Z_1^k$ is $A_1$-generic, so that $Y=O_1(a,b_1)$, and thus $Y\in O_1\restriction b_1$. So there is a member of $O_1\restriction b_1$ containing the $A_1b_1$-generic element $z\in Z_1^1$, which gives the claim. 
    \end{proof} 

    Finally, we are ready to prove the two clauses of Theorem \ref{T: preservation}.

    \begin{claim}
        $O_2$ sweeps $(Z_2^1,...,Z_2^k)$.
    \end{claim}
    \begin{proof}
       Let $(b_1,b_2)\in(Z_1^2\times...\times Z_1^k,Z_2^2\times...\times Z_2^k)$ be $(A_1,A_2)$-generic, and let $z\in Z_2^1$ be $A_2b_2$-generic. By Claim \ref{Cl: claim 5 inductive step}, there is $Y\in O_2\restriction b_2$ with $z\in Y$. Then $(z,b_2)\in Z_2^1\times...\times Z_2^k$ is $A_2$-generic and contained in $Y\in O_2$, as desired.
    \end{proof}

    \begin{claim} Let $(x_1,x_2)\in(Z_1^1\times...\times Z_1^k,Z_2^1\times...\times Z_2^k)$ be $(A_1,A_2)$-generic. Then $\varphi(O_1(x_1))=O_2(x_2)$.
    \end{claim}
    \begin{proof}
        Write $x_i=(z_i,b_i)$ with $z_i\in Z_i^1$ and $b_i\in Z_i^2\times...\times Z_i^k$. Then $$\varphi(O_1(x_1))=\varphi((O_1\restriction b_1)(z_1))=(O_2\restriction b_2)(z_2)=O_2(x_2),$$ where the middle equality is by the inductive hypothesis.
    \end{proof}

    Finally, by the two most recent claims, the proof of Theorem \ref{T: preservation} is complete.
\end{proof}

\section{Existence of Linked Orbits}

In this section, we prove Theorem \ref{T: intro link existence}, which says that there \textit{must} be an orbit in $X_2$ linked to the generic hyperplane orbit in $X_1$. As with Theorem \ref{T: intro link preservation}, it is more convenient to prove a more general statement by induction. Recall that for a parameter set $A_1$, and an $A_1$-definable linear subspace $L\leq\mathbb P^n(K_1)$, the notation $H_{A_1}^L(X_1)$ denotes the orbit of hyperplane sections $X\cap H$ where $H$ runs over all $A_1$-generic hyperplanes containing $L$.

\begin{theorem}\label{T: existence of linked orbits}
    Let $O_1=H_{A_1}^L(X_1)$, where $L$ is a generic linear space of codimension $d\geq 2$ which is definable over the set $A_1$. Let $k:=d-1$, so that $O_1$ is uniquely $k$-sweeping. Then there are a set $A_2$ and a uniquely $k$-sweeping orbit $O_2$ over $A_2$, so that $O_1$ and $O_2$ are linked.
\end{theorem}

The rest of this section will occupy the proof of Theorem \ref{T: existence of linked orbits}. We work by induction on $k\geq 1$.

\begin{assumption} \textbf{Throughout this section, we fix $A_1,L,O_1,d,k$ as in Theorem \ref{T: existence of linked orbits}, and we moreover assume Theorem \ref{T: existence of linked orbits} holds whenever $k'<k$.}
\end{assumption}

\subsection{Pigeonhole Steps}

    We begin with a `pigeonhole' argument, using that the $K_i$ are uncountable but the set of orbits is always countable (Lemma \ref{F: omega stable}). Let $G$ be the set of generic points of $X_1^k$ over $A_1$. Recall that $G$ is large (see Corollary \ref{C: non-generics are small}).
    
    For any parameter set $B_2$ from $K_2$, we have a map $f_{B_2}:G\rightarrow\operatorname{Orb}(B_2)$, sending $a\in G$ to the $B_2$-orbit of $\varphi(O_1(a))$. By Fact \ref{F: omega stable}, the set $\operatorname{Orb}(B_2)$ is countable, and thus some element of $\operatorname{Orb}(B_2)$ has large preimage.
    
    Now out of all such sets $B_2$, and all such orbits over $B_2$ with large preimage under $f_{B_2}$, let us fix one of smallest dimension. That is, we fix $A_2$ and $O_2\in\operatorname{Orb}(A_2)$ so that (1) $f_{A_2}^{-1}(O_2)$ is large, and (2) $\dim(O_2)$ is minimized with respect to these requirements. Note that we may assume $O_2$ is irreducible: indeed, $O_2$ must decompose into finitely many irreducible orbits of the same dimension over $\operatorname{acl}(A_2)$, and one of these must also have a large preimage under $f_{\operatorname{acl}(A_2)}$. Thus, we assume throughout the proof that $O_2$ is irreducible. 

    We will show that $O_1$ and $O_2$ are linked. To start, we immediately have the following weaker approximations (Lemma \ref{L: O_2 k-sweeping} and \ref{L: weakly linked} below):
    
    \begin{lemma}\label{L: O_2 k-sweeping}
        $O_2$ is $k$-sweeping.
    \end{lemma}
    \begin{proof}
        By Lemma \ref{L: generics in large set} and the choice of $O_2$, there is an $(A_1,A_2)$-generic $a\in X_1^k$ with $f_{A_2}(a)=O_2$. So $\varphi(a)$ is $A_2$-generic in $X_2^k$ and is contained in a member of $O_2$, which means $O_2$ is $k$-sweeping.
    \end{proof}

   Below, given $B_i\supset A_i$, let us say that two members $Y_i,Z_i\in(O_i)_{B_i}$ are \textit{independent} if $\operatorname{Cb}(Y_i)$ and $\operatorname{Cb}(Z_i)$ are independent over $B_i$ (equivalently, if $Z_i\in(O_i)_{C_i}$ where $C_i$ is $B_i$ together with $\operatorname{Cb}(Y_i)$). Clearly, the independent pairs of $(O_i)_{B_i}$ themselves form an automorphism orbit of pairs over $B_i$.

    \begin{lemma}\label{L: weakly linked}
        Let $B_i\supset A_i$.
        \begin{enumerate}
            \item There are $Y_i\in(O_i)_{B_i}$ with $\varphi(Y_1)=Y_2$.
            \item In fact, there are independent $Y_i,Z_i\in O_{B_i}$ with $\varphi(Y_1)=Y_2$ and $\varphi(Z_1)=Z_2$.
        \end{enumerate}
    \end{lemma}
    \begin{proof} First we prove (1). Let $S$ be the (large by assumption) set of points $a\in G$ so that $f_{A_2}(a)=O_2$. By a similar argument to above, the map $f_{B_2}$ is constant on a large set $S'\subset S$. Let $O'_2\in\operatorname{Orb}(B_2)$ be the value attained by $f_{B_2}$ on $S'$. Note that $O_2'$ is a sub-orbit of $O_2$ over $B_2$. Then by the minimality of $\dim(O_2)$, it follows that $O_2'=(O_2)_{B_2}$.
    
    By Lemma \ref{L: generics in large set}, let $a\in S'$ be $B_1$-generic in $X_1^k$, and let $Y_1=O_1(a)$ and $Y_2=\varphi(Y_1)$. Since $a$ is $B_1$-generic, $Y_1\in(O_1)_{B_1}$. Moreover, by definition, $$Y_2\in f_{B_2}(a)=O_2'=(O_2)_{B_2},$$ as desired. Thus we have shown (1).

    For (2), simply apply (1) twice: first over $(B_1,B_2)$ to get $(Y_1,Y_2)$, and then over $(C_1,C_2)$ to get $(Z_1,Z_2)$, where (as above) $C_i$ is $B_i$ together with $\operatorname{Cb}(Y_i)$.
    \end{proof}

    \subsection{Outline of the Rest of the Argument}

    Our main goal will be to show that $O_2$ is uniquely $k$-sweeping. Let us check that this will be enough:

    \begin{lemma}\label{L: weakly linked implies linked}
        Suppose $O_2$ is uniquely $k$-sweeping. Then $O_1$ and $O_2$ are linked.
    \end{lemma}
    \begin{proof}
        Let $S\subset G$ be a large set on which $f_{A_2}$ is constant. Let $S'\subset S$ contain those elements of $S$ which are $(A_1,A_2)$-generic in $X_1^k$. By Corollary \ref{C: non-generics are small}, $S'$ is still large. We claim that for all $a\in S'$ we have $$\varphi(O_1(a))=O_2(\varphi(a)),$$ which shows that $O_1$ and $O_2$ are linked. Indeed, let $a\in S'$. Since $a\in S$, $f_{A_2}(a)=O_2$. Since $\varphi(a)$ is $A_2$-generic, and $O_2$ is uniquely $k$-sweeping, there is exactly one member of $O_2$ containing $\varphi(a)$. Since $f_{A_2}(a)=O_2$, it follows that $\varphi(O_1(a))$ is one member of $O_2$ containing $\varphi(a)$ -- and thus by uniqueness it must equal $O_2(\varphi(a))$.
    \end{proof}

    We move now toward proving that $O_2$ is uniquely $k$-sweeping. As a warmup, we start with the base case:

    \begin{lemma}\label{L: link existence base case}
        If $k=1$ then $O_2$ is uniquely 1-sweeping.
    \end{lemma}
    \begin{proof}
        By Lemma \ref{L: existence of coherent sets}, we may assume $(A_1,A_2)$ is coherent (since replacing with a larger coherent pair does not affect the relevant sweeping data). So we assume coherence throughout.
        
        Since $O_2$ is 1-sweeping (Lemma \ref{L: weakly linked}), Proposition \ref{P: generic witness to non sweeping} applies. In particular, assuming $O_2$ is not uniquely 1-sweeping, we can find distinct independent members $Y_2,Z_2\in O_2$, and an $A_2$-generic $a_2\in X_2$ with $a_2\in Y_2\cap Z_2$. By Lemma \ref{L: weakly linked}(2), after applying an automorphism if necessary, we may assume there are independent members $Y_1,Z_1\in O_1$ with $\varphi(Y_1)=Y_2$ and $\varphi(Z_1)=Z_2$. Let $a_1=\varphi^{-1}(a_2)$. By coherence, $a_1$ is $A_1$-generic in $X_1$; and moreover $Y_1$ and $Z_1$ are distinct and both contain $a_1$. This contradicts that $O_1$ is uniquely 1-sweeping.
    \end{proof}

   For $k\geq 2$, the proof that $O_2$ is uniquely $k$-sweeping is really a longer analog of the proof in Lemma \ref{L: link existence base case}. The only problem is that coherent sets are unavailable in higher arities, so we can't force the genericity of $a_1$. Instead, the proof will use a more elaborate geometric idea. Let us describe this idea informally.
   
    First consider the case $X_1=\mathbb P^2(K_1)$, $A_1=\emptyset$, and $O_1$ the orbit of (generic) straight lines (so $k=2$). We want to show that, generically, a curve in $O_2$ is determined by two points. This can be done by counting intersections. Namely, let $Y_1,Y_2,Z_1,Z_2$ be as in Lemma \ref{L: weakly linked}(2). Since $Y_1\cap Z_1$ is a single point (the intersection of two lines), so is $Y_2\cap Z_2$. Thus, a generic pair of curves from $O_2$ will intersect only once -- and this suggests (and implies, with a bit of work) that a generic pair of points in $X_2$ will indeed line on exactly one curve in $O_2$.

Now to generalize slightly, let $X_1$ be any two-dimensional quasi-projective variety, and $O_1=H_{A_1}^L(X_1)$ as in Theorem \ref{T: existence of linked orbits}, where we assume $L$ has codimension 3 (so again, $k=2$). By dimension considerations, $L$ is disjoint from $X_1$. Now, a generic intersection of curves in $O_1$ might not be a unique point. But something similar still holds: namely, a generic intersection consists of exactly one \textit{equivalence class} under the `linear span over $L$' relation $E_1$. Here, the equivalence classes are finite sets obtained by intersecting $X_1$ with a codimension 2 linear subspace of $\mathbb P^n(K_1)$ containing $L$ (in other words, a linear space of dimension 1 over $L$). The fact that these sets induce an equivalence relation follows since $L$ is disjoint from $X_1$.

Now in a similar fashion, we can show that since a generic intersection of two curves in $O_1$ contains exactly one $E_1$-class, a correspondingly `linked' generic intersection of curves in $O_2$ will contain exactly one $E_2:=\varphi(E_1)$-class. We would like to conclude the same holds for \textit{all} generic intersections of $O_2$ curves (not just those linked to $O_1$). However, the typical argument (via automorphisms) only works if $E_2$ is definable in $K_2$.

The trick: we can (at least approximately) define $E_2$ using the inductive hypothesis. Namely, fix two independent generics $a_1,b_1\in X_1$. The families $P_1=O_1(a_1),Q_1=O_1(b_1)$ are of smaller dimension, so satisfy the inductive hypothesis, and are thus linked to orbits $P_2,Q_2$ in $X_2$. Using the strong conditions guaranteed by Theorem \ref{T: preservation}, it follows that $P_2,Q_2$ are very close to the precise images of $P_1,Q_1$ under $\varphi$. For simplicity, let us completely identify them by assuming $P_2=\varphi(P_1)$ and $Q_2=\varphi(Q_1)$.

Now the family of all intersections of curves from $P_1$ and $Q_1$ precisely outputs the (generic) $E_1$-equivalence classes (i.e. a linear space $S$ can always be obtained as the intersection $S\langle a_1\rangle\cap S\langle b_1\rangle$). It follows that we can equivalently generate (most of) the $E_2$-classes by intersecting curves from $P_2$ and $Q_2$. Thus $E_2$ is (mostly) definable in $K_2$, as the set of such intersections. As before, it now follows that a curve in $O_2$ is determined by a single generic $E_2$-class, and thus two independent generics of $X_2$ (being non-$E_2$-equivalent) determine a unique generic $O_2$ curve as desired. This is a bit imprecise, but can be formalized as in Lemma \ref{L: link existence base case} using Proposition \ref{P: generic witness to non sweeping}.  

The argument presented below is a generalization of the sketch presented above. Namely, in the general case, we aim to show that the `linear dependence over $L$' relation on $X_1^k$ maps (approximately) to a definable relation on $X_2^k$ (which we might call `twisted linear dependence'). As sketched above, we do this by writing the linear span of a generic tuple in $X_1^{k-1}$ as an intersection of linear spaces from smaller orbits where the inductive hypothesis holds.

So by induction, we obtain a definable `twisted span' for tuples in $X_2^{k-1}$. We then use this to characterize what we call the `generic intersection format' of $O_2$, showing that it coincides with the `format' of the twisted span of a generic $(k-1)$-tuple (here the `format' of a variety is an ad hoc term for the list of dimensions of its irreducible components; this acts as a higher-dimensional analog of the cardinality of an $E_2$-class in the two-dimensional case). Finally, we show that this sets up an application of Proposition \ref{P: generic witness to non sweeping} to show $O_2$ is uniquely $k$-sweeping. 

\subsection{The Setup}

From now on, we assume $k\geq 2$. Let $I,J\leq\mathbb P^n(K)$ be two independent generic (over $A_1$) linear spaces of dimension one over $L$, each defined over $B_1\supset A_1$ -- so $L\leq I,J$, $\dim(I)=\dim(J)=\dim(L)+1$, and $I\cap J=L$. Let $P_1=H_{B_1}^I(X_1)$ and $Q_1=H_{B_1}^J(X_1)$ be the hyperplane families over $I$ and $J$, respectively. So each of $P_1,Q_1$ is uniquely $(k-1)$-sweeping, and thus the inductive hypothesis can be applied to them. Thus there are a set $B_2$ and uniquely $(k-1)$-sweeping orbits $P_2,Q_2$ over $B_2$ which are linked to $P_1,Q_1$ respectively (we can assume $P_2,Q_2$ are over the same set because being linked is parallelism invariant (Lemma \ref{L: linked invariant})).

    \subsection{Twisted Span}

    \begin{notation}
        Given a set $A\subset X_1$, we define the \textit{span of $A$}, denoted $\operatorname{Span}_1(A)$, to be $X\cap S$, where $S$ is the smallest linear subspace of $\mathbb P^n(K_1)$ containing $L\cup A$.
        
        Now given a set $B\subset X_2$, we define the \textit{twisted span of $B$}, denoted $\operatorname{Span}_2(B)$, to be $\varphi(\operatorname{Span}_1(\varphi^{-1}(B))$.
    \end{notation}

    \begin{lemma}[Definability of twisted span]\label{L: definability of span} Let $x_2\in X_2^{k-1}$ be $(B_1,B_2)$-generic.
    
    \begin{enumerate}
        \item $\operatorname{Span}_2(x_2)$ is $\operatorname{Gal}(K_2/B_2x_2)$-invariant, and therefore definable in $K_2$ over $B_2x_2$.
        \item If $y_2\in X_2^{k-1}$ is another $(B_1,B_2)$-generic point, and $\sigma\in\operatorname{Gal}(K_2/B_2)$ with $\sigma(x_2)=y_2$, then $$\sigma(\operatorname{Span}_2(x_2))=\operatorname{Span}_2(y_2).$$ \end{enumerate}
        \end{lemma}
    \begin{proof}
        Both clauses follow since for $(B_1,B_2)$-generic $z_2\in X_2^{k-1}$ we have $$\operatorname{Span}_2(z_2)=P_2(z_2)\cap Q_2(z_2).$$ In more details: let $x_1=\varphi^{-1}(x_2)$, so that $(x_1,x_2)$ is $(B_1,B_2)$-generic in $(X_1^{k-1},X_2^{k-1})$. By Theorem \ref{T: preservation}(2), we have $\varphi(P_1(x_1))=P_2(x_2)$ and $\varphi(Q_1(x_1))=Q_2(x_2)$ (the theorem applies because we are only using it in the case $Z_i=X_i$, which is $O_i$-good by Lemma \ref{L: good preservation}). But note that $$\operatorname{Span_1}(x_1)=P_1(x_1)\cap Q_1(x_1)$$ (as the intersection of two subspaces of dimension 1 over it), and thus $$\operatorname{Span}_2(x_2)=\varphi(\operatorname{Span}_1(x_1))=\varphi(P_1(x_1))\cap\varphi(Q_1(x_1))=P_2(x_2)\cap Q_2(x_2).$$ Since $P_2,Q_2$ are $B_2$-orbits, it is clear that $P_2(x_2)$ and $Q_2(x_2)$ are $B_2x_2$-invariant, and thus so is their intersection. This proves (1).
        
        Now for (2), since $y_2$ is also $(B_1,B_2)$-generic, we similarly have $\operatorname{Span}_2(y_2)=P_2(y_2)\cap Q_2(y_2)$. But since $P_2,Q_2$ are orbits over $B_2$, and $\sigma(x_2)=y_2$, it follows that also $\sigma(P_2(x_2))=P_2(y_2)$ and $\sigma(Q_2(x_2))=Q_2(y_2)$. Thus $$\sigma(\operatorname{Span}_2(x_2))=\sigma(P_2(x_2))\cap \sigma(Q_2(x_2))=P_2(y_2)\cap Q_2(y_2)=\operatorname{Span}_2(y_2).$$ 

    \end{proof}

    \begin{definition}
        Suppose $Y_2\in(O_2)_{B_2}$ with $c_2=\operatorname{Cb}(Y_2)$, and suppose $y_2\in Y_2^{j}$ for some $j$. Say that $Y_2$ \textit{generically contains $y_2$} if $y_2$ is $B_2c_2$-generic in $Y_2^j$.
    \end{definition}

    Using Lemma \ref{L: definability of span}, we get that members of $O_2$ are `almost $\operatorname{Span}_2$-closed':

    \begin{lemma}[$O_2$ respects $\operatorname{Span}_2$]\label{L: O_2 is affine closed}
        Let $y_2\in X_2^{k-1}$ be $(B_1,B_2)$-generic. If $Y_2\in(O_2)_{B_2}$ generically contains $y_2$, then $\operatorname{Span}_2(y_2)\subset Y_2$.
    \end{lemma}
    \begin{proof}
        By Lemma \ref{L: weakly linked}(1), there are $Z_i\in(O_i)_{B_i}$ with $\varphi(Z_1)=Z_2$. Let $d_i=\operatorname{Cb}(Z_i)$, and (by Lemma \ref{L: generics exist}) let $(z_1,z_2)\in(Z_1^{k-1},Z_2^{k-1})$ be generic over $(B_1d_1,B_2d_2)$. Since each $O_i$ is $k$-sweeping, $(z_1,z_2)$ is also $(B_1,B_2)$-generic in $(X_1^{k-1},X_2^{k-1})$.

        Since $Z_1$ is a hyperplane section, $\operatorname{Span}_1(z_1)\subset Z_1$. Thus $\operatorname{Span}_2(z_2)\subset Z_2$. Now since $Y_2,Z_2$ belong to the same $B_2$-orbit, and $y_2,z_2$ are generic tuples from them, we can find $\sigma\in\operatorname{Gal}(K_2/B_2)$ with $\sigma(Z_2)=Y_2$ and $\sigma(z_2)=y_2$. Since $z_2,y_2$ are $(B_1,B_2)$-generic, Lemma \ref{L: definability of span} moreover gives that $\sigma(\operatorname{Span}_2(z_2))=\operatorname{Span}_2(y_2)$. Thus $$\operatorname{Span}_2(y_2)=\sigma(\operatorname{Span}_2(z_2))\subset\sigma(Z_2)=Y_2.$$
    \end{proof}

    \subsection{Generic Intersection Format}

    \begin{definition}
        Let $Y\subset X_i$ be a closed subvariety. We define the \textit{format} of $Y$, denoted $\operatorname{Form}(Y)$, to be the multiset of dimensions of irreducible components of $Y$.
    \end{definition}

    Clearly, format is a homeomorphism invariant. Also, if $Y\subset Z$ are closed and $\operatorname{Form}(Y)=\operatorname{Form}(Z)$, then $Y=Z$.

    \begin{definition}
        The \textit{generic intersection format} of $O_i$, denoted $\operatorname{GIF}(O_i)$, is the format of $Y\cap Z$ whenever $Y,Z\in O_i$ are independent (this is well-defined by automorphism invariance).
    \end{definition}
    
    By Lemma \ref{L: definability of span}, the orbit of $\operatorname{Span}_2(x_2)$ under $\operatorname{Gal}(K_2/B_2)$ is constant for $(B_1,B_2)$-generic $x_2\in X^{k-1}$, and thus so is the format $\operatorname{Form}(\operatorname{Span}_2(x_2))$.

    \begin{notation} Let $e:=\operatorname{Form}(\operatorname{Span}_2(x_2))$ for any $(B_1,B_2)$-generic $x_2\in X_2^{k-1}$.
    \end{notation}


    The main point of this subsection is to show that $e$ is also equal to $\operatorname{GIF}(O_i)$ for each $i$:

    \begin{lemma}\label{L: degree of intersection}
    \begin{enumerate}
        \item $\operatorname{GIF}(O_1)=\operatorname{GIF}(O_2)$.
        \item $\operatorname{GIF}(O_1)=e$.
        \item In particular, $\operatorname{GIF}(O_2)=e$.
    \end{enumerate}
    \end{lemma}
    \begin{proof}
        \begin{enumerate}
            \item Let $Y_1,Y_2,Z_1,Z_2$ be as in Lemma \ref{L: weakly linked}(2) (over any parameter sets). Then $$\operatorname{GIF}(O_1)=\operatorname{Form}(Y_1\cap Z_1)=\operatorname{Form}(Y_2\cap Z_2)=\operatorname{GIF}(O_2).$$
            \item Let $(x_1,x_2)$ be $(B_1,B_2)$-generic in $(X_1^{k-1},X_2^{k-1})$, and let $y,z$ be independent generics in $X_1$ over $B_1x_1$. Let $Y=O_1(x_1y)$ and $Z=O_1(x_1z)$. It is easy to see that $Y,Z$ are independent members of $O_1$, and that $Y\cap Z=\operatorname{Span}_1(x_1)$. Thus:
            
            $$\operatorname{GIF}(O_1)=\operatorname{Form}(Y\cap Z)=\operatorname{Form}(\operatorname{Span}_1(x_1))=\operatorname{Form}(\operatorname{Span}_2(x_2))=e.$$
            \item By (1) and (2).
        \end{enumerate}
    \end{proof}

    \subsection{Concluding}

    We now finish the proof that $O_2$ is uniquely $k$-sweeping;
    
    \begin{proposition}
        $O_2$ is uniquely $k$-sweeping.
    \end{proposition}
    \begin{proof}
    Assume not. Recall that $O_2$ is already known to be $k$-sweeping (Lemma \ref{L: O_2 k-sweeping}). So we may use Proposition \ref{P: generic witness to non sweeping}. In particular, let us fix $Y,Z\in(O_2)_{B_2}$ with canonical bases $c,d$ respectively, and a tuple $a\in X_2^k$, so that (1)-(3) of Proposition \ref{P: generic witness to non sweeping} hold for $Y,Z,c,d,a$. Moreover, we write $a=(y,z)$ with $y\in X_2^{k-1}$ and $z\in X_2$.

    So $a\in X_2^k$ is $B_2$-generic. Then after applying an automorphism in $\operatorname{Gal}(K_2/B_2)$, we may replace $a$ with a $(B_1,B_2)$-generic element. Thus, moving forward we assume that $a$ is $(B_1,B_2)$-generic in $X_2^k$.

    We now finish the proof by finding a contradiction. First we use our previous work to observe, as described earlier in our informal sketch, that our generic $O_2$ intersection is determined by $\operatorname{Span}_2$ of a single point:

    \begin{claim}\label{Cl: intersection is span of point}
        $Y\cap Z=\operatorname{Span}_2(y)$.
    \end{claim}
    \begin{proof}
        By Lemma \ref{L: O_2 is affine closed}, $\operatorname{Span}_2(y)\subset Y\cap Z$. By Lemma \ref{L: degree of intersection}, $\operatorname{Form}(\operatorname{Span}_2(y))=\operatorname{Form}(Y\cap Z)=e$. Thus, $\operatorname{Span}_2(y)=Y\cap Z$.
    \end{proof}

    Finally, we have a contradiction. Indeed, by Lemma \ref{L: definability of span}, $\operatorname{Span}_2(y)=Y\cap Z$ is $B_2y$-definable, so that $\dim(z/By)\leq\dim(Y\cap Z)\leq\hat\dim(X_2)-2$. This contradicts that $a=(y,z)$ is $B_2$-generic in $X_2^k$.
    \end{proof}

    Finally, we have now completed the inductive step, and therefore the proof of Theorem \ref{T: existence of linked orbits}.

\section{Ind-Definability of Hyperplane Images}

In this section, we use the main theorems of the previous sections (i.e. Theorems \ref{T: intro link existence}, \ref{T: intro link preservation}, and \ref{T: intro hyperplane sweeping}) to conclude that the image of the family of (irreducible components of) hyperplane sections in $X_1$ under $\varphi$ is a countable union of definable families in $X_2$. Precisely, the following is the most convenient formulation of what we will prove:

\begin{notation}
     Let $CH_1\subset X_1^{n+1}$ (for `Co-Hyperplanarity') be the set of tuples $(x_1^1,...,x_{n+1}^`)$ so that $x_1^1,...,x_{n+1}^1$ belong to a common irreducible component of $X_1\cap H$ for some hyperplane $H$. Let $CH_2=\varphi(CH_1)\subset X_2^{n+1}$.
\end{notation}

It follows from the definability of irreducibility in algebraically closed fields (Will Johnson) that $CH_1$ is $K_1$-definable. Now we will show:

\begin{theorem}\label{t: ind-definable} $CH_2$ is the union of countably many $K_2$-definable sets.
\end{theorem}

The rest of this section comprises the proof of Theorem \ref{t: ind-definable}.

\subsection{Definability of Sweeping}

The main point of Theorem \ref{t: ind-definable} (which model theorists will surely recognize) is that being swept by an orbit is `definable in families'. Here we make this precise:

\begin{lemma}\label{L: definability of sweeping}
    Let $O\in\operatorname{Orb}(A)$ be a uniquely $k$-sweeping orbit in $X_2$, and let $\{Z_t:t\in T\}$ be a $K_2$-definable family of closed irreducible subvarieties of $X_2$. Then:
    
    \begin{enumerate}
        \item The set $\{t\in T:Z_t\textrm{ is }$O$\textrm{-good}\}$ is $K_2$-definable.
        \item Assume that each $Z_t$ is $O$-good. Then for all $k$, the set $\{t\in T:O\textrm{ }$k$\textrm{-sweeps }Z_t\}$ is $K_2$-definable.
    \end{enumerate}
\end{lemma}
\begin{proof}
    Write $O$ as the generic members of a family $\{Y_u:u\in U\}$, where $U$ is an irreducible variety of dimension $\dim(O)$ and each $Y_u\subset X$ is a closed irreducible subvariety of codimension 1 (this is exactly like Lemma \ref{L: relativized sweeping is unique}). So $\dim(U)=\dim(O)=k$.
    
    \begin{enumerate}
        \item Simply note that $Z_t$ is $O$-good if and only if (a) $Z_t$ is not contained in the singular locus of $X_2$ and (b) for generic $u\in U$, $Z_t$ is not contained in $Y_u$; this is a definable condition in $t$.

        \item Here is the argument is a bit more involved. The key point is:
        \begin{claim}
            For each $t\in T$, the following are equivalent:
            \begin{enumerate}
                \item[(a)] $O$ $k$-sweeps $Z_t$. 
                \item[(b)] For generic $a\in Z_t^k$, there is $u\in U$ so that $a\in Y_u^k$ and $\dim(Y_u\cap Z_t)<\dim(Z_t)$.
            \end{enumerate}
        \end{claim}
        \begin{proof}
        (a)$\rightarrow$(b): any generic witness to $O$ $k$-sweeping $Z_t$ is an instance of (b). (That is, $k$-sweeping gives some $Y_u$ containing $a$, and $O$-goodness gives that in addition $\dim(Y_u\cap Z_t)<\dim(Z_t)$).

        (b)$\rightarrow$(a): Without loss of generality assume $X$ and the families $\{Z_t\}$ and $\{Y_u\}$ are $\emptyset$-definable. Let $a\in Z_t^k$ be $t$-generic, and let $u\in U$ be as in (b). If $u$ is $t$-generic, then $Y_u\in O_t$, and this is enough to prove $k$-sweeping.

        Now we show $u$ is $t$-generic by a dimension computation: we have $$\dim(ua/t)=\dim(u/t)+k\cdot\dim(Y_u\cap Z_t)\leq\dim(u/t)+k\cdot(\dim(Z_t)-1),$$ while $\dim(a/t)=k\cdot\dim(Z_t)$. So $\dim(u/t)\geq k=\dim(U)$, which shows $u$ is generic and thus proves the claim.
        \end{proof}
        Finally, we have now proven (2) of the lemma, because condition (b) of the claim is a definable condition in $t$.
        \end{enumerate}
\end{proof}

\subsection{Countably many families}

Now we use Lemma \ref{L: definability of sweeping} to write the family of hyperplane images as the union of countably many definable families (i.e. we prove Corollary \ref{C: intro ind definable}). This is the culmination of Theorems \ref{T: intro link existence}, \ref{T: intro link preservation}, and \ref{T: intro hyperplane sweeping}:

\begin{proposition}
     Let $\mathcal H$ be the collection of irreducible components of hyperplane sections in $X_1$. Then $\varphi(\mathcal H)$ is the union of countably many definable families in $K_2$.
\end{proposition}\label{P: ctbly many families}
\begin{proof}
  Let $O_1$ be the uniquely $n$-sweeping orbit of generic hyperplane sections in $X_1$ (i.e. $O_1=H_\emptyset^{\emptyset}(X_1)$). By Theorem \ref{T: existence of linked orbits}, there is a uniquely $n$-sweeping orbit $O_2$ in $X_2$ which is linked to $O_1$. 
  
  Let $\mathcal D$ be the collection of closed irreducible subvarieties of $X_2$ which are $O_2$-good and not $k$-swept by $O_2$. 
  \begin{claim}
    $\mathcal D$ is the union of countably many $K_2$-definable families.
    \end{claim}
    \begin{proof} There are countably many $\emptyset$-definable families $\mathcal F$ of subsets of $X_2$. For each such $\mathcal F$, form a family $\mathcal D_\mathcal F$ consisting of those members of $\mathcal F$ which are closed, irreducible, have codimension 1 in $X_2$, are $O_2$-good, and are not $k$-swept by $O_2$. By Lemma \ref{L: definability of sweeping}, each $\mathcal D_{\mathcal F}$ is a definable family; and moreover $\mathcal D=\bigcup_{\mathcal F}\mathcal D_{\mathcal F}$.
    \end{proof}
    
    Now let $Y_1\subset X_1$ be a closed irreducible codimension 1 subvariety, and let $Y_2=\varphi(Y_1)$. Then if each $Y_i$ is $O_i$-good, we have the chain of equivalences:

  $$Y_2\in\mathcal\varphi(\mathcal H)\iff Y_1\in\mathcal H\iff O_1\textrm{ does not $n$-sweep }Y_1$$ $$\iff O_2\textrm{ does not $n$-sweep }Y_2\iff Y_2\in\mathcal D.$$

  Indeed, the first equivalence above is by definition; the second is by Theorem \ref{L: hyperplane sweeping}; the fourth is Theorem \ref{T: preservation}; and the fifth is by definition of $\mathcal D$.

  So we are \textit{almost} done, because $\mathcal D$ is a countable union of definable families which \textit{almost} equals $\varphi(\mathcal H)$. To finish, we need to account for pairs $(Y_1,Y_2)$ where some $Y_i$ is not $O_i$-good. But Lemma \ref{L: good cofinite}, there are only finitely such pairs, so we can finish by making finitely many automatically definable edits.
  
  Precisely, let $(Y_1^1,Y_2^1),...,(Y_1^j,Y_2^j)$ be an enumeration of the finitely many pairs $(Y_1,Y_2)$ where:
  \begin{itemize}
      \item $Y_i$ is a closed, irreducible, codimension 1 subvariety of $X_i$, and $\varphi(Y_1)=Y_2$.
      \item Either $Y_1$ is not $O_1$-good or $Y_2$ is not $O_2$-good.
      \item $Y_2\in\mathcal\varphi(\mathcal H)$.
  \end{itemize}
    Then combining $\mathcal D$ with the additional finite family $\{Y_2^1,...,Y_2^j\}$ gives a countable union of definable families exactly enumerating $\varphi(\mathcal H)$. This proves Proposition \ref{P: ctbly many families}.
\end{proof}
\subsection{Concluding}

We now prove Theorem \ref{t: ind-definable}:

\begin{proof}[Proof of Theorem \ref{t: ind-definable}] Using Proposition \ref{P: ctbly many families}, let $\mathcal Y_2^1,...,\mathcal Y_2^j,...$ be countably many definable families in $K_2$ whose members comprise exactly the $\varphi$-images of irreducible components of hyperplane sections in $X_1$. Let $CH_2^j$ be the set of $(x_2^1,...,x_x^{n+1})$ belonging to a common member of $\mathcal Y_2^j$. Then each $CH_2^j$ is definable, and $CH_2$ is the union of the $CH_2^j$. 
\end{proof}

\section{Definable Shapes and the Definability of Hyperplane Images}

\begin{convention}
    Throughout this section, we use the term \textit{ind-definable set} to describe a countable union of definable sets (this is standard terminology in model theory).
\end{convention} 

Theorem \ref{t: ind-definable} says that $CH_2$ is ind-definable in $K_2$. The final step of Part 1 is now to upgrade Theorem \ref{t: ind-definable} to say that $CH_2$ is in fact \textit{definable}. We do this by showing (Theorem \ref{T: intro definably shaped} = Corollary \ref{C: ind-definable implies definable}) that out of the collection of all ind-definable sets, we can recognize the definable ones using purely topological means. This is easy to see inside a single curve: suppose $C$ is an irreducible curve over an uncountable algebraically closed field, and $Y\subset C$ is ind-definable. Then $Y$ is either cofinite or countable, and is definable if and only if it is either cofinite or finite.

In general, we want to do something similar, showing that definability is characterized within ind-definability by a sufficiently strong collection of homeomorphism-invariant finiteness conditions. Such conditions will axiomatize what we call the `definably shaped' sets. Then our results will really show that an ind-definable set is definable if and only if it is definably shaped.

\subsection{The setup}

Throughout this section, we work uniformly with all quasi-projective varieties $X$ over uncountable algebraically closed fields $K$. We will inductively define the following for each $n\geq 1$ (uniformly for all $K$ and $X$): 

\begin{itemize}
    \item A collection $\mathcal{DS}_n(X)$ of `definably shaped' subsets of $X^n$.
    \item A partial order $(\mathcal{DS}_n,<)$ of `definable $n$-shapes' (independent of $K$ and $X$).
    \item A map $\operatorname{Sh}:\mathcal{DS}_n(X)\rightarrow\mathcal{DS}_n$, assigning each definably shaped set to its `shape'.
\end{itemize}

These notions will inductively satisfy the following requirements:

\begin{enumerate}
    \item \textbf{(Definability)} If $D\subset X^n$ is definable, then $D\in\mathcal{DS}_n(X)$. If $\{D_t:t\in T\}$ is an $A$-definable family of subsets of $X^n$, then $\{\operatorname{Sh}(D_t):t\in T\}$ is finite and $\{t\in T:\operatorname{Sh}(D_t)=s\}$ is $A$-definable for each $s\in\mathcal{DS}_n$.
    \item \textbf{(Topological Invariance)} Let $\varphi:X_1\rightarrow X_2$ be a homeomorphism, and $D\subset X_1^n$. If $D\in\mathcal{DS}_n(X_1)$, then $\varphi(D)\in\mathcal{DS}_n(X_2)$ and $\operatorname{Sh}(\varphi(D))=\operatorname{Sh}(D)$ (thus the same holds with $X_1$ and $X_2$ flipped as well).
    \item \textbf{(Order Invariance)} Let $D,E\in\mathcal{DS}_n(X)$ with $D\subset E$. Then $\operatorname{Sh}(D)\leq\operatorname{Sh}(E)$, with equality if and only if $D=E$.
    \item \textbf{(Chain Condition)} Let $D_1\subset D_2\subset D_3...$ with each $D_i\in\mathcal{DS}_n(X)$. If $D:=\bigcup_iD_i\in\mathcal{DS}_n(X)$, then $D=D_i$ for some $i$. 
\end{enumerate}

\begin{remark}\label{R: finitely many shapes automatic}
    In proving the definability condition above, we may ignore the requirement that $\{\operatorname{Sh}(D_t):t\in T\}$ is finite, since it follows from the other clauses. Indeed, if we show that each $T_s=\{t\in T:\operatorname{Sh}(D_t)=s\}$ is $A$-definable, then $\{T_s:s\in S\}$ forms a partition of $X$ into $A$-definable sets, and by compactness, this must be a finite partition.
\end{remark}

As immediate corollaries, we will obtain:

\begin{corollary}\label{C: ind-definable and shaped implies definable}
    Let $D\subset X^n$ be ind-definable. Then $D$ is definable if and only if it is definably shaped.
\end{corollary}
\begin{proof}
    The left to right direction is the definability condition (1) above. Now suppose $D$ is definably shaped. Write $D=\bigcup E_i$ with each $E_i$ definable. Replacing each $E_i$ with $\bigcup_{j\leq i}E_j$, we may assume the $E_i$ form a chain, i.e. $E_1\subset E_2\subset...$. By definability, each $E_i\in\mathcal{DS}_n(X)$. Moreover, we have $\bigcup_iE_i=D\in\mathcal{DS}_n(X)$. So by the chain condition, it follows that $D=E_i$ for some $i$, and thus $D$ is definable since $E_i$ is. 
\end{proof}

\begin{corollary}\label{C: ind-definable implies definable}
    Let $\varphi:X_1\rightarrow X_2$ be a homeomorphism of quasi-projective varieties $X_i$ over uncountable algebraically closed fields $K_i$. If $D_1\subset X_1^n$ is definable, and $D_2:=\varphi(D_1)$ is ind-definable, then $D_2$ is definable.
\end{corollary}
\begin{proof}
    By definability, $D_1$ is definably shaped. So by topological invariance, $D_2$ is definably shaped. Now apply Corollary \ref{C: ind-definable and shaped implies definable}.
\end{proof}

\begin{remark}
    The proofs of the above corollaries did not directly use the order invariance property. Instead, order invariance will be the most crucial property used in the proof of the chain condition. 
\end{remark}

The rest of this section, then, will be a construction of $\mathcal{DS}_n$, $\mathcal{DS}_n(X)$, and $\operatorname{Sh
}$ which satisfy (1)-(4) above.

\textbf{For the rest of this section, fix $X$, a quasi-projective variety over the uncountable algebraically closed field $K$.}

\subsection{1-Shapes}

We begin with the case of definable subsets of $X$. First we give an auxilliary notion. Throughout, we use closure notation $D\mapsto\overline D$ to denote the relative closure in $X$.

\begin{definition}
    Let $D\subset X$ be definable. We let $D^P$ (the \textit{pure part} of $D$) be the union of all $\dim(D)$-dimensional irreducible components of $\overline D$.
\end{definition}

So $D^P$ is not necessarily equal to $D$, but it is \textit{almost equal} in the sense that $\dim(D^P\Delta D)<\dim(D)$.

\begin{definition}\label{D: 1 shape} We define $\mathcal{DS}_1$ and $\mathcal{DS}_1(X)$ as follows:
    \begin{enumerate}
        \item A \textit{definable 1-shape} is a polynomial in one variable with integer coefficients. That is, $\mathcal{DS}_1=\mathbb Z[x]$.
        \item We order $\mathbb Z[x]$ by declaring that $x$ is infinite (i.e. $x>n$ for all $n$; note that this determines a unique ordered ring structure on $\mathbb Z[x]$).
        \item We let $\mathcal{DS}_1(X)$ be the collection of all definable subsets of $X$.
        \item For $X\in\mathcal{DS}_1(X)$, we define $\operatorname{Sh}(X)\in\mathbb Z[x]$ by induction on $\dim(X)$ as follows:
        \begin{enumerate}
            \item If $D=D^P$, then $$\operatorname{Sh}(D)=m\cdot x^{\dim(D)},$$ where $m$ is the number of irreducible components of $D$.
            \item For general $D\in\mathcal{DS}_1(X)$, we set $$\operatorname{Sh}(D)=\operatorname{Sh}(D^P)+\operatorname{Sh}(D-D^P)-\operatorname{Sh}(D^P-D).$$ (This is a definition by induction, since $D-D^P$ and $D^P-D$ have smaller dimension than $D$).
        \end{enumerate}
    \end{enumerate}
\end{definition}

\textbf{Note, in particular, that $\mathcal{DS}_1$ is a \textit{linear} order.} (The other $\mathcal{DS}_n$ will not be linear, but we will use the linearity of $\mathcal{DS}_1$ to construct them). Let us now check that all of the above data satisfies (1)-(4):

\begin{lemma}
    Properties (1)-(4) hold for $n=1$.
\end{lemma}

\begin{proof}
    \begin{itemize}
    \item Definability: by definition, every definable subset of $X$ belongs to $\mathcal{DS}_1(X)$. Now let $\{D_t:t\in T\}$ be an $A$-definable family of subsets of $X$, and set $d=\max\{\dim(D_t):t\in T\}$. It follows by induction on $d$ that only finitely many shapes occur among the $D_t$, and each happens along an $A$-definable subset of $T$. To see this, one repeatedly uses that dimension, closure, and the formation of irreducible components are definable in families in algebraically closed fields (see Johnson's appendix to \cite{JohIrr}). This gives that the families $\{D_t^P\}$, $\{D_t-D_t^P\}$, and $\{D_t^P-D_t\}$ are $A$-definable, and by induction, each satisfies definability (for the first this is definability of irreducible components, and for the second and third it is because those families have a smaller value of $d$). Then one can express the shape map $t\mapsto\operatorname{Sh}(D_t)$ $A$-definably in terms of the shape maps of these three simpler families.
    \item Topological Invariance: everything in Definition \ref{D: 1 shape} is purely topological, so is clearly homeomorphism invariant.
    \item Order invariance: this is again a straightforward exercise. Note by the inductive construction of 1-shapes that $\operatorname{Sh}(X)$ is always a polynomial of degree $\dim(X)$ whose leading term is $\operatorname{Sh}(D^P)$ (in particular, the leading term is always non-negative). Now suppose $D\subset E\subset X$ are definable. If $\dim(D)<\dim(E)$ then $\deg(\operatorname{Sh}(D))<\deg(\operatorname{Sh}(E))$, so as both have non-negative leading terms, the choice of order gives $\operatorname{Sh}(D)<\operatorname{Sh}(E)$. Thus in this case, everything is clear.
    
    Now suppose $\dim(D)=\dim(E)$. Then $D^P\subset E^P$, so every component of $D^P$ is a component of $E^P$. If $E^p$ has more components than $D^P$, then $\operatorname{Sh}(E)$ again has a strictly larger leading term, so $\operatorname{Sh}(D)<\operatorname{Sh}(E)$ and we are again done.
    
    Now assume $D^P=E^P$. Then since $D\subset E$ we have $D-D^P\subset E-E^P$ and $D^P-D\supset E^P-E$, so that by induction we have $\operatorname{Sh}(D-D^P)\leq\operatorname{Sh}(E-E^P)$ and $\operatorname{Sh}(D^P-D)\geq\operatorname{Sh}(E^P-E)$. Adding everything up now gives $\operatorname{Sh}(D)\leq\operatorname{Sh}(E)$. Moreover, if equality holds then we must have $\operatorname{Sh}(D-D^P)=\operatorname{Sh}(E-E^P)$, so by induction $D-D^P=E-E^P$; and by a similar argument, $D^P-D=E^P-E$. These together with $D^P=E^P$ imply $D=E$.
    \item This follows by the compactness theorem (since $K$ is uncountable and therefore saturated). Indeed, if $D\neq D_i$ for any $i$ then there is a consistent partial type $p(x)$ expressing that $x$ belongs to $D$ but not to any $D_i$. By saturation, $p$ has a realization, contradicting that $D=\bigcup_iD_i$. 
\end{itemize}
\end{proof}

\subsection{Higher Arity Shapes} Now we assume we have defined $\mathcal{DS}_n$, $\mathcal{DS}_n(X)$, and $\operatorname{Sh}:\mathcal{DS}_n(X)\rightarrow\mathcal{DS}_n$ for some fixed $n\geq 1$ so that properties (1)-(4) hold. We proceed to do the same for $n+1$.

\begin{definition} We define $\mathcal{DS}_{n+1}$:
\begin{enumerate}
    \item A \textit{definable $(n+1)$-shape} is a function $f:\mathcal{DS}_1\rightarrow\mathcal{DS}_n$. The set of such maps is denoted $\mathcal{DS}_{n+1}$.
    \item We turn $\mathcal{DS}_{n+1}$ into a partial order in the obvious way: $f\leq g$ if for all $s\in\mathcal{DS}_n$ we have $f(s)\leq g(s)$.
\end{enumerate}
\end{definition}

\begin{definition}\label{D: inductive shaped} Let $D\subset X^{n+1}$, and view $X^{n+1}$ as $X^n\times X$. For $a\in X^n$, write $D_a:=\{b\in X:(a,b)\in D\}$.
    \begin{enumerate}
        \item We say that $D$ is \textit{definably shaped} if the following hold:
        \begin{enumerate}
            \item $D_a\in\mathcal{DS}_1(X)$ for all $a\in X^n$.
            \item The set $\{\operatorname{Sh}(D_a):a\in X^n\}$ is finite.
            \item For each $s\in\mathcal{DS}_1$ we have $$D_{\geq s}:=\{a\in X^n:\operatorname{Sh}(D_a)\geq s\}\in\mathcal{DS}_n(X).$$
            \end{enumerate}
        \item Let $\mathcal{DS}_{n+1}(X)$ be the set of definably shaped subsets of $X^{n+1}$.
        \item If $D\in\mathcal{DS}_{n+1}(X)$, then define $\operatorname{Sh}(D)\in\mathcal{DS}_{n+1}$ as the map sending each $s\in\mathcal{DS}_1$ to $\operatorname{Sh}(D_{\geq s})\in\mathcal{DS}_n$.
\end{enumerate}

\end{definition}

From the definition, it is clear that the topological invariance property (2) still holds for $n+1$ (i.e. everything defined above is topological). The rest of this section will prove properties (1), (3), and (4) for $n+1$. The proofs are a bit more complicated, so we give them separate subsections.

\subsection{Definability}

\begin{lemma}\label{L: definably shaped definable case inductive 1}
    Let $D\subset X^{n+1}$ be definable. Then $D\in\mathcal{DS}_{n+1}(X)$.
\end{lemma}
\begin{proof}
    We check the three parts of the definition:
    \begin{enumerate}
        \item[(a)] Since $D$ is definable, each fiber $D_a\subset X$ is definable. So by the inductive hypothesis, each $D_a\in\mathcal{DS}_n(X)$.
        \item[(b)] The family of fibers $\{D_a:a\in X^n\}$ is a definable family of subsets of $X$, so by the base case, the set $\{\operatorname{Sh}(D_a):a\in X^n\}$ is finite.
        \item[(c)] Similar to (b). Namely, again using the family of fibers $\{D_a:a\in X^n\}$, the base case gives that each $D_s=\{a:\operatorname{Sh}(D_a)=s\}$ is definable. But by (b), each $D_{\geq s}$ is just the union of finitely many $D_s$, so is also definable. Thus each $D_{\geq s}$ belongs to $\mathcal{DS}_n(X)$ by the inductive hypothesis.
        \end{enumerate}
\end{proof}

\begin{lemma}\label{L: definably shaped definable inductive 2}
    Let $\{D_t:t\in T\}$ be an $A$-definable family of subsets of $X^{n+1}$. Then:
    \begin{enumerate}
        \item $\{\operatorname{Sh}(D_t):t\in T\}$ is finite.
        \item $\{t\in T:\operatorname{Sh}(D_t)=f\}$ is $A$-definable for all $f\in\mathcal{DS}_{n+1}$.
    \end{enumerate}
\end{lemma}
\begin{proof}
    As noted in Remark \ref{R: finitely many shapes automatic}, it is enough to prove (2). So fix $f\in\mathcal{DS}_{n+1}$, and let $T_f=\{t\in T:\operatorname{Sh}(D_t)=f\}$; we will show that $T_f$ is $A$-definable. First, it is possible that there are no definable sets at all of shape $f$. In this case, $T_f=\emptyset$, which is $A$-definable. So for the rest of the proof, we assume $f$ is realized by at least one definable set, and fix such a set $Y\subset X^{n+1}$ (note that we do not know whether $Y$ is definable over $A$, but we will not need this).
    
    First, the family of fibers $\{(D_t)_a:t\in T,a\in X^n\}$ is an $A$-definable family of subsets of $X$. By the base case, it follows that only finitely many 1-shapes occur among the $(D_t)_a$. Similarly, by Lemma \ref{L: definably shaped definable case inductive 1}, only finitely many 1-shapes occur among the fibers $Y_a$ for $a\in X^n$. Let us therefore fix a finite set $\mathcal S\subset\mathcal{DS}_1$ containing $\operatorname{Sh}((D_t)_a)$ and $\operatorname{Sh}(Y_a)$ for all $t\in T$ and $a\in X^n$.
    
    For each $s\in\mathcal S$, the base case gives that $$(T\times X^n)_s=\{(t,a)\in T\times X^n:\operatorname{Sh}((D_t)_a)=s\}$$ is $A$-definable. It follows that $$(T\times M^n)_{\geq s}=\{(t,a):\operatorname{Sh}((D_t)_a)\geq s\}$$ is also $A$-definable for $s\in\mathcal S$, because $(T\times X^n)_{\geq s}$ is just the union of $(T\times M^n)_{s'}$ for the finitely many $s'\in\mathcal S$ with $s'\geq s$.
    
    Recall (see Definition \ref{D: inductive shaped}) that given $s\in\mathcal{DS}_1$ and a definable set $D\subset X^{n+1}$ we set $D_{\geq s}=\{a\in M^n:\operatorname{Sh}(D_a)\geq s\}$. Now for a given $s\in\mathcal{DS}_1$, the previous paragraph shows that the family $\{(D_t)_{\geq s}:t\in T\}$ is an $A$-definable family of subsets of $X^n$ (because $(D_t)_{\geq s}$ is just the fiber over $t$ in $(T\times X^n)_{\geq s}$). So by the inductive hypothesis, for any $s\in\mathcal S$ and $g\in\mathcal{DS}_n$, the set $$T_{s,g}=\{t\in T:\operatorname{Sh}((D_t)_{\geq s})=g\}$$ is $A$-definable.
    
    Now recalling that our given $(n+1)$-shape $f$ is a function from $\mathcal{DS}_1$ to $\mathcal{DS}_n$, we let $T_{\mathcal S,f}$ be the intersection of $T_{s,f(s)}$ over all $s\in\mathcal S$. By above, each $T_{s,f(s)}$ is $A$-definable; and since $\mathcal S$ is finite, so is $T_{\mathcal S,f}$.

    Finally, we claim that $T_{\mathcal S,f}=T_f$, which will finish the proof. It is clear that $T_f\subset T_{\mathcal S,f}$ (i.e. $T_{\mathcal S,f}$ is defined by a subset of the requirements needed to have shape $f$).

    Conversely, suppose $t\in T_{\mathcal S,f}$. We need to show that $\operatorname{Sh}(D_t)=f$. For this, fix any $s\in\mathcal{DS}_1$. We want to show that $$\operatorname{Sh}((D_t)_{\geq s})=f(s)=\operatorname{Sh}(Y_{\geq s}).$$

    If $s>s'$ for all $s'\in\mathcal S$, then by definition of $\mathcal S$, both sides of the above are empty, and we are done. Otherwise, since $\mathcal{DS}_1$ is a linear order (and $\mathcal S$ is finite), there is a smallest element of $S$ greater than or equal to $s$. Call this element $s_0\in\mathcal S$. Then by definition of $\mathcal S$, we have $(D_t)_{\geq s}=(D_t)_{\geq s_0}$ and $Y_{\geq s}=Y_{\geq s_0}$. On the other hand, since $t\in T_{\mathcal S,f}$, we in particular have $t\in T_{s_0,f(s_0)}$, and thus $\operatorname{Sh}((D_t)_{\geq s_0})=f(s_0)=\operatorname{Sh}(Y_{\geq s_0})$. So putting everything together gives $$\operatorname{Sh}((D_t)_{\geq s})=\operatorname{Sh}((D_t)_{\geq s_0})=f(s_0)=\operatorname{Sh}(Y_{\geq s_0})=\operatorname{Sh}(Y_{\geq s}),$$ and we are done.
\end{proof}

\subsection{Order Invariance}

\begin{lemma}
    Let $D,E\in\mathcal{DS}_{n+1}(X)$ with $D\subset E$. Then $\operatorname{Sh}(D)\leq\operatorname{Sh}(E)$, with equality if and only if $D=E$.
\end{lemma}
\begin{proof}
    To show $\operatorname{Sh}(D)\leq\operatorname{Sh}(E)$, let $s\in\mathcal{DS}_1$. We need to show that $\operatorname{Sh}(D_{\geq s})\leq\operatorname{Sh}(E_{\geq s})$. In fact we will show that $D_{\geq s}\subset E_{\geq s}$, from which we conclude by induction. Now to see $D_{\geq s}\subset E_{\geq s}$, fix $a\in D_{\geq s}$. So $\operatorname{Sh}(D_a)\geq s$. Now by assumption $D_a\subset E_a$, so by the base case $\operatorname{Sh}(E_a)\geq\operatorname{Sh}(D_a)\geq s$, and thus $a\in E_{\geq s}$ as desired.

    Now assume $\operatorname{Sh}(D)=\operatorname{Sh}(E)$. Let $(a,b)\in E$, and let $s=\operatorname{Sh}(E_a)$. As in the previous paragraph, we have $D_{\geq s}\subset E_{\geq s}$. Moreover, by the construction of $\operatorname{Sh}$, we know that $\operatorname{Sh}(D_{\geq s})=\operatorname{Sh}(E_{\geq s})$. It follows by the inductive hypothesis that $D_{\geq s}=E_{\geq s}$. But trivially $a\in E_{\geq s}$, so also $a\in D_{\geq s}$, and thus $\operatorname{Sh}(D_a)\geq\operatorname{Sh}(E_a)$. But cleary $D_a\subset E_a$, so by the base case we conclude $D_a=E_a$, and thus $(a,b)\in D$ as desired.
\end{proof}

\subsection{Chain Condition}

\begin{lemma}
    Let $D_1\subset D_2\subset...$ so that each $D_i\in\mathcal{DS}_{n+1}(X)$. If $D:=\cup_iD_i\in\mathcal{DS}_{n+1}(X)$, then $D=D_i$ for some $i$.
\end{lemma}
\begin{proof}
    Consider any $s\in\mathcal{DS}_1$. Note that we have a chain $(D_1)_{\geq s}\subset(D_2)_{\geq s}\subset...$, (this is by order invariance in the base case).
    \begin{claim}
        For any $s\in\mathcal{DS}_1$, we have the following:
        \begin{enumerate}
            \item $$\bigcup_i(D_i)_{\geq s}=D_{\geq s}.$$
            \item In particular, $(D_i)_{\geq s}=D_{\geq s}$ for some $i$.
        \end{enumerate}
    \end{claim}
    \begin{proof}
        By assumption each $(D_i)_{\geq s}$ and $D_{\geq s}$ is definably shaped; thus (1) implies (2) by the inductive hypothesis. Below we prove (1).
        
        Clearly each $(D_i)_{\geq s}\subset D_{\geq s}$ (by order invariance in the base case). Now suppose $a\in D_{\geq s}$. Note that $(D_1)_a\subset(D_2)_a\subset...$ forms a chain whose union is $D_a$. By definition all of these sets belong to $\mathcal{DS}_1(X)$, so by the base case we have $(D_i)_a=D_a$ for some $i$. Then $a\in(D_i)_{\geq s}$.
\end{proof} 
Now by assumption there are only finitely many shapes of the form $s(a):=\operatorname{Sh}(D_a)$ for $a\in X^n$. Let $i$ be so large that $(D_i)_{\geq s(a)}=D_{\geq s(a)}$ for all $a\in X^n$. We claim that $D_i=D$. Indeed, suppose $(a,b)\in D$. By the choice of $i$, $(D_i)_{\geq s(a)}=D_{\geq s(a)}$. But trivially $a\in D_{\geq s(a)}$, so also $a\in(D_i)_{\geq s(a)}$. In other words, we have $\operatorname{Sh}((D_i)_a)\geq\operatorname{Sh}(D_a)$. By order invariance in the base case, this implies $(D_i)_a=D_a$, so in particular $(a,b)\in D_i$. 
\end{proof}

\subsection{Conclusion}

Now that we have constructed all of the relevant data, we are ready to upgrade Theorem \ref{t: ind-definable}. \textbf{We now return to the setting of Convention \ref{Convention homeomorphism}}. As before, we let $CH_1$ be the (definable) set of $(n+1)$-tuples from $X_1$ belonging to a common irreducible component of some hyperplane section in $X_1$, and we let $CH_2=\varphi(CH_1)$. Then we finally have the main theorem of Part 1:

\begin{theorem}\label{T: definability of hyperplane images}
$CH_2$ is definable in $K_2$.
\end{theorem}
\begin{proof}
    By Theorem \ref{t: ind-definable}, $CH_2$ is a countable union of definable sets. So by Corollary \ref{C: ind-definable implies definable}, since $CH_2$ is the image of the definable set $CH_1$, it follows that $CH_2$ is definable.
\end{proof}

This concludes Part 1 of the paper.

\part{The Zilber Trichotomy}

\section{Full Relics}

This section contains the main model-theoretic content of the paper. Our goal is to apply an appropriate version of \textit{Zilber's trichotomy}, which serves as an abstract analog of the fundamental theorem of projective geometry (and replaces the use of the corresponding theorem in \cite{KLOS}). In more precise terms, we will apply Zilber's trichotomy for \textit{ACF-relics} (see Definition \ref{D: relic}), and the related \textit{isomorphism theorem for full relics}, which we discuss more below. The idea is as follows: suppose we are in the setting of Convention \ref{Convention homeomorphism}, so we have a fixed homeomorphism $\varphi:X_1\rightarrow X_2$ where $X_i$ is over $K_i$. In this context, Theorem \ref{T: definability of hyperplane images} lets us view $\varphi$ as an isomorphism of certain $K_i$-relics. Then in the present section, we give a model-theoretic analysis of the $K_i$-relics in question, showing that they are full. Finally, in the next section, we apply the isomorphism theorem for full relics to obtain strong information about $\varphi$.

\textbf{Throughout this section, fix $K$, an uncountable algebraically closed field, and $X$, an irreducible quasi-projective variety over $K$ of dimension at least 2, embedded as a non-degenerate subvariety of $\mathbb P^n(K)$. We also fix a countable algebraically closed subfield $K_0\leq K$ over which $X$ is defined.}

\subsection{Defining the Hyperplane Relic}

\begin{definition}\label{D: hyp relic}
    We define the \textit{hyperplane relic on $X$}, denoted $\mathcal X^{hyp}$, to be the following $K$-relic:

    \begin{itemize}
        \item The universe of $\mathcal X^{hyp}$ is $X$.
        \item The language consists of:
        \begin{enumerate}
            \item A unary predicate for each $K_0$-definable subset of $X$.
            \item An $(n+1)$-ary relation $CH$, interpreted as $(x_1,...,x_{n+1})\in CH$ if and only if $x_1,...,x_{n+1}$ belong to a common irreducible component of some hyperplane section in $X$.
        \end{enumerate}
    \end{itemize}
\end{definition}

 Then $\mathcal X^{hyp}$ is a $K$-relic (see Definition \ref{D: relic}). This just says that all basic relations of $\mathcal X^{hyp}$ are $K$-definable (in particular, we use here that irreducibility of Zariski closed sets is definable in families --- see Johnson's appendix to \cite{JohIrr}). 

 Recall that a $K$-relic $\mathcal M=(M,...)$ is \textit{full} if every $K$-definable subset of every $M^n$ is definable in $\mathcal M$ (again, see Definition \ref{D: relic}). In this language, our main goal is to prove:

\begin{theorem}\label{T: full relic}
    $\mathcal X^{hyp}$ is a full $K$-relic.
\end{theorem}

\begin{convention}\label{Con: notation} For the rest of the section, we use subscripts $_K$ and $_{\mathcal X}$ to denote various model-theoretic notions in the two structures. For example, $\dim_{K}$ means dimension in the sense of algebraically closed fields, while $\dim_{\mathcal X}$ means dimension in the sense of $\mathcal X^{hyp}$. Similarly, we use $\acl_{K}$ and $\acl_{\mathcal X}$, and so on.
\end{convention}

Note that all basic relations of $\mathcal X^{hyp}$ are definable in $K$ over $K_0$ (we may express this as `$\mathcal X^{hyp}$ is defined over $K_0$'). In particular, the reader can easily check:

\begin{fact}\label{F: relic acl}
    \begin{enumerate}
        \item Let $Y\subset(\mathcal X^{hyp})^{eq}$. If $Y$ is $A$-definable in $\mathcal X^{hyp}$, then $Y$ is also $K_0A$-definable in $K$.
        \item Let $a\in(\mathcal X^{hyp})^{eq}$ and $B\subset(\mathcal X^{hyp})^{eq}$. If $a\in\acl_{\mathcal X}(B)$, then $a\in\acl_{K}(K_0B)$ (but a priori not necessarily the other way around).
        \item $\mathcal X^{hyp}$ is saturated.
    \end{enumerate}
\end{fact}

\subsection{Model-theoretic definitions and facts}

We now state the model-theoretic machinery that we will use -- namely, the Zilber trichotomy and the associated characterization of fullness. We need two definitions:

\begin{definition}\label{D: 1-based}
    Let $\mathcal M$ be a saturated $\omega$-stable structure (for example, $\mathcal M$ could be $\mathcal X^{hyp}$). Then $\mathcal M$ is \textit{1-based} if for any stationary $A$-interpretable set $X$ in $\mathcal M$, the canonical base $\operatorname{Cb}(X)$ is algebraic over some (equivalently any) single $A$-generic element of $X$. 
\end{definition}

\begin{definition}\footnote{These are not the most standard definitions of internality and almost internality, but they are well-known equivalents which are more suited to our needs.}
    Let $\mathcal M$ be any structure, and let $Y$ and $Z$ be definable in $\mathcal M^{eq}$. 
    \begin{enumerate}
        \item $Y$ is \textit{internal to $Z$} if it embeds definably into $Z^{eq}$ -- that is, if there are a definable equivalence relation $E$ on some $Z^n$, and a definable injection $Y\rightarrow Z^n/E$.
        \item $Y$ is \textit{almost internal} to $Z$ if there is a finite-to-one definable map $Y\rightarrow W$ for some definable set $W$ which is internal to $Z$.
    \end{enumerate}
\end{definition}

Note that internality and almost internality are transitive notions in the obvious way -- we will use this without mention.

\begin{fact}[Zilber Trichotomy]\label{F: trichotomy} Let $\mathcal M$ be a $K$-relic.
\begin{enumerate}
    \item $\mathcal M$ interprets an infinite field if and only if it is not 1-based; in this case, any such interpreted infinite field is $K$-definably isomorphic to $K$.
    \item Suppose $\mathcal M$ interprets a field $F\cong K$. Then $\mathcal M$ is full if and only if $M$ is internal to $F$ (as a definable set in $\mathcal M$).
\end{enumerate}
\end{fact}
\begin{proof}
    The first clause of (1) is the main result of \cite{CasHasYe} (Corollary 11.6). The second clause is Theorem 4.15 of \cite{PoiGroups}. Then (2) is Proposition 4.12 of \cite{CasHasVA}.
\end{proof}

So the remainder of this section will set up applications of Fact \ref{F: trichotomy}: first we show $\mathcal X^{hyp}$ is not 1-based; then we get a field $F$ as above, and we proceed to show that $X$ is $\mathcal X^{hyp}$-internal to $F$. 

\subsection{Non-1-basedness}

We begin by showing that $\mathcal X^{hyp}$ is not 1-based. The idea is simply that the family of hyperplane sections violates 1-basedness; but making this precise requires a careful analysis of dependence and canonical bases in $\mathcal X^{hyp}$ (as a priori they may look very different in $\mathcal X^{hyp}$ and in $K$). So below we go through the argument carefully. 

\begin{proposition}\label{P: 1-based}
    $\mathcal X^{hyp}$ is not 1-based.
\end{proposition}
\begin{proof}
    First, we may assume $\dim(X)=2$. Indeed, otherwise fix any $K_0$-definable non-degenerate closed irreducible surface $S\subset X$, and let $\mathcal S^{hyp}$ be the analogously defined structure on $S$. Then $\mathcal S^{hyp}$ is a relic of $\mathcal X^{hyp}$; and if we can show that $\mathcal S^{hyp}$ is not 1-based, it follows that $\mathcal X^{hyp}$ is not either (here we use that for $K$-relics, 1-based is preserved under interpretations, because by Fact \ref{F: trichotomy} it is equivalent to the non-existence of interpretable infinite fields). So for the rest of the proof of the proposition, let us assume $\dim(X)=2$.

    Now let $O$ be the uniquely $n$-sweeping orbit of $K_0$-generic hyperplane sections in $X$, let $a\in X^n$ be $K_0$-generic (in the sense of $K$), and let $Y\in O(a)$. So $Y$ is $K_0a$-definable in $K$, and $a$-definable in $\mathcal X^{hyp}$ (as the $a$-fiber in $CH$). Note that $Y$ is an irreducible curve -- so in the structure $K$, $Y$ is a \textit{strongly minimal set} (i.e. it is infinite, and every definable subset of it is finite or cofinite; equivalently, $Y$ is one-dimensional and stationary). Clearly, the same property holds of $Y$ in $\mathcal X^{hyp}$ as well.
    
    Now let $b\in Y$ be $K_0a$-generic in the sense of $K$ (i.e. $b$ is an element of $Y$ only belonging to the cofinite $K_0a$-definable sets). Then $b$ is also $a$-generic in $Y$ in the sense of $\mathcal X^{hyp}$. Let $c=\operatorname{Cb}_{\mathcal X}(Y)$ (i.e. the canonical base of $Y$ as computed in $\mathcal X^{hyp}$). Note that $Y$ is automorphism invariant over $c$ in $\mathcal X^{hyp}$, as the unique fiber of $CH$ in the almost equality class corresponding to $c$. Thus, by invariance, it follows that $Y$ is $c$-definable in $\mathcal X^{hyp}$.
    
    Now we will reach a contradiction. Namely, assume $\mathcal X^{hyp}$ is  1-based. Then since $b$ is $\mathcal X^{hyp}$-generic in $Y$, we have $c\in\acl_{\mathcal X}(b)$, and thus $Y$ is $\mathcal X^{hyp}$-definable over $\acl_{\mathcal X}(b)$. But then by Fact \ref{F: relic acl}, $Y$ is also $K$-definable over $\acl_K(K_0b)$. Finally, this is impossible since $O$ is 2-sweeping (since $n\geq\dim(X)=2$). More precisely, since $Y$ is $\acl_K(K_0b)$-definable, we get $\dim_K(a/K_0b)\leq\dim(Y^n)=n$, and thus $$ \dim_K(ab/K_0)\leq\dim(X)+n=n+2.$$ On the other hand, counting the other way gives $\dim_K(ab/K_0)=2n+1$, which is greater than $n+2$ as long as $n\geq 2$.
    \end{proof}

\begin{corollary}\label{C: interpretable field}
$\mathcal{X}^{hyp}$ interprets an algebraically closed field.
\end{corollary}

\begin{proof}
By Proposition~\ref{P: 1-based} and Fact \ref{F: trichotomy}(1). 
\end{proof}

\begin{convention}\label{Con: interpreted field}
From now on, fix an algebraically closed field $F$ interpreted in $\mathcal{X}^{hyp}$. Since $K_0$ is algebraically closed, it is an elementary substructure of $K$. So since $\mathcal X^{hyp}$ is defined over $K_0$, we may assume that $F$ and its interpretation are also $K$-definable over $K_0$.
\end{convention}

\subsection{Internality to a non-degenerate subvariety}

Our goal now is to show that $X$ is $\mathcal X^{hyp}$-internal to $F$. We will use the following well-known characterization of internality, which follows from a standard compactness argument:

\begin{fact}\label{F: internality char} Let $Y,Z$ be definable sets in $\mathcal X^{hyp}$. Then $Y$ is $\mathcal X^{hyp}$-internal to $Z$ if and only if there is a countable set $S$ so that every element of $Y$ is $\mathcal X^{hyp}$-definable over $S\cup Z$.
 
\end{fact}

Now our main tool is the following:

\begin{proposition}\label{P: internality to subvariety}
Let $Z\subset X$ be a non-degenerate irreducible quasi-projective subvariety defined over $K_0$. Then $X$ is $\mathcal X^{hyp}$-internal to $Z$.
\end{proposition}

\begin{proof} Note that $Z$ is $\emptyset$-definable in $\mathcal X^{hyp}$ by construction. We will prove the proposition by proving two claims. As before, fix $O$ to be the orbit of $K_0$-generic hyperplane sections in $X$.

\begin{claim}
Every member of $O$ is $\mathcal X^{hyp}$-definable over $Z$.
\end{claim}

\begin{proof}
By Theorem \ref{L: hyperplane sweeping} (and the non-degeneracy of $Z$), $O$ uniquely $n$-sweeps $Z$. It follows that a member $Y\in O$ is $\mathcal X^{hyp}$-definable as the fiber in $CH$ above any generic $n$-tuple from $Y\cap Z$. 
\end{proof}

\begin{claim}
Let $x \in X$ with $x \notin \acl_K(K_0)$. Then $x$ is $\mathcal X^{hyp}$-definable over $Z$.
\end{claim}

\begin{proof}
Let $\mathcal{H}$ be the space of hyperplanes in $\mathbb{P}^n(K)$, and let $\mathcal{H}_x$ be the subspace of hyperplanes containing $x$ (so $\mathcal H_x$ is a copy of $\mathbb P^{n-1}(K)$ inside a copy of $\mathbb P^n(K)$). We claim that a $K_0x$-generic element of $\mathcal{H}_x$ is $K_0$-generic in $\mathcal{H}$. Otherwise, by Lemma \ref{L: codim 1 definable over acl}, $\mathcal{H}_x$ is defined over $\acl_K(K_0)$, and since $x$ is the intersection of all hyperplanes in $\mathcal{H}_x$, this implies $x \in \acl_K(K_0)$, a contradiction.

Thus a $K_0x$-generic hyperplane through $x$ is $K_0$-generic. Let $H_1,...,H_n$ be independent such hyperplanes, and set $Y_i = X \cap H_i$. Then each $Y_i\in O$, so by the previous claim each $Y_i$ is $\mathcal X^{hyp}$-definable over $Z$; and moreover $x$ is the unique point in $\bigcap_{i=1}^n Y_i$, so $x$ is thus also $\mathcal X^{hyp}$-definable over $Z$.
\end{proof}

Finally, let $S = X \cap \acl_K(K_0)$. Then $S$ is countable, and it follows from the previous claim that every element of $X$ is definable in $\mathcal X^{hyp}$ over $S\cup Z$ (either because it is definable over $Z$ or because it is an element of $S$). So by Fact \ref{F: internality char}, $X$ is $\mathcal X^{hyp}$-internal to $Z$.
\end{proof}

\subsection{Almost internality to the field}

We will show $X$ is $\mathcal X^{hyp}$-internal to $F$ using two applications of Proposition \ref{P: internality to subvariety}. The first application shows almost internality:

\begin{proposition}\label{P: X almost internal to F}
$X$ is $\mathcal X^{hyp}$-almost internal to $F$.
\end{proposition}

\begin{proof}
Let $C \subset X$ be a non-degenerate quasi-projective curve defined over $K_0$. Working in the structure $\mathcal X^{hyp}$, Proposition~\ref{P: internality to subvariety} gives that $X$ is internal to $C$. In particular, $C$ is non-orthogonal to $F$. But $C$ is strongly minimal, so non-orthogonality implies that $C$ is almost internal to $F$. Since $X$ is internal to $C$, it follows that $X$ is almost internal to $F$.
\end{proof}

By Proposition~\ref{P: X almost internal to F}, there is a finite-to-one $\mathcal{X}^{hyp}$-definable map $f:X\rightarrow W$ where $W$ is $\mathcal X^{hyp}$-internal to $F$. We fix such a map $f$ for the rest of this section. Since $\mathcal X^{hyp}$ and $F$ are $K$-definable over $K_0$, we may also take $f$ to be $K$-definable over $K_0$ (here again we use that $K_0$ is an elementary substructure of $K$).

\subsection{Internality to the field}

We now finish the proof by showing that $X$ is $\mathcal X^{hyp}$-internal to $F$. First we get internality on a subvariety:

\begin{lemma}\label{L: generically injective restriction}
There is a non-degenerate irreducible quasi-projective subvariety $Z \subset X$ such that $f\restriction Z$ is injective.
\end{lemma}

\begin{proof}
Throughout the proof of Lemma \ref{L: generically injective restriction}, we work in the field language (i.e. in the structure $K$). First choose a \textit{Veronese} embedding of $X$ -- that is, an embedding $X\hookrightarrow\mathbb P^m(K)$ for some $m>n$, sending each subvariety of $X$ of positive dimension to a non-degenerate subvariety of $\mathbb P^m(K)$ (for more on such embeddings, see \cite[Lecture 2]{HarrisAG}). As above, since $X$ is defined over $K_0$, we may choose this embedding to be defined over $K_0$.

Now let $Z\subset X$ be a $K_0$-generic hyperplane section seen in $\mathbb P^m(K)$, and let $c=\operatorname{Cb}(Z)$. So $Z$ is irreducible. Moreover, $Z$ is non-degenerate in $\mathbb P^n(K)$, since $\dim(c/K_0)=m>n$ (thus $Z$ cannot be a hyperplane in $n$-space). 

We claim that $f$ is injective on a dense open subset of $Z$, which implies the lemma. By compactness, it suffices to show that if $x,y \in Z$ are each $K_0c$-generic, and $f(x)=f(y)$, then $x=y$.

So take such $x,y$. Since $Z$ is a generic hyperplane section in $\mathbb P^m(K)$, and $x$ is generic over $K_0c$, we get $\dim(c/K_0)=m$ and $\dim(c/K_0x)=m-1$. But since $f$ is finite-to-one and $K_0$-definable, we also get that $x$ and $y$ are intgeralgebraic over $K_0$, so that $\dim(c/K_0xy)=\dim(c/K_0x)=m-1$. Thus codimension 1 many hyperplanes in $\mathbb P^m(K)$ pass through both $x$ and $y$, which implies $x=y$.
\end{proof}

We now fix $Y$ as in the previous lemma -- so $Y\subset X$ is irreducible, quasi-projective, and non-degenerate in $\mathbb P^n(K)$, and $f$ is injective on $Y$. Again, since $X$ and $f$ are defined over $K_0$, we may choose $Y$ to be defined over $K_0$. Finally: 

\begin{proof}[Proof of Theorem~\ref{T: full relic}]

Since $f$ is injective on $Y$, it follows that $Y$ is $\mathcal X^{hyp}$-internal to $F$. Moreover, by Proposition~\ref{P: internality to subvariety}, $X$ is $\mathcal X^{hyp}$-internal to $Y$. Thus $X$ is $\mathcal X^{hyp}$-internal to $F$, which gives fullness by Fact \ref{F: trichotomy}. 
\end{proof}

\section{Reduction to the definable case}

Our goal now is to apply Theorem \ref{T: full relic} to reduce the homeomorphism $\varphi$ from Convention \ref{Convention homeomorphism} to a definable homeomorphism. For this, our main tool is the `isomorphism theorem for full relics':

\begin{fact}\label{F: isomorphism theorem}
Let $K_1, K_2$ be algebraically closed fields, let $\mathcal{X}_1, \mathcal{X}_2$ be $K_i$-relics, and let $\varphi:\mathcal X_1\rightarrow\mathcal X_2$ be an isomorphism. Then:

\begin{enumerate}
    \item $\mathcal X_1$ is full in $K_1$ if and only if $\mathcal X_2$ is full in $K_2$.
    \item If the $X_i$ are full, then $\varphi$ has the form $\sigma \circ f$, where $\sigma$ is induced by a field isomorphism $K_1 \to K_2$ and $f$ is definable in $K_2$.
\end{enumerate}
\end{fact}
\begin{proof}
    As noted in \cite[Remark 4.13]{CasHasVA}, (1) follows by Fact \ref{F: trichotomy}(2), since the characterization of fullness there is intrinsic to the relic as a structure. Then (2) is \cite[Lemma 2.6]{CaHaDC}.
\end{proof}

We now prove:

\begin{theorem}\label{T: reduction to definable case}
In the setting of Convention~3.1, there are a field isomorphism $\sigma : K_1 \to K_2$ and a $K_2$-definable homeomorphism $f : \sigma(X_1) \to X_2$ such that on $X_1$ we have
\[
\varphi = f \circ \sigma.
\]
\end{theorem}

\begin{proof}
Endow $X_1$ with the structure $\mathcal{X}_1^{hyp}$ as in the previous seciton. By Theorem~8.2, $\mathcal X_1^{hyp}$ is a full $K_1$-relic. Let $\mathcal X_2^{hyp}$ be the image of $\mathcal X_1^{hyp}$ under $\varphi$ -- so $\mathcal X_2^{hyp}$ is a structure in the same language as $\mathcal X_1^{hyp}$, obtained by applying $\varphi$ to each basic relation in the language. Then the main point is:

\begin{claim}
    $\mathcal X_2^{hyp}$ is a $K_2$-relic.
\end{claim}
\begin{proof}
    We need to show that the image of each basic relation of $\mathcal X_1^{hyp}$ under $\varphi$ is $K_2$-definable. There are two types of basic relations in $\mathcal X_1^{hyp}$:
    \begin{enumerate}
        \item $CH_1\subset X_1^{n+1}$, whose image is $K_2$-definable by Theorem \ref{T: definability of hyperplane images}.
        \item Unary predicates for definable subsets of $X_1$, whose images are automatically $K_2$-definable since $\varphi$ is a homeomorphism (precisely, homeomorphisms preserve constructible subsets of $X^i$, which are the same as definable sets).
    \end{enumerate}
\end{proof}

Now by Fact \ref{F: isomorphism theorem}(1), it follows that $\mathcal X_2^{hyp}$ is also full in $K_2$. And then we have that $\varphi$ gives an isomorphism of full relics, so that Theorem \ref{T: reduction to definable case} follows from Fact \ref{F: isomorphism theorem}(2).
\end{proof}

\part{Proofs of the Theorems}

\section{Analysis of Generically Rational Definable Homeomorphisms}

Here we continue to work with a homeomorphism $\varphi:X_1\rightarrow X_2$ as in Convention \ref{Convention homeomorphism}. We are now ready to begin classifying the possibilities for $\varphi$. In this language, Theorem \ref{T: reduction to definable case} lets us assume that $K_1=K_2$ and $\varphi$ is definable in the language of fields -- so $\varphi$ is really a definable homeomorphism between varieties over the same field. For the next two sections, we study such maps.

Our first move is to apply quantifier elimination in algebraically closed fields. In particular, quantifier elimination gives that any definable function between irreducible varieties is the composition of a Frobenius power with a generically rational map. In this section, we ignore the Frobenius power and restrict our attention only to the generically rational case.\footnote{The reader may want to simplify the presentation by pointing out that since the Frobenius is a field automorphism $K_2\rightarrow K_2$, we may absorb it into our already set-aside field isomorphism $K_1\rightarrow K_2$, thereby concluding that our original homeomorphism is given by a field isomorphism composed with a generically rational map. This subtly different statement is actually not helpful to us, because eventually we will need to simultaneously analyze both $\varphi$ and $\varphi^{-1}$, which cannot be made  simultaneously generically rational up to field automorphisms.} Since we are working over the same field from now on, it will simplify notation to rewrite $X_1,X_2$ as $X$ and $Y$:

\begin{convention}\label{Con: generically rational setting}
\textbf{Throughout this section, we fix the following data:}
\begin{enumerate}
    \item $K$ is an uncountable algebraically closed field.
    \item $X, Y$ are irreducible quasi-projective varieties over $K$ of dimension $d\geq 2$.
    \item $\varphi : X \to Y$ is a homeomorphism which is definable in the language of fields.
    \item $r:X\dashrightarrow Y$ is a rational map agreeing with $\varphi$ on some dense open subset of $X$.
    \item $\nu_X:\widetilde X\rightarrow X$ and $\nu_Y:\widetilde Y\rightarrow Y$ are fixed normalizations of $X$ and $Y$. 
\end{enumerate}
\textbf{We work throughout in the structure $K$. Absorbing parameters, we assume that all data above is $\emptyset$-definable.}
\end{convention}

Our goal in this section is to prove:

\begin{theorem}\label{T: normalization square}
In the setting of Convention~\ref{Con: generically rational setting}, there is a morphism of $K$-varieties
\[
\tilde{\varphi} : \widetilde{X} \to \widetilde{Y}
\]
such that the following diagram commutes:
\[
\xymatrix{
\widetilde{X} \ar[r]^{\tilde{\varphi}} \ar[d]_{\nu_X} & \widetilde{Y} \ar[d]^{\nu_Y} \\
X \ar[r]_{\varphi} & Y
}
\]
\end{theorem}

\subsection{Two maps}

Throughout this section, we fix a projective closure $\overline{Y}$ of $Y$, and we consider two maps from $\widetilde{X}$ to $\overline{Y}$.

First, let
\[
f:= \varphi \circ \nu_X : \widetilde{X} \to Y \hookrightarrow \overline{Y}.
\]

Second, let
\[
\rho : \widetilde{X} \dashrightarrow \overline{Y}
\]
be the rational map induced by $r$.

Since $\widetilde{X}$ is normal and $\overline{Y}$ is projective, the map $\rho$ extends to a morphism on an open subset of $\widetilde{X}$ whose complement has codimension at least $2$ (\cite[Proposition II.6.1]{Hartshorne}).

Our goal is to show that $\rho$ extends to all of $\widetilde X$ and agrees with $f$ everywhere. This will give that $\varphi\circ\nu_
X=f$ is a morphism from $\widetilde X$ to $Y$, which (by normality of $\widetilde X$) must factor through $\tilde Y$, thereby giving Theorem \ref{T: normalization square}. This section is inspired by similar material in \cite[Section 5.1]{KLOS} -- the main difference being that we need to work a bit harder to deal with the possibility that $X,Y$ are not normal.\footnote{The goal at that point of \cite{KLOS} was to prove their Proposition 5.1.1, which is a bit different from our Theorem \ref{T: normalization square} -- but the statements do have a similar structure, and the proofs are very similar.} More precisely, Lemmas \ref{L: agreement in codim 1} and \ref{L: two maps agree everywhere} below were inspired by \cite[Lemma 5.1.3]{KLOS}, and Lemma \ref{L: rational map defined everywhere} was inspired by \cite[Lemma 5.1.4]{KLOS}.

\subsection{Subvariety through a fixed tuple} We begin with an algebraic geometry fact that will be very useful for forcing the equality of $f$ and $\rho$. This is surely well-known, but we include a proof. Fact \ref{F: curve through tuple} is the analog in our argument of \cite[Lemma 5.1.2]{KLOS}:

\begin{fact}\label{F: curve through tuple} Let $W$ be any irreducible quasi-projective variety of dimension at least 2, and let $A,B$ be disjoint finite subsets of $W$. Then there is a closed irreducible subvariety $V\subset W$ passing through each point of $A$ but passing through no points of $B$.
\end{fact}

\begin{proof}
    As in the proof of Theorem \ref{T: full relic}, we use the Veronese embedding. So suppose $W$ is embedded in $\mathbb P^n$, and choose a degree $d$ Veronese embedding of $\mathbb P^n$ into $\mathbb P^m$ (so $m=\binom{n+d}{d}-1$). It is well-known that under this embedding, any $d+1$ distinct points of $\mathbb P^n$ map to a linearly independent set in $\mathbb P^m$ (i.e. the corresponding $d$ points in $\mathbb P^m$ impose independent conditions on hyperplanes in $\mathbb P^m$). Thus, choosing $d$ large enough, we may assume the entire image of $A\cup B$ is independent in $\mathbb P^m$. So, from now on, let us just assume $W$ is already embedded in $\mathbb P^m$, and the points of $A\cup B$ impose independent conditions on hyperplanes in $\mathbb P^m$.

    Now let $H\leq\mathbb P^m$ be generic among all hyperplanes passing through each point of $A$. Since the points of $A\cup B$ impose independent conditions, $H$ does not pass through any point of $B$. Finally, set $V=H\cap W$. By Bertini's irreducibility theorem (applied to the linear system of hyperplanes through $A$), $V$ is irreducible. Thus, $V$ satisfies all of the desired conditions.
\end{proof}

\subsection{Codimension 1 Points}

We now return to the goal of identifying $f$ and $\rho$ everywhere. We first achieve this goal outside a codimension 2 locus. Recall that we set $d=\dim(X)=\dim(Y)$.

\begin{lemma}\label{L: agreement in codim 1}
Let $x \in \widetilde{X}$ be such that
\[
\dim(x) \geq d - 1.
\]
Then $\rho$ is defined at $x$, and
\[
\rho(x) = f(x).
\]
\end{lemma}

\begin{proof}
That $\rho$ is defined at $x$ is automatic from $\dim(x)\geq d-1$, since the indeterminacy locus of $\rho$ has dimension at most $d-2$. We show that $\rho(x)=f(x)$.

Let $\overline{\widetilde{X}}$ be a projective closure of $\widetilde{X}$, let $\Gamma \subset \widetilde{X} \times \overline{Y}$ be the graph of $\rho$, and let $\overline{\Gamma}$ be the closure of $\Gamma$ in $\overline{\widetilde{X}} \times \overline{Y}$. Then $\overline{\Gamma}$ is an irreducible projective variety of dimension $d$.

Let $y=\rho(x)$ and $z=f(x)$, and suppose toward a contradiction that $y\neq z$. Note that $f$ is finite-to-one (since $\nu_X$ and $\varphi$ are); so $x$ and $z$ are interalgebraic, and thus $\dim(z)\geq d-1$. It now follows by dimension considerations that the fiber of $\overline{\Gamma}\rightarrow\overline Y$ over $z$ is finite (otherwise a generic point of $\overline\Gamma$ would belong to an infinite fiber, a contradiction). So let $w_1,\dots,w_k \in \overline{\widetilde{X}}$ be the points such that $(w_i,z)\in \overline{\Gamma}$. Since $\rho(x)$ is defined and takes value $y$, we have $x\notin\{w_1,\dots,w_k\}$.

Now by Fact \ref{F: curve through tuple}, we can choose an irreducible closed subvariety $V\subset \overline{\widetilde{X}}$ passing through $x$, avoiding $\{w_1,\dots,w_k\}$, and containing a generic point $u\in \widetilde{X}$. Since $u$ is generic, $\rho$ is defined and agrees with $f$ in some open set $U\subset\widetilde X$ containing $u$. So $U$ contains a non-empty open subset of $V$, and since $V$ is irreducible, this means $V=\overline{V\cap U}$. In particular, $x\in\overline{V\cap U}$, and thus by continuity we get $$z=f(x)\in\overline{f(V\cap U)}=\overline{\rho(v\cap U)}.$$ So if we let $\overline{\Gamma}_V=\Gamma\cap(V\times\overline Y)$ and $\pi:\Gamma\rightarrow\overline Y$ the projection, then $z\in\overline{\pi(\overline{\Gamma}_V)}$. But by projectivity, $\pi(\overline{\Gamma}_V)$ is already closed, so in fact $z\in\pi(\overline{\Gamma}_V)$. Thus there is some $w\in V$ with $(w,z)\in\overline\Gamma$. So $w=w_i$ for some $i$, contradicting that each $w_i\notin V$. 
\end{proof}

\subsection{Extending the map} We now show that $\rho$ extends to a morphism, using a similar argument to \cite[Lemma 5.1.4]{KLOS}. Recall that a rational function on a normal variety always extends to a regular function, provided it is defined outside a codimension 2 locus (\cite[Theorem II.6.3A(ii)]{Hartshorne}). In particular, a rational map from a normal variety to an affine variety always extends to a morphism, provided it is defined outside a codimension 2 locus. Using this, we show:

\begin{lemma}\label{L: rational map defined everywhere}
$\rho$ is defined on all of $\widetilde X$, and all values of $\rho$ on $\widetilde X$ lie in $Y$. 
\end{lemma}
\begin{proof}
    Let $x\in\widetilde X$, and let $y=f(x)\in Y$. Let $B\subset Y$ be an affine open set containing $y$, and let $A=f^{-1}(B)$. So $A$ is open and contains $x$. By Lemma \ref{L: agreement in codim 1}, $\rho$ is defined and agrees with $f$ outside a codimension 2 locus in $A$. In particular, outside a codimension 2 locus, $\rho$ is defined and takes in $B$. But $A$ is normal and $B$ is affine, so it now follows that $\rho$ extends to a morphism $A\rightarrow B$. Thus $\rho(x)$ is defined and belongs to $Y$, which proves the lemma. 
\end{proof}

\subsection{Agreement everywhere}

Finally, we show that $f$ and $\rho$ agree everywhere. Now that we know $\rho$ takes values in $Y$ (Lemma \ref{L: rational map defined everywhere}), we no longer need to view $f$ and $\rho$ as maps to $\overline Y$. Thus, from now on, we view $f$ and $\rho$ as maps $\widetilde X\rightarrow Y$. In this setup, it follows that $f$ is a \textit{closed} map: indeed, normalizations are finite (\cite[\href{https://stacks.math.columbia.edu/tag/0BXR}{Lemma 0BXR}]{stacks-project}), thus closed (\cite[\href{https://stacks.math.columbia.edu/tag/01WJ}{Lemma 01WJ}]{stacks-project} and \cite[\href{https://stacks.math.columbia.edu/tag/01WM}{Lemma 01WM}]{stacks-project}), and $\varphi$ is automatically closed; so $f=\varphi\circ\nu_X$ is also closed. Now we show:

\begin{lemma}\label{L: two maps agree everywhere}
    $f$ and $\rho$ agree on all of $\widetilde X$.
\end{lemma}
\begin{proof}
    Let $x\in\widetilde X$, and let $y=f(x)$ and $z=\rho(x)$. Toward a contradiction, assume $y\neq z$. Noting that $f$ is finite-to-one (as the composition of two finite-to-one maps), let $w_1,...,w_k$ be the preimages of $z$ under $f$. Note that $x\notin\{w_1,...,w_k\}$ by assumption (otherwise $f(x)=z$).

    We now essentially repeat the argument of Lemma \ref{L: agreement in codim 1}, switching the roles of $f$ and $\rho$ and using the closedness of $f$ in lieu of passing to projective closures.
    
    First, by Fact \ref{F: curve through tuple}, we can let $V\subset\widetilde X$ be a closed irreducible subvariety containing $x$ and a generic $u\in\widetilde X$ but not containing any $w_i$. As in Lemma \ref{L: agreement in codim 1}, we argue the following:
    \begin{itemize}
        \item $f$ and $\rho$ agree on some open set $U$ around $x$.
        \item $x\in\overline{V\cap U}$.
        \item $\rho(x)=z\in\overline{\rho(V\cap U)}=\overline{f(V\cap U)}\subset\overline{f(V)}$.
        \item $f(V)$ is already closed, and thus $z\in f(v)$, contradicting that $w_i\notin V$ for each $i$. (This time we show $f(V)$ is closed using that $f$ is closed and $V$ is closed, rather than projectivity).
    \end{itemize}
\end{proof}

Finally:

\begin{proof}[Proof of Theorem \ref{T: normalization square}] By Lemma \ref{L: two maps agree everywhere}, $f=\varphi\circ\nu_X$ is given by a morphism $\rho:\widetilde X\rightarrow Y$. By the universal property of normalizations, this morphism factors through $\widetilde Y$, which gets the desired commuting square to prove Theorem \ref{T: normalization square}.
\end{proof}

\section{The Main Theorem}

To prove Theorem \ref{T: main}, our final task is to strengthen \ref{T: normalization square} by showing that the lifted map on normalizations is a universal homeomorphism. The key idea is to apply Theorem \ref{T: normalization square} both $\varphi$ and $\varphi^{-1}$, and deduce that the resulting morphism between normalizations must be bijective. Unfortunately, Theorem \ref{T: normalization square} only works in the generically rational case, and in general one cannot simultaneously arrange that $\varphi$ and $\varphi^{-1}$ are rational (that is, in positive characteristic we cannot factor an arbitrary homeomorphism into a field automorphism and a birational homeomorphism). So before we proceed, we need a version of Theoren \ref{T: normalization square} not requiring generic rationality:

\subsection{Lifting in the non-rational case}

\begin{lemma}\label{L: nonrational commuting square}
Let $K$ be an uncountable algebraically closed field, and let $\varphi:X\rightarrow Y$ be a $K$-definable homeomorphism between irreducible quasiprojective varieties of dimension $d\geq 2$ over $K$. Let $\nu_X:\widetilde{X}\rightarrow X,\nu_Y:\widetilde{Y}\rightarrow Y$ denote the normalizations of $X$ and $Y$. Let $\Gamma\subset X\times Y$ be the graph of $\varphi$, and let $\widetilde\Gamma$ be the preimage of $\Gamma$ in $\widetilde X\times\widetilde Y$. Then $\widetilde\Gamma$ contains an irreducible closed $d$-dimensional subvariety $V\subset\widetilde X\times\widetilde Y$ projecting bijectively to $\widetilde X$. 
\end{lemma}

\begin{proof}
The idea is that the statement of Lemma \ref{L: nonrational commuting square} is insensitive to applying a Frobenius power everywhere, and thus we can reduce to the generically rational case and apply Theorem \ref{T: normalization square}.

Specifically, let us first embed $X$ and $\widetilde X$ into some $\mathbb P^m(K)$, and $Y$ and $\widetilde Y$ into some $\mathbb P^n(K)$. Then by quantifier elimination in algebraically closed fields, $\varphi$ can be factored as $f\circ\operatorname{Fr}^k$ for some $k$, where $\operatorname{Fr
}^k$ is the $k$th Frobenius power on $\mathbb P^m(K)$ and $f:\operatorname{Fr}^k(X)\rightarrow Y$ is generically rational. Then $(\operatorname{Fr}^k,\operatorname{id})$ defines a universal homeomorphism sending $X$, $\widetilde{X}$, $\Gamma$, and $\widetilde\Gamma$ to their images under $\operatorname{Fr}^k$, thereby converting the whole setup of Lemma \ref{L: nonrational commuting square} to an equivalent setup for the generically rational map $f:\operatorname{Fr}^k(X)\rightarrow Y$. So without loss of generality, we may replace $X$ with $\operatorname{Fr}^k(X)$ and thereby assume $\varphi$ is already generically rational.

In this case, Theorem \ref{T: normalization square} gives that $\widetilde\Gamma$ contains the graph of the morphism $\widetilde\varphi$, which is a closed irreducible $d$-dimensional subvariety of $\widetilde X\times\widetilde Y$ projecting bijectively to $\widetilde X$.
\end{proof}

\subsection{The rational case revisited}

We now use Lemma \ref{L: nonrational commuting square} to show that the morphism from Theorem \ref{T: normalization square} is a universal homeomorphism:

\begin{lemma}\label{L: gen rat full}
    In the setting of Convention~\ref{Con: generically rational setting}, there is a universal homeomorphism of $K$-varieties
\[
\tilde{\varphi} : \widetilde{X} \to \widetilde{Y}
\]
such that the following diagram commutes:
\[
\xymatrix{
\widetilde{X} \ar[r]^{\tilde{\varphi}} \ar[d]_{\nu_X} & \widetilde{Y} \ar[d]^{\nu_Y} \\
X \ar[r]_{\varphi} & Y
}
\]
\end{lemma}

\begin{proof}
    Theorem \ref{T: normalization square} provides such a diagram where $\tilde\varphi$ is a morphism; we show it must be a universal homeomorphism. The trick is to apply Lemma \ref{L: nonrational commuting square} to $\varphi^{-1}$. First let $\Gamma\subset X\times Y$ be the graph of $\varphi$, and $\widetilde\Gamma\subset\widetilde X\times\widetilde Y$ its preimage. Then, noting that $\varphi^{-1}$ is also a definable homeomorphism, Lemma \ref{L: nonrational commuting square} gives that $\widetilde\Gamma$ contains an irreducible $d$-dimensional closed subvariety $V\subset\widetilde X\times\widetilde Y$ projecting bijectively to $\widetilde Y$. But the projection $\widetilde\Gamma\rightarrow\Gamma\rightarrow X$ is finite-to-one everywhere and generically bijective; so since $X$ is irreducible of dimension $d$, it follows that $\widetilde\Gamma$ can contain at most one irreducible $d$-dimensional closed subvariety of $\widetilde X\times\widetilde Y$. So since the graph of $\tilde\varphi$ is an irreducible $d$-dimensinal closed subvariety, it must be that $V$ is the graph of $\tilde\varphi$. In particular, since $V\rightarrow\widetilde Y$ is bijective, we conclude that $\tilde\varphi$ is bijective. Finally, the lemma follows since a bijective morphism of normal varieties is always a universal homeomorphism (precisely, after factoring out a purely inseparable map, we can reduce to the birational case (see e.g. \cite[Lemma 9.26]{CasHasYe}), from which we conclude by Zariski's main theorem -- see also the proof of Proposition \ref{P: qp spec}). 
\end{proof}

\subsection{Main theorem}

Finally, we have proven our main theorem:

\begin{proof}[Proof of Theorem \ref{T: main}]
By Theorem \ref{T: reduction to definable case}, we reduce to the case that $K_1=K_2=K$ and $\varphi$ is $K$-definable. Then after canceling out an appropriate Frobenius power (and adding it to the field isomorphism), we reduce to the case that $\varphi$ is generically rational. In this case, the theorem follows by Lemma \ref{L: gen rat full}.
\end{proof}

\section{Reducible Varieties}

In this section we show how Theorems \ref{T: reducible schemes} and \ref{T: main reducible} follow from Theorem \ref{T: main}. The idea is that Theorem \ref{T: reducible schemes} follows immediately, while Theorem \ref{T: main reducible} can be resolved with a bit more model theoretic analysis.

\subsection{The scheme formulation}

We begin with Theorem \ref{T: reducible schemes}. \textbf{Fix uncountable algebraically closed fields $K_1,K_2$, and quasi-projective varieties $X_i$ over $K_i$, all of whose irreducible components have dimension at least 2. Moreover, fix a homeomorphism $\varphi:X_1\rightarrow X_2$.} Recall that Theorem \ref{T: reducible schemes} says that $\varphi$ lifts to a universal homeomorphism of normalized schemes $\widetilde{X_1}\rightarrow\widetilde{X_2}$. 

\begin{proof}[Proof of Theorem \ref{T: reducible schemes}] First assume the $X_i$ are irreducible. Then by Theorem \ref{T: main}, $\varphi$ lifts to a map $\widetilde\varphi:\widetilde{X_1}\rightarrow\widetilde{X_2}$ which factors into a field isomorphism followed by a universal homeomorphism over $K_2$. Both of these factors define universal homeomorphisms of schemes, so the composite map $\widetilde\varphi$ is also a universal homeomorphism of schemes, and thus we are done.

Now assume the $X_i$ are reducible. Then $\varphi$ induces a bijection $g$ between the irreducible components of $X_1$ and $X_2$, and moreover gives a homeomorphism $C\rightarrow g(C)$ for each component $C$ of $X_1$. By separate applications of the irreducible case above, the restriction of $\varphi$ to each pair of components $(C,g(C))$ lifts to a universal homeomorphism $\widetilde C\rightarrow\widetilde{g(C)}$ of normalized schemes. But $\widetilde{X_1}$ is just the disjoint union of the $\widetilde C$, and $\widetilde{X_2}$ is the disjoint union of $\widetilde{g(C)}$. So the disjoint union map $\widetilde\varphi$ then gives a universal homeomorphism $\widetilde{X_1}\rightarrow\widetilde{X_2}$, which is a lift of $\varphi$.
\end{proof}

\begin{remark}
    The main point of Theorem \ref{T: reducible schemes} is that we forget the base fields, so that it is easier to give a morphism of normalizations (in particular, the map $\widetilde\varphi$ is in general not factorable into a field isomorphism and a morphism of varieties). 
\end{remark}

\subsection{Counterexamples from finite intersections}

Here we make explicit the reason we cannot hope Theorem \ref{T: main} to translate verbatim to the reducible case (i.e. the reason Theorem \ref{T: reducible schemes} is the best possible statement in general for reducible varieties). As in the introduction, we define:

\begin{definition}
    A variety $X$ over the algebraically closed field $K$ is \textit{strongly connected} if $X$ is not expressible as the union of two proper closed subsets with finite intersection.
\end{definition}

Lemma \ref{L: counterexample reducible} says that if we want a factorization $f\circ\sigma$ as in the irreducible case, we \textit{must} assume the varieties in question are strongly connected:

\begin{lemma}\label{L: counterexample reducible}
Let $K$ be an algebraically closed field, and let $X$ be a variety over $K$, all of whose irreducible components have dimension at least 2. Assume $X$ is not strongly connected. Then there is a homeomorphism $\varphi:X\rightarrow X$ for which the conclusion of Theorem \ref{T: main} fails. That is, there does not exist a lift $\widetilde\varphi:\widetilde X\rightarrow\widetilde X$ factoring into a field isomorphism composed with a morphism of varieties.
\end{lemma}
\begin{proof}
    We essentially copy \cite[Lemma 10.10]{CasHasAV}. Let $X=C\cup D$ where $C,D$ are proper closed subsets with finite intersection. Then let $F\leq K$ be a finitely generated field defining $X$, $C$, $D$, and each point of $C\cap D$. Then for each $\tau\in\operatorname{Gal}(K/F)$, we get a homeomorphism $\varphi_\tau:X\rightarrow X$ by acting as the identity on $C$ and as $\tau$ on $D$. Suppose for some $\tau$ that $\varphi_\tau$ lifts to a map factoring as $f\circ\sigma$ as in Theorem \ref{T: main}. Then $\sigma$ just inverts $f$ on $f(\widetilde C)$. In particular, $\sigma$ is entirely determined by $f$. Moreover, $\tau$ is determined entirely by $\sigma$ and $f$, since $\tau$ acts as $f\circ\sigma$ on $\widetilde D$. Put together, $\tau$ is determined by $f$.

    But the order of the group $\operatorname{Gal}(K/F)$ is strictly greater than the number of possible morphisms $f$ (the latter is $|K|$ at most, and the former can be checked to be strictly larger than $|K|$, using that $F$ is finitely generated). So there must be some choice of $\tau$ for which no such $f$ and $\sigma$ exist.
\end{proof}

\subsection{Model theory of strongly connected varieties}

From now on, we consider the strongly connected case, with the goal of proving Theorem \ref{T: main reducible}. To recall, this theorem says that the statement of Theorem \ref{T: main} works verbatim for reducible varieties which are strongly connected. In this subsection we give the model-theoretic content, by showing how to endow a strongly connected variety with a full relic structure.

\begin{convention}\label{Con: reducible full} \textbf{In this subsection, we fix the following:}
\begin{enumerate}
    \item An uncountable algebraically closed field $K$.
    \item A strongly connected quasi-projective variety $X$ over $K$, all of whose irreducible components have dimension at least 2.
    \item A countable algebraically closed subfield $K_0\leq K$ over which $X$ (and thus all of its irreducible components) are defined.
\end{enumerate}
    Moreover:
    \begin{enumerate}
        \item[(a)] We let $Y_1,...,Y_m$ be the irreducible components of $X$.
        \item[(b)] For each $Y_i$, we fix a structure $\mathcal Y_i^{hyp}$ as in Definition \ref{D: hyp relic}, \textbf{potentially obtained via different projective embeddings of each component} (the idea is that we need to choose a non-degenerate embedding for each $Y_i$).
        \item[(c)] We let $\mathcal X^{hyp}$ be the structure on $X$ obtained by naming each $Y_i$ with a unary predicate, and then endowing each $Y_i$ with the structure $\mathcal Y_i^{hyp}$ in disjoint languages.
    \end{enumerate}
\end{convention}

So $\mathcal X^{hyp}$ is a $K$-relic, which really consists of $m$ different $K$-relics $\mathcal Y_i^{hyp}$ in $m$ different languages, which are glued together along their intersections. A priori, each $\mathcal Y_i^{hyp}$ is full (by Theorem \ref{T: full relic}), but $\mathcal Y$ might not be. \textbf{The main content of this section is that if $X$ is strongly connected, $\mathcal Y$ \textit{is} full}. This is really a further computation with internality -- the idea is that strong connectedness gives just enough interaction between the $\mathcal Y_i^{hyp}$ that they all become internal to each other. Now we give the proof:

\begin{theorem}\label{T: reducible full}
    In the setting of Convention \ref{Con: reducible full}, the relic $\mathcal X^{hyp}$ is full.
\end{theorem}
\begin{proof}
    By Theorem \ref{T: full relic} applied to each component, we get that each $\mathcal Y_i^{hyp}$ is full. So by Fact \ref{F: trichotomy}(3), for each $i$ there is a $\mathcal Y_i^{hyp}$-definable algebraically closed field $F_i$ such that $Y_i$ is internal to $F_i$ in $\mathcal Y_i^{hyp}$. Of course, this implies that $F_i$ is $\mathcal X^{hyp}$-definable and $Y_i$ is internal to $F_i$ in $\mathcal X^{hyp}$.

    We now draw a graph $G$ on the vertex set $\{1,...,m\}$, by putting an edge between $i$ and $j$ (for $i\neq j$) if and only if $Y_i\cap Y_j$ is infinite. Then we immediately have:

    \begin{claim}
        $G$ is connected.
    \end{claim}
    \begin{proof}
        If not, we can partition $\{1,...,m\}$ into two non-empty subsets $I$ and $J$ so that there are no edges between $I$ and $J$. Let $Y_I$ be the union of $Y_i$ for $i\in I$, and let $Y_J$ be the union of $Y_j$ for $j\in J$. Then for all $(i,j)\in I\times J$ the intersection $Y_i\cap Y_j$ is finite; and adding over each of the finitely many pairs $(i,j)$, it follows that $Y_I\cap Y_J$ is finite. This shows that $X$ is not strongly connected, which gives a contradiction.
    \end{proof}

    Now the main point of the proof is:
    \begin{claim}
        Suppose $(i,j)$ is an edge in $G$. Then the fields $F_i$ and $F_j$ are internal to each other in $\mathcal X^{hyp}$.\footnote{With a bit more work, one can even show that $F_i$ and $F_j$ are \textit{$\mathcal X^{hyp}$-definably isomorphic} (precisely, this follows from internality and \cite[Theorem 4.15]{PoiGroups}). We do not need this, however.}
    \end{claim}
    \begin{proof} Since $(i,j)$ is an edge, the intersection $Y_i\cap Y_j$ is infinite. So there is an irreducible curve $C\subset Y_i\cap Y_j$. Then $C$ is $\mathcal X^{hyp}$-definable (as a closed subset of the full relic $Y_i$); and then since it is contained in both $Y_i$ and $Y_j$, we get that $C$ is internal (as a definable set in $\mathcal X^{hyp}$) to both $F_i$ and $F_j$. Note that $C$ is strongly minimal in $\mathcal X^{hyp}$ (since it is an irreducible curve). Similarly, $F_i$ and $F_j$ are strongly minimal (this is a well-known consequence of the fact that they are $K$-definably isomorphic to $K$, see \cite[Theorem 4.15]{PoiGroups}).

    Now since $C$ is internal to $F_i$, and both are strongly minimal, it follows that $F_i$ is \textit{almost internal} to $C$ (this is because almost internality is equivalent to non-orthogonality for strongly minimal sets, and non-orthogonality is symmetric). In particular, since $C$ is \textit{also} internal to $F_j$, we conclude that $F_i$ is almost internal to $F_j$. In particular, $F_i$ and $F_j$ are non-orthogonal.

    Finally, it follows by results in \cite{CasHasVA} that for definable fields in $K$-relics, the non-orthogonality relation is the same as the internality relation. Precisely, since $F_i,F_j$ are fields, Fact \ref{F: trichotomy} gives that they are both full as $K$-relics (with their induced structures from $\mathcal X^{hyp}$). In particular, they are \textit{very ample} as strongly minimal sets (\cite[Theorem 4.14]{CasHasVA}), so that internality in each direction follows from non-orthogonality (\cite[Proposition 4.1]{CasHasVA}). 
    \end{proof}
    We immediately conclude:
    \begin{claim}
        For all $i$ and $j$, the fields $F_i$ and $F_j$ are internal to each other in $\mathcal X^{hyp}$.
    \end{claim}
    \begin{proof} By following a path through $G$ from $i$ to $j$ (using that $G$ is connected).
    \end{proof}

    Finally, it now follows that $X$ is internal to $F_1$ in $\mathcal X^{hyp}$. Indeed, we can fix a single countable set $A$ witnessing the internality of each $Y_i$ to $F_i$, and witnessing the internality of each $F_i$ to $F_1$. Then every element of $X$ is definable over $A\cup F_1$. Indeed, let $x\in X$. So $x\in Y_i$ for some $i$, and thus $x$ is definable over $A\cup F_i$. But each element of $F_i$ is in turn definable over $A\cup F_1$, so in fact $x$ is definable over $A\cup F_1$, as desired.

    So we have shown that $X$ is internal to $F_1$, and this gives fullness by Fact \ref{F: trichotomy}.
\end{proof}

\subsection{Proof of the theorem}

Finally, we are ready to conclude our final main theorem.

\begin{convention}\label{Con: reducible strongly connected}
    \textbf{We now fix:}
    \begin{enumerate}
        \item Uncountable algebraically closed fields $K_1,K_2$.
        \item Quasi-projective varieties $X_i$ over $K_i$, all of whose irreducible components have dimension at least 2.
        \item A homeomorphism $\varphi:X_1\rightarrow X_2$.
    \end{enumerate}
\end{convention}

Recall that Theorem \ref{T: main reducible} says that, just as in Theorem \ref{T: main}, we can find a field isomorphism $\sigma:K_1\rightarrow K_2$ and a diagram

\[
\xymatrix{
\widetilde{\sigma(X_1)} \ar[r]^{\tilde{f}} \ar[d] & \widetilde{X_2} \ar[d] \\
\sigma(X_1) \ar[r]^{f} & X_2
}
\]
where $\varphi=f\circ\sigma$ on $X_1$ and $\tilde{f}:\widetilde{\sigma(X_1)}\rightarrow\widetilde{X_2}$ is a universal homeomorphism of $K_2$-varieties. \textbf{Theorem \ref{T: main reducible} now follows essentially immediately}, by arguing exactly as in Theorem \ref{T: main} and using the fullness of $\mathcal X^{hyp}$ (as provided by Theorem \ref{T: reducible full}). Let us be a bit more precise:
\begin{proof}[Proof of Theorem \ref{T: main reducible}]
    Endow $X_1$ with the full $K_1$-relic structure $\mathcal X_1^{hyp}$ as in Theorem \ref{T: reducible full}, and let $\mathcal X_2^{hyp}$ be the image of $\mathcal X_1^{hyp}$ under $\varphi$. As in Theorem \ref{T: main}, it follows that $\mathcal X_2^{hyp}$ is a $K_2$-relic (by arguing separately in each component). So, continuing to argue exactly as in Theorem \ref{T: main}, we may reduce to the case that $K_1=K_2=K$ and $\varphi$ is $K$-definable.

    Now on each component, $\varphi$ is given by a sufficiently negative Frobenius power composed with a generically rational map. Since there are only finitely many components, one negative Frobenius power will do for all. So, factoring in a negative Frobenius power into the field isomorphism $K_1\rightarrow K_2$, we further reduce to the case that $\varphi$ is generically rational on each component. Then, continuing to argue as in Theorem \ref{T: main}, we separately obtain a commuting square as above on each component, where the top map in each case is a universal homeomorphism of normalized $K$-varieties. Finally, since the normalization of each $X_i$ is just the disjoint union of the normalizations of the components, we can freely glue these separate universal homeomorphisms into a global universal homeomorphism between all of $X_1$ and $X_2$. 
    \end{proof}

    \begin{remark} The proof of Theorem \ref{T: main reducible} is really the same as the proof of Theorem \ref{T: reducible schemes}; the only difference is that because the relics $\mathcal X_i^{hyp}$ are full, we can simultaneously reduce to the case that $\varphi$ is generically rational on \textit{all} components. In this case, we simply repeat the proof of Theorem \ref{T: reducible schemes} in the category of $K_2$-varieties instead of the category of schemes.
    \end{remark}

    \section{Proofs of the KLOS Speculations}

    We finish the paper with a very short section explicitly deriving the relevant questions asked in \cite{KLOS}:

    \subsection{Quasi-projective case}

    The following is \cite[Speculation 2.2.14]{KLOS} specialized to uncountable algebraically closed fields:

    \begin{proposition}\label{P: qp spec}
        Let $K_1,K_2$ be uncountable algebraically closed fields of characteristic zero, and let $X_i$ be an irreducible, normal, quasi-projective variety over $K_i$ of dimension at least 2. Let $\varphi:X_1\rightarrow X_2$ be a homeomorphism. Then $\varphi$ factors into a field isomorphism $K_1\rightarrow K_2$ followed by an isomorphism of $K_2$-varieties. 
    \end{proposition}
    \begin{proof}
        Theorem \ref{T: main} applies, except that the normalizations of the $X_i$ are isomorphisms (as the $X_i$ are already normal), so there is no need to lift to the normalizations. Thus, we write $\varphi$ as the composition of a field isomorphism $K_1\rightarrow K_2$ and a universal homeomorphism of $K_2$-varieties $f:\sigma(X_1)\rightarrow X_2$. But then in characteristic zero, $f$ is automatically separable, and thus is a birational morphism. By normality, it then follows by Zariski's Main Theorem (since $f$ is clearly quasi-finite) that $f$ is an open immersion. But $f$ is also surjective (as it is a homeomorphism), and a surjective open immersion is an isomorphism.  
    \end{proof}

    \subsection{Positive characteristic}

    The following generalizes \cite[Speculation 2.2.12]{KLOS} in the case of uncountable algebraically closed fields:

    \begin{proposition}
        Let $K_1,K_2$ be uncountable algebraically closed fields of different characteristics, and let $X_i$ be an irreducible variety of dimension at least 2 over $K_i$. Then $X_1$ and $X_2$ are not homeomorphic.
    \end{proposition}
    \begin{proof}
        Theorem \ref{T: main} in particular gives a field isomorphism $K_1\rightarrow K_2$, which means the $K_i$ have the same characteristic. 
    \end{proof}

    Then the following generalizes \cite[Speculation 2.2.16]{KLOS} in the case of uncountable algebraically closed fields:

    \begin{proposition}
        Let $K_1,K_2$ be uncountable algebraically closed fields, and let $X_i$ be an irreducible, normal, quasi-projective variety over $K_i$ of dimension at least 2. Let $\varphi:X_1\rightarrow X_2$ be a homeomorphism. Then $\varphi$ factors into a field isomorphism $K_1\rightarrow K_2$ followed by a purely inseparable morphism of $K_2$-varieties. 
    \end{proposition}
    \begin{proof}
        As in the Proposition \ref{P: qp spec}, we may apply Theorem \ref{T: main} without passing to the normalizations. So we write $\varphi=f\circ\sigma$ where $\sigma$ is induced by a field isomorphism and $f$ is a universal homeomorphism of $K_2$-varieties. Finally, any such morphism is purely inseparable.
    \end{proof}

    \subsection{Reducible varieties}

    The following generalizes \cite[Speculation 2.2.10]{KLOS} in the case of uncountable algebraically closed fields:

    \begin{proposition}
                Let $K_1,K_2$ be uncountable algebraically closed fields, and let $X_i$ be a quasi-projective variety over $K_i$, all of whose irreducible components have dimension at least 2. Let $\varphi:X_1\rightarrow X_2$ be a homeomorphism. Then $\varphi$ lifts uniquely to a homeomorphism of normalizations $\widetilde\varphi:\widetilde{X_1}\rightarrow\widetilde{X_2}$. 
    \end{proposition}

    \begin{proof}
        We showed existence of the lift in Theorem \ref{T: reducible schemes}. For uniqueness, let $\widetilde\varphi_1,\widetilde\varphi_2$ be two such lifts. Applying Theorem \ref{T: reducible schemes} to the $\widetilde\varphi_i$ shows that each $\widetilde\varphi_i$ is a universal homeomorphism of schemes (i.e. $\widetilde\varphi_i$ lifts to a universal homeomorphism of normalizations, but since the $\widetilde{X_i}$ are already normal, the $\widetilde\varphi_i$ are thus already universal homeomorphisms). Finally, the $\widetilde\varphi_i$ agree on a dense open subset of $X_1$ (i.e. the locus where the normalizations are isomorphisms), and since they are both morphisms, it follows that the $\widetilde\varphi_i$ are equal.
    \end{proof}

\noindent\textit{Acknowledgments:} We thank Martin Olsson and Tom Scanlon for making us aware of the topic of this paper. We also thank Martin Olsson, Tom Scanlon, Assaf Hasson, and Lou van den Dries for helpful discussions on the material.

\bibliography{References}
\bibliographystyle{plain}

\end{document}